\documentclass[12pt]{article}
\usepackage{graphics,amsmath,amssymb,amsthm,mathrsfs,indentfirst,bm,color}

\newcommand{\lap}[1]{\Delta #1}

\newtheorem{thm}{Theorem}[section]
\newtheorem{lemma}[thm]{Lemma}
\newtheorem{cor}[thm]{Corollary}
\newtheorem{remark}[thm]{Remark}

\numberwithin{equation}{section}
\allowdisplaybreaks
\begin{document}

\title{Semiclassical states of a linearly coupled critical  fractional Schr\"{o}dinger system }

\author{Shijie Qi$^{1,2}$\footnote{
 Corresponding author. E-mail: qishj15@lzu.edu.cn (S. Qi); zhaoph@lzu.edu.cn (P. Zhao).
 } \quad Peihao Zhao$^1$  \\ \small{$^1$School of Mathematics and Statistics, Lanzhou University, Lanzhou 730000, PR China}\\
 \small{$^2$Department of Mathematical Sciences,
Yeshiva University,
New York NY 10033, USA}}

\date{ }
\maketitle

\begin{abstract}
This paper focuses on the  linearly coupled   critical fractional Schr\"{o}dinger system
\begin{equation*}
\begin{cases}
\epsilon^{2s}(-\triangle)^s u +a(x)u=u^p+\lambda v\quad &\text{in}\ \mathbb{R}^N,\\
\epsilon^{2s}(-\triangle)^s v +b(x)v=v^{2_s^*-1}+\lambda u\quad &\text{in}\ \mathbb{R}^N,
\end{cases}
\end{equation*}
where $N>2s,$ $s\in(0,1),$ $p\in(1,2_s^*),$  $\epsilon$ and $\lambda$ are positive parameters, $a,b\in C{(\mathbb{R}^N)}$ are positive potentials, and  $(-\triangle)^s$ is the fractional Laplacian operator. Under certain assumptions on $a$ and $\lambda,$ we obtain the existence, decay estimates and concentration property of  positive vector ground states  for small $\epsilon.$
Furthermore, under an additional assumption on potentials $a$ and $b$, we consider the multiplicity  of positive vector solutions for small $\epsilon$, which turn out to have similar decay estimate and concentration property to those of the ground state for small $\epsilon$.
\vspace{0.3cm}
\end{abstract}

\section{Introduction and main results}

In this paper, we consider the existence, multiplicity and concentration  property  of positive vector solutions for  the following linearly coupled fractional Schr\"{o}dinger system
\begin{equation}\label{op}
\begin{cases}
\epsilon^{2s}(-\triangle)^s u +a(x)u=u^p+\lambda v\quad &\text{in}\ \mathbb{R}^N,\\
\epsilon^{2s}(-\triangle)^s v +b(x)v=v^{2_s^*-1}+\lambda u\quad &\text{in}\ \mathbb{R}^N,
\end{cases}
\end{equation}
where $N>2s,$ $s\in(0,1),$ $p\in(1,2_s^*),$ $2_s^*=2N/(N-2s),$  $\epsilon,\ \lambda>0$ are parameters, $(-\triangle)^s$ is the fractional Laplacian operator,   and $a,b\in C{(\mathbb{R}^N)}$ are positive potentials.  System \eqref{op} arises in the study of   time-dependent nonlinear Schr\"{o}dinger system
\begin{equation}\label{49}
\begin{cases}
i\epsilon \frac{\partial\psi}{\partial t}=\epsilon^{2s}(-\triangle)^s \psi +P(x)\psi-|\psi|^{p-1}\psi-\lambda \phi\quad
 &\text{in}\ \mathbb R\times\mathbb{R}^N,\\
 i\epsilon \frac{\partial\phi}{\partial t}=\epsilon^{2s}(-\triangle)^s \phi +Q(x)\phi-|\phi|^{2_s^*-2}\phi-\lambda \psi\quad
 &\text{in}\ \mathbb R\times\mathbb{R}^N.
\end{cases}
\end{equation}
In fact,  if $(\psi,\phi)$ is a standing wave  of system \eqref{49}, that is, a solution with the form $(\psi(t,x),\phi(t,x))$
$=(e^{-iEt/\epsilon}u(x), e^{-iEt/\epsilon}u(x))$ satisfying  $u(x)\to 0$ and $v(x)\to 0$  as
  $|x|\to \infty$.
  Then $(u,v)$ is a solution of system
  \begin{equation*}
\begin{cases}
\epsilon^{2s}(-\triangle)^s u +(a(x)-P)u=|u|^{p-1}u+\lambda v\quad &\text{in}\ \mathbb{R}^N,\\
\epsilon^{2s}(-\triangle)^s v +(b(x)-Q)v=|v|^{2_s^*-2}v+\lambda u\quad &\text{in}\ \mathbb{R}^N.
\end{cases}
\end{equation*}
  For sufficiently small $\epsilon>0$, the standing wave
is referred to as the semiclassical state. A solution  $(u,v)$ of \eqref{op} is called a positive vector solution if
$u>0$ and $v>0$ in $\mathbb R^N.$

 The fractional Laplacian  operator
$(-\triangle)^s$   arises in many fields such as phase transitions, flame propagation, stratified materials and others, see \cite{ab,bgr,sirevaldinoci} and references therein. In particular, it can be understood as the infinitesimal generator of a stable Levy process (see \cite{valdinoci}).
There are various equivalent definitions of the fractional Laplacian operator \cite{Kwasnicki}.
In particular, if $u$ belongs to the Schwartz class, then $(-\triangle)^su$ can be defined as
$$\mathcal{F}\left((-\triangle)^su\right)(\xi)=|\xi|^{2s}\mathcal{F}(u)(\xi),\quad\xi\in\mathbb R^N,$$
where $\mathcal F$ denotes the Fourier transform. The fractional Laplacian can also be defined by the singular integral
\begin{equation}\label{50}
\begin{aligned}
(-\triangle)^{s}u (x) := C_{N,s} PV \int_{\mathbb{R}^N}\frac{u(x)-u(y)}{|x-y|^{N+2s}}dy
= C_{N,s}\lim_{\epsilon\to 0}\int_{\mathbb{R}^N\setminus B_{\epsilon}(x)}\frac{u(x)-u(y)}{|x-y|^{N+2s}}dy
\end{aligned}
\end{equation}
for any real number $0<s<1$, where PV stands for the Cauchy principal value and  $C_{N,s}$ is the normalized  constant.
If $u\in C_{\text{loc}}^{1.1}(\mathbb R^N)\cap L_{s},$ then  the integral in \eqref{50} converges
(see, e.g., \cite{CLM,CQ}), and hence $(-\lap)^{s}$ is well-defined for such $u$. Here
$$L_{s}:=\left\{u \in L_{loc}^1(\mathbb{R}^N) : \int_{\mathbb{R}^N}\frac{|u(x)|}{1+|x|^{N+2s}} \, d x <\infty\right\}.$$
Alternatively, it can be expressed without using the Cauchy principal value as
$$(-\lap)^{s} u(x):= \frac{C_{N, s}}{2}
 \int_{\mathbb{R}^N}
 \frac{2u(x)-u(x+y)-u(x-y)}{|y|^{N+2s}} d y.$$

The nonlinear Schr\"{o}dinger equations and systems have attracted a great deal of attentions. In particular,
there has been a great interest in the study of standing waves. For the local cases (s=1), there
are many significant references, we refer to \cite{Berestycki-Lions, Byeon-Jeanjean,Byeon-Wang2002,Byeon-Wang2003,Coffman,DelPino-Felmer1996,DelPino-Felmer1997,Montefusco-Pellacci-Squassina,Pomponio,
Strauss} and references therein.

In recent years,  an ever-growing  interest has been devoted to consider Schr\"{o}dinger equations and systems involving in the nonlocal operator. However, the study of these problems becomes much more complicated since the nonlocal character of operators causes some essential difficulties. For the fractional Schr\"{o}dinger equation
\begin{equation}\label{51}
\epsilon^2(-\triangle )^s u+V(x)u=f(u)\quad \text{in\ } \mathbb R^N,
\end{equation}
Fall,  Mahmoudi and  Valdinoci \cite{Fall-Mahmoudi-Valdinoci} showed that the concentration points must be critical points of  $V$ under some suitable assumptions. Moreover, if the potential $V$ is coercive and has a unique global minimum, then ground states have the concentration property  as $\epsilon$ tends to zero. In addition, if the potential $V$ is radial, then the minimizer is unique for small $\epsilon$.
D\'{a}vila, del Pino and  Wei \cite{Davila-DelPino-Wei}   proved the existence of positive
solutions which have multiple spikes near given topologically nontrivial critical points of $V$ or
cluster near a given local maximum point of $V$ by  the Lyapunov-Schmidt reduction method.
Alves and Miyagaki \cite{Alves-Miyagaki} studied  the existence and concentration property  of positive  solutions  via penalization method developed in \cite{DelPino-Felmer1996} for $f$ with subcritical growth. He and Zou \cite{he-zou} considered the existence, multiplicity and  concentration property of  positive  solutions of \eqref{51} with critical nonlinearities. For more results concerning the existence and concentration property for fractional Schr\"{o}dinger equations, we refer to \cite{Barrios- Colorado-Servadei-Soria,Barrios-Colorado-Pabl-Sanchezo,Barrios-Colorado-Pablo,Jin-Li-Xiong,Shang-Zhang} and references therein.

However, only few results are known in the literature on the study of concentration property of standing waves even for the subcritical case
when the fractional Sch\"{o}dinger systems are incorporated into consideration.
Guo and He \cite{guo-he} considered the  existence and concentration property of ground states for the following weakly coupled fractional Schr\"{o}dinger system  with subcritical
nonlinearities
\begin{equation}\label{55}
\begin{cases}
\epsilon^{2s}(-\triangle)^{2s} u +P_1(x)u=(|u|^{2p}+b|u|^{p-1}|v|^{p+1})u\quad&\text{in}~\mathbb R^N,\\
\epsilon^{2s}(-\triangle)^{2s} v +P_2(x)v=(|v|^{2p}+b|v|^{p-1}|u|^{p+1})v\quad&\text{in}~\mathbb R^N,
\end{cases}
\end{equation}
where $2p+2<2_s^*,$ potentials $P_i\ (i=1,2)$ are continuous and have the positive global infimum, and there
exists a smooth bounded open set $\Lambda\subset\mathbb R^N$ such that
\begin{equation}\label{54}
\inf_{\Lambda}P_1(x)<\min_{\partial \Lambda}P_1(x)~~\text{and }\inf_{\Lambda}P_2(x)<\min_{\partial \Lambda}P_2(x).
\end{equation}
Yu,  Zhao and  Zhao \cite{Yu-Zhao-Zhao} investigated  the subcritical fractional Schr\"{o}dinger-Poisson system
 \begin{equation*}
 \begin{cases}
 \epsilon^{2s}(-\triangle)^su +V(x)u+uv=K(x)|u|^{p-2}u\quad&\text{in}~\mathbb R^3,\\
 \epsilon^{2s}(-\triangle)^sv=u^2&\text{in}~\mathbb R^3,
 \end{cases}
 \end{equation*}
where $3/4<s<1,$ $4<p<2_s^*,$ potential $V\in C(\mathbb R^3)\cap L^\infty (\mathbb R^3)$ has a positive global minimum, and $K\in C(\mathbb R^3)\cap L^\infty (\mathbb R^3)$ is positive
and has global maximum. The authors proved the existence of  positive vector  ground state by using variational
methods for each $\epsilon > 0$ sufficiently small, and  determined a concrete set related to the
potentials $V$ and $K $ as the concentration position of these ground state solutions as $\epsilon\to 0.$
For more results on the existence of solutions for fractional Schr\"{o}dinger systems, we can refer to \cite{Bhattarai,Teng,Zhang-do-Squassina} and references therein.

Inspired by the works mentioned above,  the current paper is devoted to the study of semiclassical states of the nonlocal critical Schr\"{o}dinger system \eqref{op} by using the variational methods, Ljusternik-Schnirelmann theory and penalization approach.
Clearly, if $s=1,$ then system \eqref{op} becomes the local  Schr\"{o}dinger system
\begin{equation}\label{52}
\begin{cases}
-\epsilon^{2}\triangle u +a(x)u=u^p+\lambda v\quad &\text{in}\ \mathbb{R}^N,\\
-\epsilon^{2}\triangle v +b(x)v=v^{2_s^*-1}+\lambda u\quad &\text{in}\ \mathbb{R}^N,
\end{cases}
\end{equation}
which has been investigated in \cite{Ambrosetti-Colorado- Ruiz2007,Ambrosetti-Cerami-Ruiz2008,chen-zou2012,chen-zou2014}.
In particular, the authors in \cite{chen-zou2012} considered the existence and nonexistence  of positive vector solutions of \eqref{52} in autonomous case with $\epsilon=1$ by the Nehari
manifold approach and blow up analysis. Concretely, the authors showed  that the radial positive vector ground state $(u_n,v_n)$ of the subcritical system obtained by replacing $2_s^*-1$ by $2_s^*-1-\tau_n$ with $\tau_n\to 0$ as $n\to \infty$
approximates to a  radial positive vector ground state of \eqref{52} if the ground state energy is less than
$\frac{1}{N}\widetilde S^{\frac{N}{2}},$ where $\widetilde S$ denotes the sharp embedding constant from $H^1(\mathbb R^N) $ to $L^{2^*}(\mathbb R^N).$ Based on the results established in \cite{chen-zou2012}, the authors in \cite{chen-zou2014} further considered the concentration property of the ground state for \eqref{52} using the penalization method. Naturally, we except to investigate the existence, multiplicity and concentration property of positive vector solutions for the nonlocal case \eqref{op}.
Before stating our main results, we first give some assumptions on the potential $a$ and $b$ as follows.
\begin{itemize}
 \item[(P1)] There exist positive constants $a_0$ and $b_0$ such that
 $$\inf_{x\in\mathbb{R}^N}b(x)=b_0,\quad \inf_{x\in\mathbb{R}^N}a(x)=a_0\leq \mu_0:=\left(\frac{2s(p+1)}{N(p-1)}S^{\frac{N}{2s}}C_{p+1}^{-\frac{p+1}{p-1}}\right)
 ^{\left(\frac{p+1}{p-1}-\frac{N}{2s}\right)^{-1}},$$
 where $S$ and $C_{p+1}$ are the sharp embedding constants from $\chi_s$ to $L^{2_s^*}(\mathbb R^N)$ and
 $\mathcal E$ to $L^{p+1}(\mathbb R^N)$ defined in \eqref{67} and \eqref{sharpconstant} below, respectively .

 \item[(P2)]There is a smooth bounded open domain $\Lambda\subset\mathbb{R}^N$ such that
 $$\inf_{x\in\Lambda}a(x)\leq \mu_0<a_1:=\inf_{x\in\partial\Lambda}a(x).$$
\end{itemize}
Noting that we only assume some local conditions (P1)-(P2) on $a$ and $b$ rather than assumptions in \cite{Yu-Zhao-Zhao}, where
the authors posed the global boundedness to the potential.
Moreover, compared with assumption \eqref{54} concerning both potentials $P_1$ and $P_2$ for system  \eqref{55},
in the present paper, we only need assumption (P2) on the potential $a$, which is not involved in the potential $b$.

As usual, if the ground state of \eqref{op} exists and has concentration property, we except that it converges to a ground state of a autonomous system. Consequently, for any fixed $x_0\in\mathbb R^N,$ we first consider the system
\begin{equation}\label{autonomoussys}
\begin{cases}
(-\triangle)^su +a(x_0) u=u^p+\lambda v\quad &\text{in} \mathbb~ \mathbb R^N,\\
(-\triangle)^sv +b(x_0) v=v^{2_s^*-1}+\lambda u &\text{in} \mathbb~ \mathbb R^N.
\end{cases}
\end{equation}
\begin{thm}
Let $N>2s,$ $s\in(0,1)$ and $p\in(1,2_s^*).$   Assume $\lambda<\sqrt{a(x_0)b_0(x_0)}$ and $a(x_0)\leq \mu_0,$ where
$\mu_0$ is defined in $(P1).$ Then system \eqref{autonomoussys} admits a positive vector ground state.
\end{thm}
\begin{remark}\rm
If $p=2_s^*-1$ and $\lambda<\sqrt{a_0b_0},$ then  there is no positive vector solution to \eqref{autonomoussys} for any $x_0\in\mathbb R^N$
by the Pohozaev identity. Therefore, we always assume $p<2_s^*-1$ in this work.
\end{remark}

In the present paper,  we investigate the existence of positive vector ground state for \eqref{autonomoussys} using the extension methods, mountain pass theorem and Nehari manifold approach, which is very different to that of \cite{chen-zou2012} for local case.
Indeed, in \cite{chen-zou2012}, the authors showed  that the radial positive vector ground state $(u_n,v_n)$ of the subcritical system obtained by replacing $2_s^*-1$ by $2_s^*-1-\tau_n$ with $\tau_n\to 0$ as $n\to \infty$
is bounded uniformly  in $H^1(\mathbb R^N)\times H^1(\mathbb R^N).$  As a result, there exists a subsequence which converges weakly to some $(u,v)\in H^1(\mathbb R^N)\times H^1(\mathbb R^N).$ Then the authors established the uniform $L^\infty(\mathbb R^N)$ estimates of $(u_n,v_n)$ by blow up analysis when the ground state energy is less than $\frac{1}{N}\widetilde S^{\frac{N}{2}}$, which along with the radial character  of $(u_n,v_n)$ concludes that $(u_n,v_n)$
converges strongly to $(u,v)\in H^1(\mathbb R^N)\times H^1(\mathbb R^N).$  Consequently, $(u,v)$ is a positive vector solution of   autonomous case of \eqref{52} with $\epsilon=1.$

For the scalar  equation \eqref{51}, as mentioned above, Fall, Mahmoudi and  Valdinoci \cite{Fall-Mahmoudi-Valdinoci} showed that the concentration positions must be critical points of  $V$ under suitable assumptions. While for systems,  the positions of concentration points become more complicated.
For any $x\in \mathbb R^N,$ let $c_x$ be the  ground state energy of \eqref{autonomoussys} with $x_0$ replaced by $x.$
Furthermore, under assumption (P2), we define
\begin{equation}\label{o}
\mathcal O=\left\{x\in \Lambda: c_x=c_0:=\inf_{y\in\Lambda } c_y\right\}.
\end{equation}

\begin{thm}\label{53}
Assume that $(P1)$ and $(P2)$ hold. Let $N>2s,$ $s\in(0,1)$ and $p\in(1,2_s^*)$. If  $\lambda<\min\{\sqrt{a_0b_0},$ $\sqrt{ (a_1-\mu_0)b_0}\}.$ Then $\mathcal O\neq\emptyset,$ $\mathcal O\subset\subset \Lambda$ and there exists $\epsilon_0>0$ such that for any $\epsilon\in
(0,\epsilon_0),$ system \eqref{op} has a positive vector ground state $(u_\epsilon, v_\epsilon)$ with the following properties:
\begin{itemize}
\item [\rm (I)] There exist $x_\epsilon\in \Lambda$ and $x_0\in\mathcal O $ such that
$$u_\epsilon(x_\epsilon)+v_\epsilon(x_\epsilon)=\max_{x\in\mathbb R^N}(u_\epsilon(x)+v_\epsilon(x)).$$
Moreover, $x_\epsilon\to x_0$ as $\epsilon\to 0.$
\item[\rm (II)] Define $(\widetilde{u}_\epsilon(x),\widetilde{v}_\epsilon(x))=(u_\epsilon(\epsilon x+x_\epsilon),v_\epsilon(\epsilon x+x_\epsilon))$ for any $\epsilon\in (0,\epsilon_0).$
    Then, up to a subsequence, $ (\widetilde{u}_\epsilon,\widetilde{v}_\epsilon)$
       converges as $\epsilon\to 0$ to a ground state of \eqref{autonomoussys} with $x_0$ given in (I).
\item[\rm (III)] There exists a positive constant $C$ independent of $\epsilon$ such that
$$u_\epsilon(x)+v_\epsilon(x)\leq\frac{C\epsilon^{N+2s}}{\epsilon^{N+2s}+|x-\epsilon x_\epsilon|^{N+2s}}\quad\text{for any } x\in\mathbb R^N.$$
\end{itemize}
\end{thm}
Noting that Theorem \ref{53} focuses on the existence, concentration property and decay estimate  of the positive vector ground state.
Naturally, we want to ask weather other positive vector solutions exist or not for \eqref{op},
and if exist, weather they have the same properties to those of the ground state obtained in Theorem \ref{53}.

\begin{thm}\label{mul}
Under all  assumptions of Theorem \ref{53}, if we suppose in addition that
\begin{itemize}
 \item[\rm (P3)]there exists $x_0\in \Lambda$ such that
$a(x_0)=a_0$ and $b(x_0)=b_0.$
\end{itemize}
Then for any $\delta>0$  such that
$\mathcal O_\delta:=\{x\in\mathbb R^N: dist(x,\mathcal O)\leq \delta\}\subset\subset\Lambda,$
there exists $\epsilon_\delta>0$ such that system \eqref{op} admits at least $cat_{\mathcal O_\delta}(\mathcal O)$
positive vector solutions for any $\epsilon\in (0,\epsilon_\delta)$. Moreover, the properties  (I)-(III) in Theorem \ref{53} also hold for these solutions.
\end{thm}

 We obtain the polynomial decay results for the positive vector solutions of system \eqref{op} instead of the exponential decay  in the local case. Moreover, we conclude not only the existence, concentration property and decay estimate of the positive vector ground state,
but also the multiplicity of positive vector solutions and similar properties to those of the ground state in Theorem \ref{mul}.

We would like to mention here that, there are some essential difficulties in dealing with our system \eqref{op}.
The first one of the main difficulties arises in the nonlocal character of the operator $(-\triangle)^s.$
One useful method to study the fractional Laplacian is the integral equations method, which turns a given fractional Laplacian equation into its equivalent integral equation, and then various properties of the original equation can be obtained by investigating the integral equation,
see \cite{cfy,CLM} and references therein.
Recently, Chen and Li et al. have developed a direct method to investigate the nonlocal problems, see \cite{chenliliyan1,CLM,CQ} and references therein.
However, these methods do not turn the nonlocal operator into a  local one,
which makes many traditional methods in studying the local differential operators no longer work.
To overcome this difficulty, we use the extension method introduced by Cafarelli and Silvestre \cite{cs},
which turns nonlocal problems involving the fractional Laplacian into local ones in higher dimensions,
and therefore some additional difficulties followed with, for example, the extension functional is not homogeneous
and the truncation argument becomes more delicate since we need to take care the trace of the involved functions
which defined in an upper half space $\mathbb R^{N+1}_+.$
The second main difficulty comes from the critical nonlinearities in using variational methods due to the lack of compactness.
To overcome this, we will use a version of the concentration compactness principle established in \cite{Dipierro-Medina-Valdinoci}.
Another difficulty arises in the linearly coupled terms, which makes our analysis more complicated in establishing various estimates.

The rest of the paper is organized as follows. In section 2, we translate our system into its  extension form, and then consider the variational character of  a modified extension system. Section 3 is devoted to the study of the ground state for the
autonomous system \eqref{autonomoussys}. In section 4, the  multiplicity of the positive vector solutions are concluded for the modified extension system for small $\epsilon$. In section 5, we prove that the
positive vector solutions obtained in section 4 for the modified extension system also solve the extension system of the original problem by making a rescaling,
and then we complete the proofs of our main theorems.\\

{\bf Notation}
\begin{itemize}
\item We use $S$ and $C_{p+1}$ to stand for the positive constants defined in \eqref{67} and \eqref{sharpconstant} below, respectively.
\item For any $x\in \mathbb R^N,$ $y\geq 0$ and $r>0,$ the symbol $B_r(x)$ denotes the ball  centered at  $x$ with radius $r$ in $\mathbb R^N,$
and $\mathcal B_r^+((x,y))$ denotes the ball centered at  $(x,y)$ with radius $r$ in $\mathbb R^{N+1}_+=\mathbb R^N\times \mathbb(0,+\infty).$
Particularly, we denote respectively $B_r(0)$ and  $\mathcal B^+_r((0,0))$ by $B_r$ and $\mathcal B_r^+$
\item For any  $U$ belongs to  $\chi_s$ defined in section 2, we denote its trace by $u.$
\item We denote the norm in $L^q(\mathbb R^N)$ by $||\cdot||_q$.
\end{itemize}

\section{The modified extension problem}

 Since the fractional Laplacian operator $(-\triangle)^s$ is  nonlocal, many powerful methods for local elliptic
 equation are not available any more. To overcome this difficulty, we use the extension method developed by Cafarelli and Silvestre  in \cite{cs}.  Concretely,  for a function $u:\mathbb{R}^N \to \mathbb{R}$, consider its extension $U:\mathbb{R}^N\times[0, \infty) \to \mathbb{R}$ that satisfies
\begin{equation*}
\left\{\begin{array}{ll}
{\rm div}(y^{1-2s} \nabla U)=0\quad &  \text{in~} \mathbb{R}^N\times[0, \infty),\\
U(x, 0) = u(x).
\end{array}
\right.
\label{Loc}
\end{equation*}
Then there hold
$$(-\lap)^{s}u (x) = - k_s \displaystyle\lim_{y \to 0^+}
y^{1-2s} \frac{\partial U(x,y)}{\partial y},  \quad x \in \mathbb{R}^N,$$
and
\begin{equation}\label{56}
U(x,y)=P_y^s*u(x)=C_{N,s}\int_{\mathbb R^N}\frac{y^{2s}}{(|x-\xi|^2+y^2)^{\frac{N+2s}{2}}}u(\xi)d\xi,\quad (x,y)\in \mathbb R^{N+1}_+,
\end{equation}
where $k_s$ and $C_{N,s}$ are normalized positive constants (see, e.g., \cite{Barrios-Colorado-Pablo,cs,Jin-Li-Xiong}).
The extension operator is well defined for any $u\in\dot{H}(\mathbb R^N)$, which is defined as the  completion of $C_c^\infty({\mathbb R^{N}})$ under the norm
$$||u||_{\dot{H}}^2:= C_{N,s}\int_{\mathbb R^{2N}}\frac{(u(x)-u(y))^2}{|x-y|^{N+2s}}dxdy.$$
 We define the weighted Sobolev
space $\chi_s$ as the completion of $C_c^\infty(\overline{\mathbb R^{N+1}_+})$ under the norm
$$||U||_{\chi_s}^2:=k_s\int_{\mathbb R^{N+1}_+}y^{1-2s}|\nabla U|^2dxdy.$$
For any $U\in \chi_s$, its trace on $\mathbb R^N\times \{y=0\}$ is well defined and denoted by $U(\cdot,0)$ or $u$ for simplicity in this paper.
Furthermore, we define a Hilbert space $\mathcal E$ as
\begin{equation}\label{65}
\mathcal E:=\left\{U\in \chi_s:\int_{\mathbb R^N}|u(x)|^2dx<\infty\right\}
\end{equation}
equipped with the norm
$$||U||_{\mathcal E}^2:=||U||_{\chi_s}^2+||u||_2^2.$$
Next, we state some  useful embedding results in \cite{Barrios-Colorado-Pablo,Jin-Li-Xiong,Dipierro-Medina-Valdinoci}.
\begin{lemma}\label{59} There hold
\begin{itemize}
\item [\rm (I)]The embedding $\mathcal E\hookrightarrow L^q(\mathbb R^N)$ for any $q\in [2,2_s^*]$  is continuous. Moreover,
the  embedding $\mathcal E\hookrightarrow L^q_{\text{loc}}(\mathbb R^N)$ for any $q\in [2,2_s^*)$ is compact.
\item [\rm (II)] There exists a sharp constant $S=S(s,N)$ such that
\begin{equation}\label{67}
S||u||_{2_s^*}^{2}\leq ||U||^2_{\chi_s}\quad\text{for any~} U\in \chi_s.
\end{equation}
Moreover, the equality holds on the family of functions $W_\epsilon$ which is the extension of
$$w_\epsilon=\frac{\epsilon^{\frac{N-2s}{2}}}{(\epsilon^2+|x|^2)^{{\frac{N-2s}{2}}}},~~\epsilon>0.$$
\item [\rm (III)] There exist positive constants $\widehat C$ and $\gamma=1+\frac{2}{N-2s}$ such that, for any $U\in\chi_s,$
$$\left(\int_{\mathbb R^{N+1}_+}y^{1-2s}| U|^{2\gamma}\right)^{1/2\gamma}\leq \widehat C\left(\int_{\mathbb R^{N+1}_+}y^{1-2s}|\nabla U|^2dxdy\right)^{1/2}.$$
\item [\rm (IV)] If $\{U_n\}$ is a bounded sequence in $\chi_s,$  then it is pre-compact in $L_{\text{loc}}^2(y^{1-2s},\mathbb R^{N+1}_+), $  where $L^2(y^{1-2s},\mathbb R^{N+1}_+)$ denotes the weighted Lebesgue  space equipped with the norm
    $$||U||^2_{L^2\left(y^{1-2s},\mathbb R^{N+1}_+\right)}:=\int_{\mathbb R^{N+1}_+}y^{1-2s}U^2dxdy.$$
\end{itemize}
\end{lemma}
For more information on the extension method and extension spaces, we refer to \cite{Barrios-Colorado-Pablo,cs,Jin-Li-Xiong,Dipierro-Medina-Valdinoci} and references therein.

 Now, let us turn to our system \eqref{op}. Note that if $(\tilde{u},\tilde{v})$ is a solution of \eqref{op}, then by a direct calculation, $(u(x),v(x)):=(\tilde{u}(\epsilon x),\tilde{v}(\epsilon x))$
 solves the system
 \begin{equation}\label{cp}
\begin{cases}
(-\triangle)^s u +a(\epsilon x)u=u^p+\lambda v\quad &\text{in}\ \mathbb{R}^N,\\
(-\triangle)^s v +b(\epsilon x)v=v^{2_s^*-1}+\lambda u\quad &\text{in}\ \mathbb{R}^N.
\end{cases}
\end{equation}
By the extension method, we can translate system \eqref{cp} into
 \begin{equation}\label{ep}
 \begin{cases}
 -{\rm div} (y^{1-2s}\nabla U)=0\quad &\text{in} \ \mathbb R^{N+1}_+,\\
 -{\rm div} (y^{1-2s}\nabla V)=0& \text{in}\ \mathbb R^{N+1}_+,\\
 -k_s\mathop{\lim}\limits_{y\to0 }y^{1-2s}\frac{\partial U(x,y)}{\partial y}+a(\epsilon x)u=u^p+\lambda v &\text{on}\ \mathbb R^N,\\
 -k_s\mathop{\lim}\limits_{y\to0 }y^{1-2s}\frac{\partial V(x,y)}{\partial y}+b(\epsilon x)v=v^{2_s^*-1}+\lambda u\quad &\text{on}\ \mathbb R^N.
 \end{cases}
 \end{equation}
 Now we define a function space
 $$X_{s,\epsilon}:=\{(U,V)\in \chi_s\times\chi_s:||(U,V)||_{X_{s,\epsilon}}<\infty\},$$
 where
 $$||(U,V)||_{X_{s,\epsilon}}=\left(||U||_{a,\epsilon}^2+||V||_{b,\epsilon}^2\right)^{\frac{1}{2}},$$
 $$||U||_{a,\epsilon}^2 :=k_s\int_{\mathbb R^{N+1}_+}y^{1-2s}|\nabla U|^2dxdy +\int_{\mathbb R^N}a(\epsilon x)u^2dx,$$
 $$||V||_{b,\epsilon}^2 :=k_s\int_{\mathbb R^{N+1}_+}y^{1-2s}|\nabla V|^2dxdy +\int_{\mathbb R^N}b(\epsilon x)v^2dx.$$
In what follows, we omit the normalized constant $k_s$ for the convenience.

We call $(U,V)\in X_{s,\epsilon}$ a weak solution of \eqref{ep} if and only if  for any  $(\Phi, \Psi)\in X_{s,\epsilon}$, there holds
\begin{equation}\label{57}
\begin{aligned}
&\int_{\mathbb R^{N+1}_+} y^{1-2s} \nabla U\cdot\nabla \Phi dxdy+
\int_{\mathbb R^{N+1}_+} y^{1-2s} \nabla V\cdot\nabla \Psi dxdy +
\int_{\mathbb R^N}a(\epsilon x)u\phi dx +\int_{\mathbb R^N}b(\epsilon x)v\psi dx\\
=&\int_{\mathbb R^N}u^{p}\phi dx+\int_{\mathbb R^N}v^{2_s^*-1}\psi +\lambda\int_{\mathbb R^N}u\psi dx +
\lambda\int_{\mathbb R^N}v\phi dx .
\end{aligned}
\end{equation}
 Observe that if $(U,V)\in X_{s,\epsilon}$ is a weak solution of \eqref{ep}, then its trace $(u,v)$ is a weak solution of \eqref{cp}, and  $(U,V)$ is a critical point of the functional  $J_\epsilon:X_{s,\epsilon}\to \mathbb R $ defined as
 \begin{equation}\label{Jepsilon}
 \begin{aligned}
 J_\epsilon(U,V)=&\frac{1}{2}\int_{\mathbb R^{N+1}_+} y^{1-2s} |\nabla U|^2 dxdy+\frac{1}{2}\int_{\mathbb R^{N+1}_+} y^{1-2s} |\nabla V|^2 dxdy+\frac{1}{2}\int_{\mathbb R^N}a(\epsilon x)u^2 dx \\
 &+\frac{1}{2}\int_{\mathbb R^N}b(\epsilon x)v^2 dx+\frac{1}{p+1}\int_{\mathbb R^N}u^{p+1}dx +\frac{1}{2_s^*}
 \int_{\mathbb R^N}v^{2_s^*}dx +\lambda\int_{\mathbb R^N}uvdx.
 \end{aligned}
 \end{equation}
Define
 \begin{equation}\label{fandg}
 f(s)=
 \begin{cases}
 |s|^{p-1}s^+\quad&\text{if} \ s\leq\alpha,\\
 \alpha^{p-1}s^+&\text{if} \ s>\alpha,
 \end{cases}
 \quad\quad
 g(s)=
 \begin{cases}
 |s|^{2_s^*-2}s^+\quad&\text{if} \ s\leq\alpha,\\
 \alpha^{2_s^*-2}s^+&\text{if} \ s>\alpha,
 \end{cases}
 \end{equation}
 where
 \begin{equation}\label{alpha}
 0<\alpha<\min\left\{1, \left(\frac{\delta_0a_0}{2}\right)^{\frac{1}{p-1}},\left(\frac{\delta_0b_0}{2}\right)^{\frac{1}{2_s^*-2}}\right\},
 \end{equation}
$\delta_0$ is defined in \eqref{58} below, and $s^+:=\max\{s,0\}.$
Furthermore, we denote
$$\Lambda_\epsilon=\{x\in \mathbb R^N: \epsilon x\in \Lambda\},$$
where $\Lambda$ is given in (P2). Set
 \begin{equation}\label{fepsilon}
 f_\epsilon(x,s)=\chi_{\Lambda_\epsilon}(x)|s|^{p-1}s^++(1-\chi_{\Lambda_\epsilon}(x))f(s),\quad  (x,s)\in \mathbb R^N\times \mathbb R,
 \end{equation}
 \begin{equation}\label{gepsilon}
 g_\epsilon(x,s)=\chi_{\Lambda_\epsilon}(x)|s|^{2_s^*-2}s^++(1-\chi_{\Lambda_\epsilon}(x))g(s),\quad  (x,s)\in \mathbb R^N\times \mathbb R,
 \end{equation}
and
 \begin{equation} \label{primitive}
 F_\epsilon(x,s)=\int_0^sf_\epsilon(x,s)ds,\quad G_\epsilon(x,s)=\int_0^sg_\epsilon(x,s)ds.
 \end{equation}
Note that if $x\in \Lambda_\epsilon$ then
\begin{equation}\label{arfin}
(p+1)F_\epsilon(x,s)=|s^+|^{p+2}=f_\epsilon(x,s)s,\quad \forall s\in \mathbb R,
\end{equation}
 and
\begin{equation}\label{argin}
2_s^*G_\epsilon(x,s)=|s^+|^{2_s^*}=g_\epsilon(x,s)s, \quad\forall s\in \mathbb R.
\end{equation}
If $x\notin \Lambda_\epsilon,$  then for any $s\leq0,$  there hold $f_\epsilon(x,s)=0=F_\epsilon(x,s)$  and
$g_\epsilon(x,s)=0=G_\epsilon(x,s).$ For any $0<s\leq\alpha,$ the equalities  \eqref{arfin} and \eqref{argin} hold. Moreover, for any  $s>\alpha$,
\begin{equation}\label{arfout}
\begin{aligned}
F_\epsilon(x,s)&=\int_0^\alpha t^pdt+\alpha^{p-1}\int_\alpha^stdt
=\frac{1}{p+1}\alpha^{p+1} +\frac{1}{2}\alpha^{p-1}s^2-\frac{1}{2}\alpha^{p+1}
\leq\frac{1}{2}f_\epsilon(x,s)s.
\end{aligned}
\end{equation}
Similarly, by a straightforward calculation we have
$G_\epsilon(x,s) \leq \frac{1}{2}g_\epsilon(x,s)s$ for any  $s>0,$
which along with \eqref{arfin}, \eqref{argin} and \eqref{arfout} implies that
\begin{equation}\label{60}
2F_\epsilon(x,s)\leq f_\epsilon(x,s)s,\quad 2G_\epsilon(x,s)\leq g_\epsilon(x,s)s, \quad \forall (x,s)
\in \mathbb{R}^N\times\mathbb R.
\end{equation}

Now, we consider the following modified  system
\begin{equation}\label{mp}
 \begin{cases}
 -{\rm div} (y^{1-2s}\nabla U)=0\quad &\text{in} \ \mathbb R^{N+1}_+,\\
 -{\rm div} (y^{1-2s}\nabla V)=0& \text{in}\ \mathbb R^{N+1}_+,\\
 -\mathop{\lim}\limits_{y\to0 }y^{1-2s}\frac{\partial U(x,y)}{\partial y}+a(\epsilon x)u=f_\epsilon(x,u)+\lambda v &\text{on}\ \mathbb R^N,\\
 -\mathop{\lim}\limits_{y\to0 }y^{1-2s}\frac{\partial V(x,y)}{\partial y}+b(\epsilon x)v=g_\epsilon(x,v)+\lambda u\quad &\text{on}\ \mathbb R^N.
 \end{cases}
 \end{equation}
Noting that if $(U,V)\in X_{s,\epsilon}$ is a weak solution of \eqref{mp} defined by a similar way to \eqref{57}, then $(U,V)$ is a critical point of the functional $I_\epsilon:X_{s,\epsilon}\to \mathbb R $ defined as
 \begin{equation}\label{Iepsilon}
 \begin{aligned}
 I_\epsilon(U,V)=&\frac{1}{2}\int_{\mathbb R^{N+1}_+} y^{1-2s} |\nabla U|^2 dxdy+\frac{1}{2}\int_{\mathbb R^{N+1}_+} y^{1-2s} |\nabla V|^2 dxdy+\frac{1}{2}\int_{\mathbb R^N}a(\epsilon x)u^2 dx \\
 &+\frac{1}{2}\int_{\mathbb R^N}b(\epsilon x)v^2 dx-\int_{\mathbb R^N}F_\epsilon(x,u)dx -
 \int_{\mathbb R^N}G_\epsilon(x,v)dx -\lambda\int_{\mathbb R^N}uvdx.
 \end{aligned}
 \end{equation}

Recall that $\lambda<\sqrt{a_0b_0},$  we can fix a positive constant $\delta_0=\delta_0(\lambda)$ such that
 \begin{equation}\label{58}
0<\delta_0<\left(1-\frac{\lambda}{\sqrt{a_0b_0}}\right).
\end{equation}
This along with assumption (P1) implies that for any $\epsilon>0$ and $x\in \mathbb R^N$, there hold
\begin{equation}\label{delta0}
\lambda uv<(1-\delta_0)\sqrt{a_0b_0}|u||v|\leq\frac{1}{2}(1-\delta_0)a(\epsilon x)u^2+\frac{1}{2}(1-\delta_0)b(\epsilon x)v^2
.
\end{equation}

 \begin{lemma}\label{mmp}
 Assume $(P1)$. If $\lambda<\sqrt{a_0b_0},$ then $I_\epsilon$ satisfies the mountain geometry, that is,
 \begin{itemize}
 \item[\rm (I)] there exist $r,\tau>0$ independent of $\epsilon$ such that
 $$I_\epsilon(U,V)\geq\tau>0\quad \text{for any}~ ||(U,V)||_{X_{s,\epsilon}}=r; $$
 \item[\rm (II)] there is $(U,V)\in X_{s,\epsilon}$ with $||(U,V)||_{X_{s,\epsilon}}>r$ such that
 $I_\epsilon(U,V)<0.$
 \end{itemize}
 \end{lemma}

\begin{proof}
By the definitions of $F_\epsilon$ and $G_\epsilon$, we have that
\begin{equation}\label{FG}
F_\epsilon(x,s)\leq \frac{1}{p+1}(s^+)^{p+1},~~G_\epsilon(x,s)\leq \frac{1}{2_s^*}(s^+)^{2_s^*}
\quad \text{for any~}(x,s)\in \mathbb R^N\times \mathbb R,
\end{equation}
we then by \eqref{Iepsilon} and \eqref{delta0} deduce that
\begin{equation}\label{mp1}
 \begin{aligned}
 I_\epsilon(U,V)\geq&\frac{1}{2}\int_{\mathbb R^{N+1}_+} y^{1-2s} |\nabla U|^2 dxdy+\frac{1}{2}\int_{\mathbb R^{N+1}_+} y^{1-2s} |\nabla V|^2 dxdy+\frac{\delta_0}{2}\int_{\mathbb R^N}a(\epsilon x)u^2 dx \\
 &+\frac{\delta_0}{2}\int_{\mathbb R^N}b(\epsilon x)v^2 dx-\frac{1}{p+1}\int_{\mathbb R^N}(u^+)^{p+1}dx -\frac{1}{2_s^*}
 \int_{\mathbb R^N}(v^+)^{2_s^*}dx\\
\geq &\frac{\delta_0}{2}||(U,V)||_{X_{s,\epsilon}}^2-\frac{1}{p+1} C_{p+1}^{-\frac{p+1}{2}}||(U,V)||_{X_{s,\epsilon}}^{p+1}-\frac{1}{2_s^*} S^{-\frac{2_s^*}{2}}||(U,V)||_{X_{s,\epsilon}}^{2_s^*},
 \end{aligned}
 \end{equation}
where the last inequality follows from Lemma \ref{59}. This implies that (I) holds for small $r>0.$
Choose a nonnegative function $\Phi\in C_c^\infty(\overline{\mathbb R^{N+1}_+})$  with $\emptyset\neq supp~\phi\subset\subset\Lambda_\epsilon,$ then
\begin{equation*}
\begin{aligned}
I_\epsilon(t\Phi,0)=&\frac{t^2}{2}\int_{\mathbb R^{N+1}_+} y^{1-2s} |\nabla \Phi|^2 dxdy+\frac{t^2}{2}\int_{\mathbb R^N}a(\epsilon x)\phi^2 dx
 -\frac{t^{p+1}}{p+1}\int_{\mathbb R^N}\phi^{p+1}dx\\
 &\to -\infty \quad\text{as}~ t \to \infty.
\end{aligned}
 \end{equation*}
Then (II) holds. The proof is complete.
\end{proof}

Now, we define $$\Gamma:=\{\gamma\in C([0,1],X_{s,\epsilon}):\gamma(0)=(0,0),\ I_\epsilon(\gamma(1))<0\},$$
then $\Gamma\neq\emptyset$ thanks to Lemma \ref{mmp}. Furthermore, set
\begin{equation}\label{mpenergy}
c_\epsilon:=\inf_{\gamma\in\Gamma}\max_{t\in[0,1]}I_\epsilon(\gamma(t)),
\end{equation}
then it follows from \eqref{mp1} that $c_\epsilon>\tau>0$ for any $\epsilon>0.$
We define the Nehari manifold \cite{szukin-weth} associated to $I_\epsilon$ by
$$N_\epsilon:=\{(U,V)\in X_{s,\epsilon}:I^\prime_\epsilon(U,V)(U,V)=0\}.$$

\begin{lemma}\label{nehari}
Assume $(P1)$. If $\lambda<\sqrt{a_0b_0},$  then
\begin{itemize}
\item[\rm (I)]there exists a positive constant $C_0>0$ independent of $\epsilon$ such that
$$||(U,V)||_{X_{s,\epsilon}}\geq C_0\quad \text{for any~} (U,V)\in N_\epsilon.$$
\item[\rm(II)] For any $(U,V)\in N_\epsilon,$  we have
$|supp\ u^+\cap\Lambda_\epsilon|+|supp\ v^+\cap\Lambda_\epsilon|>0.$
\item[\rm(III)] $I_\epsilon$ is bounded from below on $N_\epsilon$ by a positive constant.
\end{itemize}
\end{lemma}
\begin{proof}
For any $(U,V)\in N_\epsilon,$ in view of $I^\prime_\epsilon(U,V)(U,V)=0$, we have
 \begin{equation*}
 \begin{aligned}
&\int_{\mathbb R^{N+1}_+} y^{1-2s} |\nabla U|^2 dxdy+\int_{\mathbb R^{N+1}_+} y^{1-2s} |\nabla V|^2 dxdy+\int_{\mathbb R^N}a(\epsilon x)u^2 dx +\int_{\mathbb R^N}b(\epsilon x)v^2 dx\\
 =&\int_{\mathbb R^N}f_\epsilon(x,u)udx +
 \int_{\mathbb R^N}g_\epsilon(x,v)vdx +2\lambda\int_{\mathbb R^N}uvdx.
 \end{aligned}
 \end{equation*}
 It then follows from \eqref{delta0},\eqref{FG} and Lemma \ref{59} that
 $$\delta_0||(U,V)||_{X_{s,\epsilon}}^2 \leq C_{p+1}^{-\frac{p+1}{2}}||(U,V)||_{X_{s,\epsilon}}^{p+1}+ S^{-\frac{2_s^*}{2}}||(U,V)||_{X_{s,\epsilon}}^{2_s^*}.$$
 Consequently, there exists a positive constant $C_0>0$ independent of $\epsilon$ such that
$$||(U,V)||_{X_{s,\epsilon}}\geq C_0.$$

Assume to the contrary that (II) is not true, then there exists $(U,V)\in N_\epsilon$ such that
$$supp\ u^+\cap\Lambda_\epsilon=\emptyset \text{ ~and~ } supp\ v^+\cap\Lambda_\epsilon=\emptyset.$$
This along with $I^\prime_\epsilon(U,V)(U,V)=0,$ \eqref{fepsilon}, \eqref{gepsilon}, \eqref{delta0} and the choice of $\alpha$ shows that
\begin{equation*}
 \begin{aligned}
0=&\int_{\mathbb R^{N+1}_+} y^{1-2s} |\nabla U|^2 dxdy+\int_{\mathbb R^{N+1}_+} y^{1-2s} |\nabla V|^2 dxdy+\int_{\mathbb R^N}a(\epsilon x)u^2 dx +\int_{\mathbb R^N}b(\epsilon x)v^2 dx\\
 &-\int_{\mathbb R^N\setminus\Lambda_\epsilon}f_\epsilon(x,u)udx -
 \int_{\mathbb R^N\setminus\Lambda_\epsilon}g_\epsilon(x,v)vdx -2\lambda\int_{\mathbb R^N}uvdx\\
 \geq&\int_{\mathbb R^{N+1}_+} y^{1-2s} |\nabla U|^2 dxdy+\int_{\mathbb R^{N+1}_+} y^{1-2s} |\nabla V|^2 dxdy+\int_{\mathbb R^N}a(\epsilon x)u^2 dx +\int_{\mathbb R^N}b(\epsilon x)v^2 dx\\
 &-\alpha^{p-1}\int_{\mathbb R^N}u^2 dx-\alpha^{2^*_s-2}\int_{\mathbb R^N}v^2 dx-
 (1-\delta_0)\int_{\mathbb R^N}a(\epsilon x)u^2 dx +(1-\delta_0)\int_{\mathbb R^N}b(\epsilon x)v^2 dx\\
 \geq& \frac{\delta_0}{2}||(U,V)||_{X_{s,\epsilon}}^2.
 \end{aligned}
 \end{equation*}
Then we conclude that $||(U,V)||_{X_{s,\epsilon}}^2=0,$ which contradicts (I). Therefore (II) holds.

 For any $(U,V)\in N_\epsilon,$  it follows from $I_\epsilon^\prime(U,V)(U,V)=0$ and \eqref{60} that
 \begin{equation}\label{1}
 \begin{aligned}
 I_\epsilon(U,V)=&I_\epsilon(U,V)-\frac{1}{2}I_\epsilon^\prime(U,V)(U,V)\\
 =&\left(\frac{1}{2}-\frac{1}{p+1}\right)\int_{\Lambda_\epsilon}(u^+)^{p+1}dx+
 \left(\frac{1}{2}-\frac{1}{2_s^*}\right)\int_{\Lambda_\epsilon}(v^+)^{2_s^*}dx\\
 &+\int_{\mathbb R^N\setminus\Lambda_\epsilon}\frac{1}{2}f_\epsilon(x,u)u-F_\epsilon(x,u)dx+\int_{\mathbb R^N\setminus\Lambda_\epsilon}\frac{1}{2}g_\epsilon(x,v)v-G_\epsilon(x,v)dx\\
 \geq&\left(\frac{1}{2}-\frac{1}{p+1}\right)\int_{\Lambda_\epsilon}(u^+)^{p+1}dx+
 \left(\frac{1}{2}-\frac{1}{2_s^*}\right)\int_{\Lambda_\epsilon}(v^+)^{2_s^*}dx.
 \end{aligned}
 \end{equation}
We then from  (II) obtain  that $I_\epsilon(U,V)>0$ for any $(U,V)\in N_\epsilon.$ Now we suppose to the contrary
 that there exists a sequence $\{(U_n,V_n)\}\subset N_\epsilon$ such that
 $I_\epsilon(U_n,V_n)\to 0$ \text{as} $n\to\infty,$
 then it follows from \eqref{1} that
 $$\lim_{n\to\infty}\int_{\Lambda_\epsilon}(u_n^+)^{p+1}dx=0,\quad
 \lim_{n\to\infty}\int_{\Lambda_\epsilon}(v_n^+)^{2_s^*}dx=0.$$
 Combining with \eqref{Iepsilon} \eqref{fepsilon} and \eqref{gepsilon}, we have
 \begin{equation*}
 \begin{aligned}
 o_n(1)=&\frac{1}{2}\int_{\mathbb R^{N+1}_+} y^{1-2s} |\nabla U_n|^2 dxdy+\frac{1}{2}\int_{\mathbb R^{N+1}_+} y^{1-2s} |\nabla V_n|^2 dxdy+\frac{1}{2}\int_{\mathbb R^N}a(\epsilon x)u_n^2 dx \\
 &+\frac{1}{2}\int_{\mathbb R^N}b(\epsilon x)v_n^2 dx-\int_{\mathbb R^N\setminus\Lambda_\epsilon}F_\epsilon(x,u_n)dx -
 \int_{\mathbb R^N\setminus\Lambda_\epsilon}G_\epsilon(x,v_n)dx -\lambda\int_{\mathbb R^N}u_nv_ndx\\
 \geq&\frac{\delta_0}{4}||(U_n,V_n)||_{X_{s,\epsilon}}^2.
 \end{aligned}
 \end{equation*}
Hence $||(U_n,V_n)||_{X_{s,\epsilon}}\to 0$ as $n\to \infty,$ which contradicts (I). The proof is complete.
\end{proof}

 Since the trace $(u,v)$ of  any  solution $(U,V)$ of \eqref{mp} satisfies  (II) in Lemma \ref{nehari},  we define
$$X_{s,\epsilon}^+=\{(U,V)\in X_{s,\epsilon}:|supp\ u^+\cap\Lambda_\epsilon|+|supp\ v^+\cap\Lambda_\epsilon|>0\},$$
and $S_{s,\epsilon}^+$  as the intersect of $X_{s,\epsilon}^+$ with the unit sphere $S_{s,\epsilon},$ namely,
$S_{s,\epsilon}^+=S_{s,\epsilon}\cap X_{s,\epsilon}^+.$

\begin{lemma}\label{Xepsilon}
Assume $(P1)$. If $\lambda<\sqrt{a_0b_0},$ then
\begin{itemize}
\item[\rm(I)] for any $(U,V)\in X_{s,\epsilon}^+,$ if we define $h_{UV}(t)=I_\epsilon(tU,tV),$
then $h_{UV}\in C^1(0,\infty) $ and there exists a unique $t_{UV}>0$
such that
$$h_{UV}(t_{UV})=\max_{t>0}h_{UV}(t).$$
Moreover, $h_{UV}^\prime(t)>0$ for any $t\in (0,t_{UV}),$ and $h_{UV}^\prime(t)<0$ for any $t\in (t_{UV}, \infty).$
\item[\rm(II)] $(t_{UV}U,t_{UV}V)\in N_\epsilon.$ Moreover if $\mathcal{W}\subset X_{s,\epsilon}^+ $ is bounded and there is $\theta>0$ such that $||(U,V)||_{X_{s,\epsilon}}>\theta$ for any $(U,V)\in \mathcal{W}$ , then there exists a positive constant $c_\mathcal{W}$ independent of $\epsilon$ such that $t_{UV}\geq c_\mathcal{W}>0.$
 In addition, if $\mathcal{W} $ is compact, then there is a positive constant $C_\mathcal{W}$ such that
 $t_{UV}\leq C_\mathcal{W}.$
\item[\rm(III)] Set $\tau_\epsilon : (U,V)\mapsto (t_{UV}U,t_{UV}V),$ then $\tau_\epsilon$ is continuous from $X_{s,\epsilon}^+$ to $N_\epsilon$, and  $\overline{\tau_\epsilon}:=\tau_\epsilon|_{S_{s,\epsilon}^+}$is a homomorphism between $S_{s,\epsilon}^+$ and $N_\epsilon.$
\end{itemize}
\end{lemma}

\begin{proof}
For any $(U,V)\in X_{s,\epsilon}^+,$ by the definition of $h_{UV}(t)$ and straight computations, we have
\begin{equation}\label{2}
\begin{aligned}
h^\prime_{UV}(t)=&t\int_{\mathbb R^{N+1}_+} y^{1-2s} |\nabla U|^2 dxdy+t\int_{\mathbb R^{N+1}_+} y^{1-2s} |\nabla V|^2 dxdy+t\int_{\mathbb R^N}a(\epsilon x)u^2 dx\\
& +t\int_{\mathbb R^N}b(\epsilon x)v^2 dx-\int_{\mathbb R^N}f_\epsilon(x,tu)udx -\int_{\mathbb R^N}g_\epsilon(x,tv)vdx -2\lambda t\int_{\mathbb R^N}uvdx\\
\geq&t\frac{\delta_0}{2}||(U,V)||_{X_{s,\epsilon}}^2- t^pC_{p+1}^{-\frac{p+1}{2}}||(U,V)||_{X_{s,\epsilon}}^{p+1}- t^{2_s^*-1}S^{-\frac{2_s^*}{2}}||(U,V)||_{X_{s,\epsilon}}^{2_s^*},
\end{aligned}
\end{equation}
which implies that $h_{UV}\in C^1(0,\infty)$ and $ h^\prime_{UV}(t)>0$ for small $t>0.$ On the other hand,
\begin{equation*}
\begin{aligned}
h^\prime_{UV}(t)\leq&t\int_{\mathbb R^{N+1}_+} y^{1-2s} |\nabla U|^2 dxdy+t\int_{\mathbb R^{N+1}_+} y^{1-2s} |\nabla V|^2 dxdy+t\int_{\mathbb R^N}a(\epsilon x)u^2 dx\\
& +t\int_{\mathbb R^N}b(\epsilon x)v^2 dx-{t^{p}}\int_{\Lambda_\epsilon}(u^+)^{p+1}dx -{t^{2_s^*-1}}\int_{\Lambda_\epsilon}(v^+)^{2_s^*}dx -2\lambda t\int_{\mathbb R^N}uvdx\\
\to &-\infty  \quad\text{as~} t\to \infty,
\end{aligned}
\end{equation*}
we further from \eqref{2} and $h_{UV}(0)=0$ conclude  that there exists   $t_{UV}>0$
such that
$$h_{UV}(t_{UV})=\max_{t>0}h_{UV}(t).$$
To prove (I),  we only need to prove that the equation  $h^\prime_{UV}(t)=0$ has a unique  solution in $(0,\infty).$
In terms of  $h^\prime_{UV}(t)=0,$  we have

\begin{equation}\label{61}
\begin{aligned}
&\int_{\mathbb R^{N+1}_+} y^{1-2s} |\nabla U|^2 dxdy+\int_{\mathbb R^{N+1}_+} y^{1-2s} |\nabla V|^2 dxdy+\int_{\mathbb R^N}a(\epsilon x)u^2 dx\\&+\int_{\mathbb R^N}b(\epsilon x)v^2 dx+2\lambda \int_{\mathbb R^N}uvdx\\
=& \int_{\mathbb R^N}\frac{f_\epsilon(x,tu)udx}{t} +\int_{\mathbb R^N}\frac{g_\epsilon(x,tv)v}{t}dx .\\
\end{aligned}
\end{equation}
It follows from the definitions of $f_\epsilon$ and $g_\epsilon$ that $f_\epsilon(x,s)/s $ and $g_\epsilon(x,s)/s $
 are nondecreasing with respect to $s$ in $(0,\infty)$ for any fixed $x\in\mathbb R^N.$ Moreover, if $x\in \Lambda_\epsilon$ then
 $f_\epsilon(x,s)/s $ and $g_\epsilon(x,s)/s $
 are strictly increasing with respect to $s$  in $(0,\infty).$ As a result, there exists a unique $t_{UV}>0$ such  that \eqref{61} holds.
 Namely, we have shown  (I).

In view of the equality in \eqref{2} and a direct calculation, it follows that
$$I^\prime_\epsilon(t_{UV}U,t_{UV}V)(t_{UV}U,t_{UV}V)=t_{UV}h^\prime(t_{UV})=0,$$
that is, $(t_{UV}U,t_{UV}V)\in N_\epsilon.$ Let $(U,V)\in\mathcal{W},$ then by a similar estimate to \eqref{2}, we have
$$t_{UV}^2\frac{\delta_0\theta^2}{2}\leq t_{UV}^2\frac{\delta_0}{2}||(U,V)||_{X_{s,\epsilon}}^2\leq t_{UV}^{p+1}C_{p+1}^{-\frac{p+1}{2}}||(U,V)||_{X_{s,\epsilon}}^{p+1}+ t_{UV}^{2_s^*}S^{-\frac{2_s^*}{2}}||(U,V)||_{X_{s,\epsilon}}^{2_s^*}.$$
As a result, there exists a positive constant $c_\mathcal{W}$ only depending on the uniformly upper and below bounds of the norm of elements in $\mathcal{W}$   such that $t_{UV}\geq c_\mathcal{W}>0.$  Assume that there exists a sequence $\{(U_n,V_n)\}\subset\mathcal{W} $
such that $t_{U_nV_n}\to \infty$ as $n\to\infty.$ We denote $t_{U_nV_n}$ by $t_n$ for convenience.  Since $\mathcal{W}$  is compact,  there exist a subsequence denoted still by $\{(U_n,V_n)\}$ and  positive constants $c_1,c_2$ and $N_0$ such that  for any $n>N_0$, there hold
$$\int_{\Lambda_\epsilon}(u_n^+)^{p+1}dx>c_1\text{~~or~~} \int_{\Lambda_\epsilon}(v_n^+)^{2_s^*}dx>c_2. $$
We assume without loss of generality that the former one holds. Observing that
\begin{equation}\label{64}
\begin{aligned}
I_\epsilon(t_nU_n,t_nV_n)\leq&\frac{ t_n^2}{2}\int_{\mathbb R^{N+1}_+} y^{1-2s} |\nabla U_n|^2 dxdy+\frac{ t_n^2}{2}\int_{\mathbb R^{N+1}_+} y^{1-2s} |\nabla V_n|^2 dxdy+\frac{ t_n^2}{2}\int_{\mathbb R^N}a(\epsilon x)u_n^2 dx\\
& +\frac{ t_n^2}{2}\int_{\mathbb R^N}b(\epsilon x)v_n^2 dx-{\frac{ t_n^{p+1}}{p+1}}\int_{\Lambda_\epsilon}(u_n^+)^pdx -{\frac{ t_n^{2_s^*}}{2_s^*}}\int_{\Lambda_\epsilon}(v_n^+)^{2_s^*-1}dx -\lambda t_n^2\int_{\mathbb R^N}u_nv_ndx\\
\leq &t_n^2||(U_n,V_n)||_{X_{s,\epsilon}}-{\frac{ t_n^{p+1}}{p+1}}\int_{\Lambda_\epsilon}(u_n^+)^pdx -{\frac{ t_n^{2_s^*}}{2_s^*}}\int_{\Lambda_\epsilon}(v_n^+)^{2_s^*-1}dx\\
\leq&t_n^2\sup_{n}||(U_n,V_n)||_{X_{s,\epsilon}}-{\frac{ t_n^{p+1}}{p+1}}c_1 \\
\to& -\infty.
\end{aligned}
\end{equation}
This contradicts (III) in Lemma \ref{nehari}.
Consequently, (II) holds.

Suppose that there exist $\{(U_n,V_n)\}\subset X_{s,\epsilon}^+$ and $(U,V)\in X_{s,\epsilon}^+$ such that
$(U_n,V_n)\to(U,V)$ \text{in} $X_{s,\epsilon}^+.$
To verify the continuity of $\tau_\epsilon,$ we only need to prove $\mathop{\lim}\limits_{n\to\infty}t_{U_nV_n}=t_{UV}.$
Thanks to (II),  there exists $t_0>0$ such that, up to a subsequence,
$\mathop{\lim}\limits_{n\to\infty}t_{U_nV_n}=t_0.$
Let $n$ tends to infinity in the equality $I^\prime_\epsilon(t_{U_nV_n}U_n,t_{U_nV_n}V_n)(t_{U_nV_n}U_n,t_{U_nV_n}V_n)=0$,
we have
\begin{equation*}
\begin{aligned}
 0=&\frac{t_0^2}{2}\int_{\mathbb R^{N+1}_+} y^{1-2s} |\nabla U|^2 dxdy+\frac{t_0^2}{2}\int_{\mathbb R^{N+1}_+} y^{1-2s} |\nabla V|^2 dxdy+\frac{t_0^2}{2}\int_{\mathbb R^N}a(\epsilon x)u^2 dx \\
 &+\frac{t_0^2}{2}\int_{\mathbb R^N}b(\epsilon x)v^2 dx-\int_{\mathbb R^N}f_\epsilon(x,t_0u)t_0udx -
 \int_{\mathbb R^N}g_\epsilon(x,t_0v)t_0vdx -\lambda t_0^2\int_{\mathbb R^N}uvdx.
 \end{aligned}
\end{equation*}
Hence $h_{UV}^\prime(t_0)=0.$ By virtue of (I), there holds $t_{UV}=t_0.$ Therefore, $\tau_\epsilon$ is continuous.

  Define
$\varsigma_\epsilon:N_\epsilon\to S_{s,\epsilon}^+$ as
$$\varsigma_\epsilon(U,V)=\frac{(U,V)}{||(U,V)||_{X_{s,\epsilon}}}\quad\text{for any~} (U,V)\in N_\epsilon.$$
Then $\varsigma$ is well defined thanks to (II) in Lemma \ref{nehari}, and for any $(U,V)\in S_{s,\epsilon}^+$, we have
$$\varsigma\left(\overline{\tau_\epsilon}(U,V)\right)=\frac{(t_{UV}U,t_{UV}V)}{||(t_{UV}U,t_{UV}V)||_{X_{s,\epsilon}}}=
(U,V).$$
 Therefore $\overline{\tau_\epsilon}^{-1}=\varsigma_\epsilon.$ Namely, $\overline{\tau_\epsilon}$ is a homomorphism between $S_{s,\epsilon}^+$ and $N_\epsilon.$ The proof is complete.
\end{proof}

\begin{lemma}\label{eeq}
Assume $(P1)$. If $\lambda<\sqrt{a_0b_0},$ then
\begin{equation*}
\inf_{(U,V)\in N_\epsilon}I_\epsilon(U,V)
=\inf_{(U,V)\in X_{s,\epsilon}^+}\max_{t>0}I_\epsilon(tU,tV)=\inf_{\gamma\in\Gamma}\max_{t\in[0,1]}I_\epsilon(\gamma(t))=c_\epsilon.
\end{equation*}
\end{lemma}
\begin{proof}
The first equality is a direct result of  (I) and (II) of Lemma \ref{Xepsilon}. We only prove the second one.
In fact, for any $(U,V)\in X_{s,\epsilon}^+$ there exists $T>t_{UV}$ such that $I_\epsilon(TU,TV)<0.$ We define
$\gamma(t)=(tTU,tTV)$ for any $t\in [0,1].$ Then $\gamma\in\Gamma$  and
$$\max_{t\in [0,1]}I_\epsilon(\gamma(t))=\max_{t\in [0,1]}I_\epsilon(tTU,tTU)=I_\epsilon(t_{UV}U,t_{UV}V)
=\max_{t>0}I_\epsilon(tU,tV).$$
It then follows that
\begin{equation}\label{62}
\inf_{\gamma\in\Gamma}\max_{t\in[0,1]}I_\epsilon(\gamma(t))\leq\inf_{(U,V)\in X_{s,\epsilon}^+}\max_{t>0}I_\epsilon(tU,tV).
\end{equation}
On the other hand,
note that  for any $\gamma\in\Gamma,$ $I^\prime_\epsilon(\gamma(s))\gamma(s)>0$ if $s$ sufficiently  small.  Moreover,  it follows from $I_\epsilon(\gamma(1))<0$
that $\gamma(1)\in X_{s,\epsilon}^+$ and $t_{\gamma(1)}<1,$ hence $I_\epsilon^\prime(\gamma(1))(\gamma(1))<0.$ We then drive that there exists $s_0\in (0,1)$ such that $I_\epsilon^\prime(\gamma(s_0))(\gamma(s_0))=0.$ Namely, $\gamma(s_0)\in N_\epsilon.$
 As a result, we have
 $$\max_{t\in[0,1]}I_\epsilon(\gamma(t))\geq I_\epsilon(\gamma(s_0))\geq\inf_{(U,V)\in N_\epsilon}I_\epsilon(U,V),$$ which together with \eqref{62} completes the proof.
\end{proof}

Now we define a function $\Phi_\epsilon:X_{s,\epsilon}^+\to \mathbb R$ as
\begin{equation}\label{phi}
\Phi_\epsilon(U,V)=I_\epsilon(\tau_\epsilon(U,V))\quad \text{for any } (U,V)\in X_{s,\epsilon}^+,
\end{equation}
where $\tau_\epsilon$ is defined in Lemma \ref{Xepsilon}. Moreover, set $\overline{\Phi}_\epsilon={\Phi_\epsilon}|_{S_{s,\epsilon}^+}.$
Then by a similar discussion to \cite[Corollary 2.3]{szukin-weth}, we have the following results.

\begin{lemma}\label{sphere}
Assume $(P1)$. If $\lambda<\sqrt{a_0b_0},$  then
\begin{itemize}
\item[\rm (I)] $\Phi_\epsilon\in C^1(X_{s,\epsilon}^+)$ and
$$\Phi_\epsilon^\prime(U,V)(W,Z)=\frac{||\tau_\epsilon(U,V)||_{X_{s,\epsilon}}}{||(U,V)||_{X_{s,\epsilon}}}
I^\prime(\tau_\epsilon(U,V))(W,Z),~~~ \forall~(U,V)\in X_{s,\epsilon}^+,~ (W,Z)\in X_{s,\epsilon}. $$
\item [\rm (II)] $\overline{\Phi}_\epsilon\in  C^1(S_{s,\epsilon}^+)$ and
$$\overline{\Phi}_\epsilon^\prime(U,V)(W,Z)=||\overline{\tau}_\epsilon(U,V)||_{X_{s,\epsilon}}
I^\prime(\overline{\tau}_\epsilon(U,V))(W,Z),~~~ \forall~(U,V)\in S_{s,\epsilon}^+,~ (W,Z)\in T_{UV}S_{s,\epsilon}^+,$$
where $$ T_{UV}S_{s,\epsilon}^+:=\{(W,Z)\in{X_{s,\epsilon}}:\left<(U,V),(W,Z)\right>_{X_{s,\epsilon}} =0\}.$$
\item[\rm (III)] If $\{(U_n,V_n)\}\subset S_{s,\epsilon}^+$ is a $(PS)_d$ sequence of $\overline{\Phi}_\epsilon$, then
$\{\overline{\tau}_\epsilon(U_n,V_n)\}$ is a $(PS)_d$ sequence of $I_\epsilon.$ Moreover, if  $\{(U_n,V_n)\}\subset N_\epsilon$
is a bounded $(PS)_d$ sequence of $I_\epsilon,$ then $\overline{\tau}_\epsilon^{-1}(U_n,V_n)$ is a  $(PS)_d$ sequence of $\overline{\Phi}_\epsilon$.
\item[\rm (IV)] $(U,V)\in  S_{s,\epsilon}^+$ is a critical point of $\overline{\Phi}_\epsilon$ if and only if $\overline{\tau}_\epsilon(U,V)$ is a
critical point of $I_\epsilon.$ Moreover, the critical values coincide and
$$\inf_{S_\epsilon^+}\overline{\Phi}_\epsilon=\inf_{N_\epsilon} I_\epsilon.$$
\end{itemize}
\end{lemma}

\begin{lemma} \label{psbound}
Assume $(P1)$ and $\lambda<\sqrt{a_0b_0}$.
If $\{(U_n, V_n)\} \subset X_{s,\epsilon}$ is a $(PS)_d$ sequence for $I_\epsilon$ with $d>0,$ then
\begin{itemize}
\item[\rm (I)] there exists $C_0>0$ such that
$||(U_n, V_n)||_{X_{s,\epsilon}}\leq C_0.$
\item[\rm (II)] For any $\theta>0$, there exists $r=r(\theta)$ such that
$$\limsup_{n\to \infty}\left(\int_{\mathbb R^{N+1}_+\setminus \mathcal B^+_r}y^{1-2s}(|\nabla U_n|^2+|\nabla V_n|^2)dxdy
+\int_{\mathbb{R}^N\setminus  B_r}a(\epsilon x)u_n^2+b(\epsilon x) v_n^2dx\right)<\theta. $$
\end{itemize}
\end{lemma}
\begin{proof}
Let $\{(U_n, V_n)\} \subset X_{s,\epsilon}$ be a $(PS)_d$ sequence for $I_\epsilon$ with $d>0,$ then it follows from \eqref{fepsilon},
\eqref{gepsilon}, \eqref{arfout}  and
\eqref{delta0} that
\begin{align*}
&d+o_n(1)+o\left(||(U_n,V_n)||_{X_{s,\epsilon}}\right)\\=&I_\epsilon(U_n,V_n)-\frac{1}{p+1}I^\prime_\epsilon(U_n,V_n)(U_n,V_n)\\
\geq& \left(\frac{1}{2}-\frac{1}{p+1}\right)\left(||(U_n,V_n)||^2_{X_{s,\epsilon}}-\int_{\mathbb R^N\setminus\Lambda_\epsilon}f_\epsilon(x,u_n)u_ndx-\int_{\mathbb R^N\setminus\Lambda_\epsilon}g_\epsilon(x,v_n)v_ndx
-2\lambda\int_{\mathbb R^N}u_nv_ndx\right)\\
\geq&\left(\frac{1}{2}-\frac{1}{p+1}\right)\left(||(U_n,V_n)||^2_{X_{s,\epsilon}}-\alpha^{p-1}\int_{\mathbb R^N}u_n^2dx -\alpha^{2_s^*-2}\int_{\mathbb R^N}v_n^2dx-2\lambda\int_{\mathbb R^N}u_nv_ndx\right)\\
\geq &\frac{\delta_0}{2}\left(\frac{1}{2}-\frac{1}{p+1}\right)||(U_n,V_n)||^2_{X_{s,\epsilon}}.
\end{align*}
Hence there exists $C_0>0$ such that
$||(U_n, U_n)||_{X_{s,\epsilon}}\leq C_0.$
Namely, (I) holds.

For any fixed $\epsilon,$ we first choose a constant $r>1$ such that $\Lambda_\epsilon\subset\subset  B_{\frac{r}{2}}$
and a cut-off function $\eta\in C^\infty(\overline{\mathbb R^{N+1}_+})$
such that $\eta(x,y)=1$ if $(x,y)\in \overline{\mathbb R^{N+1}_+}\setminus \mathcal B_r^+,$ $\eta(x,y)=0$ if $(x,y)\in \mathcal B_{\frac{r}{2}}^+,$ $0\leq \eta(x,y)\leq 1$ and $|\nabla \eta(x,y)|\leq\frac{C}{r}$ with some positive constant $C$ for any $(x,y)\in \overline{\mathbb R^{N+1}_+}.$
It then follows that $\{(\eta U_n,\eta V_n )\}\subset X_{s,\epsilon}$  is bounded due to (I). Hence,
$$I^\prime_\epsilon(U_n,V_n)(\eta U_n,\eta V_n)=o_n(1).$$
By a direct calculation, this is equivalent to
\begin{equation}\label{3}
\begin{aligned}
&\int_{\mathbb R^{N+1}_+}y^{1-2s}(|\nabla U_n|^2+|\nabla V_n|^2)\eta(x,y)dxdy
+\int_{\mathbb{R}^N}\left(a(\epsilon x)u_n^2+b(\epsilon x) v_n^2\right)\eta(x,0)dx\\
&-\int_{\mathbb R^N}f_\epsilon(x,u_n)u_n\eta(x,0)dx -
 \int_{\mathbb R^N}g_\epsilon(x,v_n)v_n\eta(x,0)dx -2\lambda\int_{\mathbb R^N}u_nv_n\eta^2(x,0)dx\\
 =&o_n(1)-\int_{\mathbb R^{N+1}_+}y^{1-2s}\nabla U_n\cdot\nabla\eta U_ndxdy-\int_{\mathbb R^{N+1}_+}y^{1-2s}\nabla V_n\cdot\nabla\eta V_ndxdy.
\end{aligned}
\end{equation}
Since $\Lambda_\epsilon\subset\subset  B_{\frac{r}{2}}$ and $\eta(x,y)=0$ if $(x,y)\in \mathcal B^+_{\frac{r}{2}},$
we then find
\begin{equation*}
\begin{aligned}
&\frac{\delta_0}{2}\left(\int_{\mathbb R^{N+1}_+}y^{1-2s}(|\nabla U_n|^2+|\nabla V_n|^2)\eta(x,y)dxdy
+\int_{\mathbb{R}^N}\left(a(\epsilon x)u_n^2+b(\epsilon x) v_n^2\right)\eta(x,0)dx\right)\\
\leq& o_n(1)+\frac{C}{r}\left(\int_{\mathcal B_r^+\setminus \mathcal B_{\frac{r}{2}}^+}y^{1-2s}|\nabla U_n|^2dxdy+\int_{\mathcal B_r^+\setminus \mathcal B_{\frac{r}{2}}^+}y^{1-2s} U_n^2dxdy+\int_{\mathcal B_r^+\setminus \mathcal B_{\frac{r}{2}}^+}y^{1-2s}|\nabla V_n|^2dxdy\right . \\
& \left.\qquad\qquad\qquad +\int_{\mathcal B_r^+\setminus \mathcal B_{\frac{r}{2}}^+}y^{1-2s} V_n^2dxdy\right),
\end{aligned}
\end{equation*}
It then follows from Lemma \ref{59}, (I) and the definition of $\eta$ that
\begin{equation*}
\begin{aligned}
\int_{\mathbb R^{N+1}_+\setminus \mathcal B^+_r}y^{1-2s}(|\nabla U_n|^2+|\nabla V_n|^2)dxdy
+\int_{\mathbb{R}^N\setminus  B_r}a(\epsilon x)u_n^2+b(\epsilon x) v_n^2dx
\leq o_n(1)+\frac{C}{r}.
\end{aligned}
\end{equation*}
Let $r>\frac{2C}{\theta},$ then  (II) holds. The proof is complete.
\end{proof}

Next, we conclude a compactness for the (PS) sequence of $I_\epsilon$. Firstly we state a version of the concentration compactness principle in \cite{Dipierro-Medina-Valdinoci}.

\begin{lemma}\label{68}
Assume  that $\{U_n\}\subset\chi_s$  converges weakly to $U$ in $\chi_s.$ Let $\mu,\nu$ be two nonnegative measures  on $\mathbb R^{N+1}_+$ and $\mathbb R^N$ respectively and such that
$$\lim_{n\to\infty}y^{1-2s}|\nabla U_n|^2\to \mu,\quad \lim_{n\to\infty}|u_n|^{2_s^*}\to \nu$$
in the sense of measure. If for any $\delta>0$ there exists $\rho>0$ such that
$$\int_{\mathbb R^{N+1}_+\setminus\mathcal B_\rho^+}y^{1-2s}|\nabla U_n|^2dxdy\leq\delta\quad \text{for any~} n\in\mathbb N.$$
Then there exist an at most countable set $J$ and three families $\{x_j\}_{j\in J}\subset\mathbb R^N,$
$\{\mu_j\}_{j\in J}$ and $\{\nu_j\}_{j\in J},$ $\mu_j,\nu_j\geq 0$ such that
$$\mu\geq y^{1-2s}|\nabla U|^2+\sum_{j\in J}\mu_j\delta_{(x_j,0)},~~ \nu= |u|^{2_s^*}+\sum_{j\in J}\nu_j\delta_{x_j},~~
\mu_j\geq S\nu_j^{2/2_s^*}  \text{~for any~} j\in J.$$
\end{lemma}

\begin{lemma}\label{pscondition}
Assume $(P1)$ and $\lambda<\sqrt{a_0b_0}.$ If $\{(U_n, V_n)\} \subset X_{s,\epsilon}$ is a $(PS)_d$ sequence for $I_\epsilon$ with $0<d<\frac{s}{N}S^{\frac{N}{2s}},$ then
there exists a convergent subsequence.
\end{lemma}

\begin{proof}
Let $\{(U_n, V_n)\} \subset X_{s,\epsilon}$ be a $(PS)_d$  sequence for $I_\epsilon$ with $0<d<\frac{s}{N}S^{\frac{N}{2s}},$ then there exists $C_0>0$ such that
$||(U_n, V_n)||_{X_{s,\epsilon}}\leq C_0$   by (I) in Lemma \ref{psbound}.
Consequently, we have a subsequence (denoted still by $\{(U_n, V_n)\}$ for convenience) and
$(U_\epsilon,V_\epsilon)$ such that
\begin{equation*}
(U_n,V_n)\rightharpoonup (U_\epsilon,V_\epsilon) \quad \text{ in}~X_{s,\epsilon},
\end{equation*}
\begin{equation*}
(U_n,V_n)\to (U_\epsilon,V_\epsilon) \quad \text{a.e. in}~\mathbb R^{N+1}_+,
\end{equation*}
\begin{equation}\label{stronglycon}
(u_n,v_n)\to (u_\epsilon,v_\epsilon) \quad \text{in}~L_{\text{loc}}^s(\mathbb R^N) \times L_{\text{loc}}^t(\mathbb R^N) ~\text{with}~ 1\leq s,t< 2_s^*.
\end{equation}
It then follows from $I^\prime_\epsilon(U_n,V_n)\to 0$ that $(U_\epsilon,V_\epsilon)$ is a critical point of $I_\epsilon.$
Thanks to \eqref{fepsilon},\eqref{gepsilon} and \eqref{stronglycon}, we have
$$f_\epsilon(x,u_n)\to f_\epsilon(x,u_\epsilon), ~~g_\epsilon(x,u_n)\to g_\epsilon(x,u_\epsilon)\quad \text{in}~
L_{\text{loc}}^1(\mathbb R^N).$$
By virtue of (II) in Lemma \ref{psbound}, Lemma \ref{68} and \eqref{fepsilon}, we can conclude that for any $\theta>0$ there exist $r=r(\theta)$ and $N_0>0$
such that $\Lambda_\epsilon\subset\subset B_r,$ and
\begin{align*}
\int_{B_r}|f_\epsilon(x,u_n)u_n-f_\epsilon(x,u_\epsilon)u_\epsilon|dx&\leq\frac{\theta}{3},\quad\text{for any~} n>N_0,\\
\int_{\mathbb R^N\setminus B_r}|f_\epsilon(x,u_n)u_n|dx\leq\alpha^{p-1}\int_{\mathbb R^N\setminus B_r}u_n^2dx&\leq\frac{\theta}{3},\quad\text{for any~} n>N_0,\\
\int_{\mathbb R^N\setminus B_r}|f_\epsilon(x,u_\epsilon)u_\epsilon| dx\leq\alpha^{p-1}\int_{\mathbb R^N\setminus B_r}u_\epsilon^2dx&\leq\frac{\theta}{3}.
\end{align*}
As a consequence, for any $n>N_0$, we find that
\begin{equation*}
\begin{aligned}
&\int_{\mathbb R^N}|f_\epsilon(x,u_n)u_n-f_\epsilon(x,u_\epsilon)u_\epsilon|dx\\
\leq &\int_{ B_r}|f_\epsilon(x,u_n)u_n-f_\epsilon(x,u_\epsilon)u_\epsilon|dx+\int_{\mathbb R^N\setminus B_r}|f_\epsilon(x,u_n)u_n|dx+\int_{\mathbb R^N\setminus B_r}|f_\epsilon(x,u_\epsilon)u_\epsilon| dx\\
\leq &\theta,
\end{aligned}
\end{equation*}
that is
\begin{equation}\label{4}
\lim_{n\to\infty}\int_{\mathbb R^N}f_\epsilon(x,u_n)u_ndx=\int_{\mathbb R^N}f_\epsilon(x,u_\epsilon)u_\epsilon dx.
\end{equation}
In a similar manner, we obtain that $u_nv_n\to u_\epsilon v_\epsilon$ in $L^1(\mathbb R^N)$ as $n\to\infty$.
It follows from $I^\prime(U_n,V_n)(U_n,0)=o_n(1)$ that
\begin{equation}\label{6}
\begin{aligned}
&\int_{\mathbb R^{N+1}_+} y^{1-2s} |\nabla U_n|^2 dxdy+\int_{\mathbb R^N}a(\epsilon x)u_n^2 dx\\
=&o_n(1)+\int_{\mathbb R^N}f_\epsilon(x,u_n)u_ndx +\lambda\int_{\mathbb R^N}u_nv_ndx\\
\to &\int_{\mathbb R^N}f_\epsilon(x,u_\epsilon)u_\epsilon dx +\lambda\int_{\mathbb R^N}u_\epsilon v_\epsilon dx\\
=&\int_{\mathbb R^{N+1}_+} y^{1-2s} |\nabla U_\epsilon|^2 dxdy+\int_{\mathbb R^N}a(\epsilon x)u_\epsilon^2 dx.
\end{aligned}
\end{equation}
By using (II) in Lemma \eqref{psbound} and the concentration compactness principle, we readily conclude that there exist an at most countable
set $J$ and three families $\{x_j\}_{j\in J},$ $\{\mu_j\}_{j\in J}$ and $\{\nu_j\}_{j\in J}$  with $\mu_j\geq 0,$ $\nu_j\geq0$ for any $j\in J$ such that
$$\mu\geq y^{1-2s}|\nabla V^+_\epsilon|^2+\sum_{j\in J}\mu_j\delta_{(x_j,0)},~
\nu=|v^+_\epsilon|^{2_s^*}+\sum_{j\in J}\nu_j\delta_{x_j},~
\mu_j>S\nu_j^{\frac{2}{2_s^*}}\quad \text{for any}\ j\in J.$$
If $\nu_j=0$   for any $j\in J$, then
$$\lim_{n\to \infty}\int_{\mathbb R^N}|v_n^+|^{2_s^*}dx=\int_{\mathbb R^N}|v_\epsilon^+|^{2_s^*}dx.$$
If there is some $j\in J$ such that $\nu_j\neq0,$  then we claim that $x_j\notin \Lambda_\epsilon.$ Otherwise, there exists $\rho>0$ such that $B_\rho(x_j)\subset\subset \Lambda.$  Let $\eta\in C_c^\infty(\overline{\mathbb R^{N+1}_+})$ be a cut-off function such that $\eta(x,y)=1$ in $\mathcal{B}_{\frac{1}{2}}^+$, $\eta(x,y)=0$ in $\overline{\mathbb R^{N+1}_+}\setminus\mathcal{B}_1^+$ and $0\leq\eta(x,y)\leq 1$ for any $(x,y)\in \overline{\mathbb R^{N+1}_+}.$
Define
$$\eta_\rho(x,y)=\eta\left(\frac{x-x_i}{\rho},\frac{y}{\rho}\right)\quad\text{for any~}(x,y)\in \overline{\mathbb R^{N+1}_+}.$$
Owing to   $I^\prime(U_n,V_n)(0,V_n\eta_\rho)=o_n(1),$  we obtain
\begin{equation}\label{5}
\begin{aligned}
&\int_{\mathbb R^{N+1}_+}y^{1-2s}|\nabla V_n|^2\eta_\rho(x,y)dxdy
+\int_{\mathbb{R}^N}b(\epsilon x) v_n^2\eta_\rho(x,0)dx\\
&=o_n(1)+\int_{\mathbb R^N}g_\epsilon(x,v_n)v_n\eta_\rho(x,0)dx +\lambda\int_{\mathbb R^N}u_nv_n\eta_\rho(x,0)dx-\int_{\mathbb R^{N+1}_+}y^{1-2s}\nabla V_n\cdot\nabla\eta_\rho V_ndxdy.
\end{aligned}
\end{equation}
Noting from (P1) and \eqref{stronglycon} that
$$\lim_{\rho\to 0}\lim_{n\to \infty}\int_{\mathbb R^N}b(\epsilon x)v^2_n\eta_\rho(x,0)dx=\lim_{\rho\to 0}\int_{B_\rho(x_j)}b(\epsilon x)v^2_\epsilon\eta_\rho(x,0) dx=0.$$
Similarly,
$$\lim_{\rho\to 0}\lim_{n\to \infty}\int_{\mathbb R^N}u_nv_n\eta_\rho(x,0)dx=0.$$
By virtue of $supp~ \eta_\rho(\cdot,0)\subset B_\rho(x_j)\subset\subset \Lambda_\epsilon$ and \eqref{gepsilon},
\begin{equation*}
\begin{aligned}
&\lim_{\rho\to 0}\limsup_{n\to\infty}\int_{\mathbb R^N}g_\epsilon(x,v_n)v_n\eta_\rho(x,0)dx
=\lim_{\rho\to 0}\limsup_{n\to\infty}\int_{\mathbb R^N}|v^+_n|^{2_s^*}\eta_\rho(x,0) dx\\
=&\lim_{\rho\to 0}\int_{\mathbb R^N}(|v_\epsilon^+|^{2_s^*}+\sum_{j\in J}\nu_j\delta_{x_j})\eta_\rho(x,0) dx=\nu_j.
\end{aligned}
\end{equation*}
Similarly,
$$\lim_{\rho\to 0}\limsup_{n\to\infty}\int_{\mathbb R^{N+1}_+}y^{1-2s}|\nabla V_n|^2\eta_\rho(x,y)dxdy\geq
\lim_{\rho\to 0}\limsup_{n\to\infty}\int_{\mathbb R^{N+1}_+}y^{1-2s}|\nabla V^+_n|^2\eta_\rho(x,y)dxdy\geq \mu_j.$$
Note that
\begin{equation}\label{a}
\begin{aligned}
\int_{\mathbb R^{N+1}_+}y^{1-2s}\nabla V_n\cdot\nabla\eta_\rho V_ndxdy=&\int_{\mathbb R^{N+1}_+}y^{1-2s}\nabla V_n\cdot\nabla\eta_\rho (V_n-V_\epsilon)dxdy\\
&+\int_{\mathbb R^{N+1}_+}y^{1-2s}\nabla V_n\cdot\nabla\eta_\rho V_\epsilon dxdy\\
:=&A_1+A_2
\end{aligned}
\end{equation}
Using the H\"{o}lder inequality, we deduce
\begin{equation*}
|A_1|
\leq\frac{1}{\rho}\left(\int_{\mathcal B_\rho^+((x_j,0)) }y^{1-2s}|\nabla V_n|^{2}dxdy\right)^{\frac{1}{2}}
\left(\int_{\mathcal B_\rho^+((x_j,0)) }y^{1-2s}|V_n-V_\epsilon|^{2}dxdy\right)^{\frac{1}{2}},
\end{equation*}
It follows  from Lemma \ref{59} that $\mathop{\lim}\limits_{n\to\infty}A_1=0.$
From the  H\"{o}lder inequality again, we have
\begin{equation*}
\begin{aligned}
|A_2|\leq& \frac{1}{\rho}\left(\int_{\mathbb R^{N+1}_+}y^{1-2s}|\nabla V_n|^2dxdy\right)^{\frac{1}{2}}
\left(\int_{\mathcal B^+_\rho((x_j,0))}y^{1-2s}|V_\epsilon|^{2\gamma}dxdy\right)^{\frac{1}{2\gamma}}
\left(\int_{\mathcal B^+_\rho((x_j,0))}y^{1-2s}dxdy\right)^{\frac{\gamma-1}{2\gamma}}\\
\leq&\rho^{\frac{(N+2-2s)(\gamma-1)}{2\gamma}-1}\sup_{n}||V_n||_{\chi_s}\left(\int_{\mathcal B^+_\rho((x_j,0))}y^{1-2s}|V_\epsilon|^{2\gamma}dxdy\right)^{\frac{1}{2\gamma}}.
\end{aligned}
\end{equation*}
In view of $\gamma=1+\frac{2}{N-2s},$ there holds $\frac{(N+2-2s)(\gamma-1)}{2\gamma}-1=0,$ which together with Lemma \ref{59}  implies that
$$\lim_{\rho\to0}\mathop{\lim}\limits_{n\to\infty}A_2=0.$$
We further from \eqref{a}  conclude that
$$\lim_{\rho\to 0}\limsup_{n\to\infty}\int_{\mathbb R^{N+1}_+}y^{1-2s}\nabla V_n\cdot\nabla\eta_\rho V_ndxdy=0.$$
By passing to the limit as $n\to\infty$ and then $\rho\to0$ in \eqref{5}, we have
$\mu_j\leq\nu_j,$ which combined with $\mu_j\geq S \nu_j^{\frac{2}{2_s^*}}$ and $\nu_j>0$ further shows that $\nu_j>S^{\frac{N}{2s}}.$ Since  $0<d<\frac{s}{N}S^{\frac{N}{2s}},$ there exists $N_0>0$ such that for any $n>N_0$,
\begin{equation*}
\begin{aligned}
\frac{s}{N}S^{\frac{N}{2s}}>d&=I_\epsilon(U_n,V_n)-\frac{1}{2}I^\prime_\epsilon(U_n,V_n)(U_n,V_n)+o_n(1)\\
&=\int_{\mathbb R^N}\frac{1}{2}f_\epsilon(x,u_n)u_n-F_\epsilon(x,u_n) dx +
\int_{\mathbb R^N}\frac{1}{2}g_\epsilon(x,v_n)v_n-G_\epsilon(x,v_n) dx+o_n(1)\\
&\geq \frac{s}{N}\int_{\Lambda_\epsilon}|v_n^+|^{2_s^*} dx+o_n(1)\\
&\geq \frac{s}{N}\int_{\mathbb R^N}|v_n^+|^{2_s^*}\eta_\rho(x,0) dx+o_n(1)\\
&\to\frac{s}{N}\int_{B_\rho(x_j)}|v_\epsilon^+|^{2_s^*}\eta_\rho(x,0)dx+\frac{s}{N}\int_{\mathbb R^N}\sum_{j\in J}\nu_j\delta_{x_j}\eta_\rho(x,0)dx\\
&\geq\frac{s}{N}\int_{B_\rho(x_j)}|v_\epsilon^+|^{2_s^*}\eta_\rho(x,0)dx +\frac{s}{N}S^{\frac{N}{2s}},
\end{aligned}
\end{equation*}
we reach a contradiction. As a consequence, $x_j\notin \Lambda_\epsilon$ and then
\begin{equation*}
\lim_{n\to \infty}\int_{\Lambda_\epsilon} |v_n^+|^{2_s^*} dx =\int_{\Lambda_\epsilon} |v_\epsilon^+|^{2_s^*} dx.
\end{equation*}
By a similar discussion to \eqref{4}, we infer
\begin{equation*}
\lim_{n\to\infty}\int_{\mathbb R^N\setminus \Lambda_\epsilon}g_\epsilon(x,v_n)v_ndx=\int_{\mathbb R^N\setminus \Lambda_\epsilon}g_\epsilon(x,v_\epsilon)v_\epsilon dx.
\end{equation*}
To this end, it follows from $I^\prime_\epsilon(U_n,V_n)(0,V_n)\to 0$ that
\begin{equation*}
\begin{aligned}
\end{aligned}
\begin{aligned}
&\int_{\mathbb R^{N+1}_+}y^{1-2s}|\nabla V_n|^2dxdy
+\int_{\mathbb{R}^N}b(\epsilon x) v_n^2dx\\
=&o_n(1)+\int_{ \Lambda_\epsilon}|v_n^+|^{2_s^*}dx+\int_{\mathbb R^N\setminus \Lambda_\epsilon}g_\epsilon(x,v_n)v_ndx +\lambda\int_{\mathbb R^N}u_nv_ndx\\
\to&\int_{ \Lambda_\epsilon}|v_\epsilon^+|^{2_s^*}dx+\int_{\mathbb R^N\setminus \Lambda_\epsilon}g_\epsilon(x,v_\epsilon)v_\epsilon dx +\lambda\int_{\mathbb R^N}u_\epsilon v_\epsilon dx\\
=&\int_{\mathbb R^N}g_\epsilon(x,v_\epsilon)v_\epsilon dx +\lambda\int_{\mathbb R^N}u_\epsilon v_\epsilon dx\\
=&\int_{\mathbb R^{N+1}_+}y^{1-2s}|\nabla V_\epsilon|^2dxdy
+\int_{\mathbb{R}^N}b(\epsilon x) v_\epsilon^2dx,
\end{aligned}
\end{equation*}
which together with \eqref{6} implies
$(U_n,V_n)\to (U_\epsilon,V_\epsilon)$ {in} $X_{s,\epsilon}.$
The proof is complete.
\end{proof}
\begin{cor}\label{spherepscondition}
Assume $(P1)$ and $\lambda<\sqrt{a_0b_0}.$ Then
the function $\overline{\Phi}_\epsilon$ defined by \eqref{phi} satisfies $(PS)_d$  condition  with $0<d<\frac{s}{N}S^{\frac{N}{2s}}$ on $S_{s,\epsilon}^+.$
\end{cor}
\begin{proof}
Assume that $\{(U_n,V_n)\}\subset S_{s,\epsilon}^+ $ is a $(PS)_d$ sequence with $0<d<\frac{s}{N}S^{\frac{N}{2s}}$ for
$\overline{\Phi}_\epsilon$, then $\{\overline{\tau}_\epsilon (U_n,V_n)\}$ is a $(PS)_d$ sequence for $I_\epsilon$ due to
Lemma \ref{sphere}. It then follows from Lemma \ref{pscondition} that there exist  $(U,V)\in X_{s,\epsilon}$
and a subsequence (still denoted by $\{\overline{\tau}_\epsilon (U_n,V_n)\}$) such that
\begin{equation}\label{7}
\overline{\tau}_\epsilon (U_n,V_n)\to (U,V)\quad \text{in~} X_{s,\epsilon}.
\end{equation}
 We then by(I) and (II) in Lemma \ref{Xepsilon} infer that $(U,V)\in N_\epsilon$ and $\overline{\tau}_\epsilon^{-1}(U,V)\in S_{s,\epsilon}^+. $
Thanks to \eqref{7} and (III) in Lemma \ref{Xepsilon}, we have
$$ (U_n,V_n)\to \overline{\tau}^{-1}_\epsilon(U,V)\quad \text{in~} X_{s,\epsilon}.$$
The proof is complete.
\end{proof}
\begin{lemma}\label{energybound}
Assume $(P1)$ and $\lambda<\sqrt{a_0b_0}.$
If there exists $x_0\in \Lambda$ such that system \eqref{autonomoussys} has a positive vector ground state with  energy $c_{x_0}$ less than $\frac{s}{N}S^{\frac{N}{2s}},$ then
$$\limsup_{\epsilon\to0}c_\epsilon\leq c_{x_0}<\frac{s}{N}S^{\frac{N}{2s}}.$$
  Moreover, system \eqref{mp} admits a positive vector ground state $(U_\epsilon,V_\epsilon)\in X_{s,\epsilon}$.
\end{lemma}
\begin{proof}
 Let $(U,V)$ be a positive vector ground state of \eqref{autonomoussys} with  $x_0\in \Lambda$, then there exists $r>0$ such that
$B_r(x_o)\subset\subset\Lambda.$ Take a cut-off function  $\eta\in C_c^\infty (\overline{\mathbb R^{N+1}_+})$
such that
$\eta(x,y)=1$ if $(x,y)\in \mathcal B_{\frac{r}{2}}^+,$ $\eta(x,y)=0$ if $(x,y)\in \overline{\mathbb R^{N+1}_+}\setminus\mathcal B_{r}^+,$ $1\leq\eta(x,y)\leq1$ and $|\nabla \eta(x,y)|\leq \frac{C}{r}$ with constant $C>0$ for any $(x,y)\in \overline{\mathbb R^{N+1}_+}.$ Set
\begin{equation}\label{wide}
\left(\widetilde{U}_{\epsilon,x_0}(x,y),\widetilde{V}_{\epsilon,x_0}(x,y)\right)=\left(U\left(x-\frac{x_0}{\epsilon},y\right)
\eta(\epsilon x-x_0, \epsilon y),V\left(x-\frac{x_0}{\epsilon},y\right)\eta(\epsilon x-x_0, \epsilon y)\right).
\end{equation}
We omit the subscript $x_0$ for convenience.
Then   there exists $t_\epsilon>0$ such that $(t_\epsilon\widetilde{U}_\epsilon,t_\epsilon\widetilde{V}_\epsilon)\in N_\epsilon.$ Moreover, by a straight calculation, there exist constants $\theta,\widetilde C>0$ independent of $\epsilon$
such that $$0<\theta \leq ||(\widetilde{U}_\epsilon,\widetilde{V}_\epsilon)||_{X_{s,\epsilon}}\leq \widetilde C. $$
Noting that $supp~ \tilde{u}_\epsilon\subset B_{\frac{r}{\epsilon}}(\frac{x_0}{\epsilon})\subset\subset\Lambda_\epsilon$
and $supp~ \tilde{v}_\epsilon\subset B_{\frac{r}{\epsilon}}(\frac{x_0}{\epsilon})\subset\subset\Lambda_\epsilon,$ we obtain by a change of variables $\bar{x}=x-\frac{x_0}{\epsilon}$ that
\begin{equation*}
\int_{\Lambda_\epsilon}(\widetilde u_\epsilon^+)^{p+1}dx\geq \int_{B_{\frac{r}{2\epsilon}}}u(\bar x)^{p+1}d\bar x
\geq \int_{B_{\frac{r}{2}}}u(\bar x)^{p+1}d\bar x:=c_1 >0.
\end{equation*}
Similar arguments to \eqref{64} further implies that $\{t_\epsilon\} $ is bounded.
Moreover, if there exists a subsequence denoted still by $\{t_\epsilon\}$ such that $t_\epsilon\to t_0$ as $\epsilon\to 0,$
then $t_0>0$ by a  similar argument to \eqref{2}.
Furthermore, from  a straight calculation and  changes of variables, we conclude that
\begin{align*}
I_\epsilon(t_\epsilon\widetilde{U}_\epsilon,t_\epsilon\widetilde{V}_\epsilon)=&\frac{t_\epsilon^2}{2}\int_{\mathbb R^{N+1}_+}y^{1-2s}|\nabla U(\bar{x},y)|^2|\eta(\epsilon \bar{x},\epsilon y)|^2d\bar{x}dy\\
&+\frac{t_\epsilon^2}{2}\int_{\mathbb R^{N+1}_+}y^{1-2s}|\nabla V(\bar{x},y)|^2|\eta(\epsilon \bar{x},\epsilon y)|^2d\bar{x}dy\\
&+\frac{t_\epsilon^2\epsilon^2}{2}\int_{\mathbb R^{N+1}_+}y^{1-2s}|\nabla \eta(\epsilon\bar{x},\epsilon y)|^2|U( \bar{x}, y)|^2d\bar{x}dy\\
&+\frac{t_\epsilon^2\epsilon^2}{2}\int_{\mathbb R^{N+1}_+}y^{1-2s}|\nabla \eta(\epsilon\bar{x},\epsilon y)|^2|V( \bar{x}, y)|^2d\bar xdy\\
&+t_\epsilon\epsilon \int_{\mathbb R^{N+1}_+}y^{1-2s}\nabla \eta(\epsilon\bar{x},\epsilon y)\cdot \nabla U(\bar{x},y)  U( \bar{x}, y)\eta(\epsilon\bar{x},\epsilon y) d\bar{x}dy\\
&+t_\epsilon\epsilon \int_{\mathbb R^{N+1}_+}y^{1-2s}\nabla \eta(\epsilon\bar{x},\epsilon y)\cdot \nabla V(\bar{x},y)  V( \bar{x}, y)\eta(\epsilon\bar{x},\epsilon y) d\bar{x}dy\\
&+\frac{t_\epsilon^2}{2}\int_{\mathbb R^N}a(\epsilon \bar{x}+x_0) (u(\bar{x}))^2(\eta(\epsilon \bar{x},0))^2d\bar{x}\\
&+\frac{t_\epsilon^2}{2}\int_{\mathbb R^N}b(\epsilon \bar{x}+x_0) (v(\bar{x}))^2(\eta(\epsilon \bar{x},0))^2d\bar{x}\\
&-\frac{t_\epsilon^{p+1}}{p+1}\int_{\mathbb R^N}(u(\bar{x}))^{p+1}(\eta(\epsilon \bar{x},0))^{p+1}d\bar{x}
-\frac{t_\epsilon^{2_s^*}}{2_s^*}\int_{\mathbb R^N}(v(\bar{x}))^{2_s^*}(\eta(\epsilon \bar{x},0))^{2_s^*}d\bar{x}\\
&-\lambda t_\epsilon^2\int_{\mathbb R^N}u(\bar{x})v(\bar{x})(\eta(\epsilon \bar{x},0))^2d\bar{x}.
\end{align*}
By Lemma \ref{59}, H\"{o}lder inequality and $\gamma=1+\frac{2}{N-2s},$  we have
\begin{align*}
&\frac{t_\epsilon^2\epsilon^2}{2}\int_{\mathbb R^{N+1}_+}y^{1-2s}|\nabla \eta(\epsilon\bar{x},\epsilon y)|^2|U( \bar{x}, y)|^2d\bar{x}dy\\
\leq&\frac{Ct_\epsilon^2\epsilon^2}{2r^2}\left(\int_{\mathcal B^+_{\frac{r}{\epsilon}}\setminus\mathcal B^+_{\frac{r}{2\epsilon}}}y^{1-2s} U^{2\gamma}d\bar{x}dy\right)^{\frac{1}{\gamma}}
\left(\int_{\mathcal B^+_{\frac{r}{\epsilon}}\setminus\mathcal B^+_{\frac{r}{2\epsilon}}}y^{1-2s}d\bar{x}dy\right)^{\frac{\gamma-1}{\gamma}}\\
\leq&\frac{Ct_\epsilon^2}{2}\left(\int_{\mathcal B^+_{\frac{r}{\epsilon}}\setminus\mathcal B^+_{\frac{r}{2\epsilon}}}y^{1-2s} U^{2\gamma}d\bar{x}dy\right)^{\frac{1}{\gamma}}
\to 0 \quad{as~} \epsilon\to 0,
\end{align*}
and
\begin{equation*}
\begin{aligned}
&\epsilon \int_{\mathbb R^{N+1}_+}y^{1-2s}\nabla \eta(\epsilon\bar{x},\epsilon y)\cdot \nabla U(\bar{x},y)  U( \bar{x}, y)\eta(\epsilon\bar{x},\epsilon y) d\bar{x}dy\\
\leq&\frac{\epsilon}{r}\left(\int_{\mathcal B^+_{\frac{r}{\epsilon}}\setminus\mathcal B^+_{\frac{r}{2\epsilon}}}
y^{1-2s}|\nabla U|^2d\bar{x}dy\right)^{\frac{1}{2}}
\left(\int_{\mathcal B^+_{\frac{r}{\epsilon}}\setminus\mathcal B^+_{\frac{r}{2\epsilon}}}y^{1-2s} U^{2\gamma}d\bar{x}dy\right)^{\frac{1}{2\gamma}}
\left(\int_{\mathcal B^+_{\frac{r}{\epsilon}}\setminus\mathcal B^+_{\frac{r}{2\epsilon}}}y^{1-2s}d\bar{x}dy\right)^{\frac{\gamma-1}{2\gamma}}\\
\leq&C\left(\int_{\mathcal B^+_{\frac{r}{\epsilon}}\setminus\mathcal B^+_{\frac{r}{2\epsilon}}}
y^{1-2s}|\nabla U|^2d\bar{x}dy\right)^{\frac{1}{2}}
\left(\int_{\mathcal B^+_{\frac{r}{\epsilon}}\setminus\mathcal B^+_{\frac{r}{2\epsilon}}}y^{1-2s} U^{2\gamma}d\bar{x}dy\right)^{\frac{1}{2\gamma}}
\to0\quad{as~} \epsilon\to 0.
\end{aligned}
\end{equation*}
Similarly, we obtain that
$$\frac{t_\epsilon^2\epsilon^2}{2}\int_{\mathbb R^{N+1}_+}y^{1-2s}|\nabla \eta(\epsilon\bar{x},\epsilon y)|^2|V( \bar{x}, y)|^2d\bar{x}dy\to 0\quad{as~} \epsilon\to 0$$
and
$$\epsilon \int_{\mathbb R^{N+1}_+}y^{1-2s}\nabla \eta(\epsilon\bar{x},\epsilon y)\cdot \nabla V(\bar{x},y)  V( \bar{x}, y)\eta(\epsilon\bar{x},\epsilon y) d\bar{x}dy\to 0\quad{as~} \epsilon\to 0.$$
It further follows from the dominated convergence theorem  that
\begin{equation}\label{9}
\begin{aligned}
\limsup_{\epsilon\to 0}I_\epsilon(t_\epsilon\widetilde{U}_\epsilon,t_\epsilon\widetilde{V}_\epsilon)
=&\frac{t_0^2}{2}\int_{\mathbb R^{N+1}_+}y^{1-2s}|\nabla U|^2dxdy+
\frac{t_0^2}{2}\int_{\mathbb R^{N+1}_+}y^{1-2s}|\nabla V|^2dxdy\\
&+\frac{t_0^2a(x_0)}{2}\int_{\mathbb R^N} u^2d{x}
+\frac{t_0^2b(x_0)}{2}\int_{\mathbb R^N} v^2d{x}
-\frac{t_0^{p+1}}{p+1}\int_{\mathbb R^N}u^{p+1}d{x}\\
&-\frac{t_0^{2_s^*}}{2_s^*}\int_{\mathbb R^N}v^{2_s^*}dx
-\lambda t_0^2\int_{\mathbb R^N}uvd{x}\\
=&I_{x_0}(t_0U,t_0V)\\
\leq&c_{x_0}.
\end{aligned}
\end{equation}
As a consequence,
$$\limsup_{\epsilon\to0}c_\epsilon\leq \limsup_{\epsilon\to0}I_\epsilon(t_\epsilon\widetilde{U}_\epsilon,t_\epsilon\widetilde{V}_\epsilon)\leq c_{x_0}. $$

Now we proceed to show that system \eqref{mp} admits a positive vector ground state.
Owing to Lemma \ref{mmp}, Lemma \ref{pscondition} and the mountain pass theorem, we find that
$I_\epsilon$  admits a critical point $(U_\epsilon,V_\epsilon)\in X_{s,\epsilon}$ such that $I_\epsilon(U_\epsilon,V_\epsilon)=c_\epsilon.$ By virtue of Lemma \ref{eeq},  $(U_\epsilon,V_\epsilon)$ is a ground
state for $I_\epsilon.$ Next, we prove that $u_\epsilon>0$ and $v_\epsilon>0$ in $\mathbb R^N.$
Observe that $I_\epsilon^\prime (U_\epsilon,V_\epsilon)(-U_\epsilon^-,-V_\epsilon^-)=0$, then
\begin{equation*}
\begin{aligned}
&\int_{\mathbb R^{N+1}_+} y^{1-2s} |\nabla U^-_\epsilon|^2 dxdy+\int_{\mathbb R^{N+1}_+} y^{1-2s} |\nabla V^-_\epsilon|^2 dxdy+\int_{\mathbb R^N}a(\epsilon x)(u^-_\epsilon)^2 dx
 \int_{\mathbb R^N}b(\epsilon x)(v^-_\epsilon)^2 dx\\
  \leq&2\lambda \int_{\mathbb R^N}u^-_\epsilon v^-_\epsilon dx,
 \end{aligned}
\end{equation*}
which along with  $\lambda\leq(1-\delta_0) \sqrt{a_0b_0}$ and (P1) shows that
$||(U_\epsilon^-,V_\epsilon^-)||_{X_{s,\epsilon}}=0.$
As a result, $u_\epsilon\geq0,v_\epsilon\geq0$ in $\mathbb R^N.$ Assume there exists $\bar{x}\in\mathbb R^N$ such that
$u_\epsilon(\bar{x})=0,$ then
\begin{equation}\label{10}
(-\triangle)^su_{\epsilon}(\bar{x})=C_{N,s}PV\int_{\mathbb R^N}\frac{u_\epsilon(\bar{x})-u_\epsilon(y)}{|\bar{x}-y|^{N+2s}}dy\leq 0.
\end{equation}
On the other hand,
$$(-\triangle)^su_{\epsilon}(\bar{x})=\lambda v_{\epsilon}(\bar{x})\geq 0.$$
Therefore, $(-\triangle)^su_{\epsilon}(\bar{x})=0,$ and then $u_\epsilon=0$ in $\mathbb R^N$.
By \eqref{mp}, we further see $v_\epsilon=0$ in $\mathbb R^N.$ Thanks to the extension formula \eqref{56}, there holds $(U_\epsilon,V_\epsilon)=(0,0).$ This contradicts (I) in Lemma \ref{mmp}. The proof is complete.
\end{proof}

\section{The autonomous system}

In this section, we discuss  the ground state of  the linearly coupled autonomous system \eqref{autonomoussys}.
By the extension method, we  translate \eqref{autonomoussys}
  into
 \begin{equation}\label{eqautonomoussys}
 \begin{cases}
 -{\rm div} (y^{1-2s}\nabla U)=0\quad &\text{in} \ \mathbb R^{N+1}_+,\\
 -{\rm div} (y^{1-2s}\nabla V)=0& \text{in}\ \mathbb R^{N+1}_+,\\
 -\mathop{\lim}\limits_{y\to0 }y^{1-2s}\frac{\partial U(x,y)}{\partial y}+a(x_0)u=u^p+\lambda v &\text{on}\ \mathbb R^N,\\
 -\mathop{\lim}\limits_{y\to0 }y^{1-2s}\frac{\partial V(x,y)}{\partial y}+b(x_0)v=v^{2_s^*-1}+\lambda u\quad &\text{on}\ \mathbb R^N.
 \end{cases}
 \end{equation}
 The Euler-Lagrange functional associated to system \eqref{eqautonomoussys} is given by
 \begin{equation}\label{aueuler}
 \begin{aligned}
 J_{x_0}(U,V)=&\frac{1}{2}\int_{\mathbb R^{N+1}_+} y^{1-2s} |\nabla U|^2 dxdy+\frac{1}{2}\int_{\mathbb R^{N+1}_+} y^{1-2s} |\nabla V|^2 dxdy+\frac{1}{2}\int_{\mathbb R^N}a(x_0)u^2 dx \\
 &+\frac{1}{2}\int_{\mathbb R^N}b( x_0)v^2 dx+\frac{1}{p+1}\int_{\mathbb R^N}u^{p+1}dx +\frac{1}{2_s^*}
 \int_{\mathbb R^N}v^{2_s^*}dx +\lambda\int_{\mathbb R^N}uvdx.
 \end{aligned}
 \end{equation}
To consider system \eqref{eqautonomoussys}, we  define a Hilbert  space
$$X_{s,x_0}=\{(U,V)\in \chi_s\times\chi_s:||(U,V)||_{X_{s,x_0}}<\infty\},$$
 where
 $$||(U,V)||_{X_{s,x_0}}=\left(||U||_{a,x_0}^2+||V||_{b,x_0}^2\right)^{\frac{1}{2}},$$
 $$||U||_{a,x_0}^2 :=\int_{\mathbb R^{N+1}_+}y^{1-2s}|\nabla U|^2dxdy +\int_{\mathbb R^N}a( x_0)u^2dx,$$
 $$||V||_{b,x_0}^2 :=\int_{\mathbb R^{N+1}_+}y^{1-2s}|\nabla V|^2dxdy +\int_{\mathbb R^N}b(x_0)v^2dx.$$
 Since we are interested in the positive vector solution, we  first discuss the functional $I_{x_0}:X_{s,x_0}\to \mathbb R$
 defined as
 \begin{equation}\label{modaueuler}
 \begin{aligned}
 I_{x_0}(U,V)=&\frac{1}{2}\int_{\mathbb R^{N+1}_+} y^{1-2s} |\nabla U|^2 dxdy+\frac{1}{2}\int_{\mathbb R^{N+1}_+} y^{1-2s} |\nabla V|^2 dxdy+\frac{1}{2}\int_{\mathbb R^N}a(x_0)u^2 dx \\
 &+\frac{1}{2}\int_{\mathbb R^N}b( x_0)v^2 dx+\frac{1}{p+1}\int_{\mathbb R^N}(u^+)^{p+1}dx +\frac{1}{2_s^*}
 \int_{\mathbb R^N}(v^+)^{2_s^*}dx +\lambda\int_{\mathbb R^N}uvdx.
 \end{aligned}
 \end{equation}
 It is easy to see that if $I_{x_0}$ has a critical point $(U,V)$ satisfying  $u\geq0$
 and $v\geq0$ in $\mathbb R^N,$ then it is a critical point of $J_{x_0}$ as well.
 The Nehari manifold associated to $I_{x_0}$ is given by
 $$N_{x_0}:=\{(U,V)\in X_{s,x_0}:I^\prime_{x_0}(U,V)(U,V)=0\}.$$

 \begin{remark}\rm
 By a minor modification, we can prove that the results of Lemma \ref{mmp}, (I) and (III) in Lemma \ref{nehari} also hold for $I_{x_0}$ and $N_{x_0}.$ Moreover,
 For any $(U,V)\in N_{x_0},$  we have
$$|supp\ u^+|+|supp\ v^+|>0.$$
 \end{remark}
Now we define
$$X_{s,x_0}^+=\{(U,V)\in X_{s,x_0}:|supp\ u^+|+|supp\ v^+|>0\},$$
and let $S_{s,x_0}^+$ be the intersect of $X_{s,x_0}^+$ with the unit sphere $S_{s,x_0},$ namely,
$S_{s,x_0}^+=S_{s,x_0}\cap X_{s,x_0}^+.$
\begin{remark}\rm
The results in Lemmas \ref{Xepsilon}, \ref{eeq} and  \ref{sphere} hold as well if we replace the subscript $\epsilon$ by $x_0$.
\end{remark}

\begin{lemma}\label{liebtype}
If there exist a bounded sequence $\{U_n\}\subset \chi_s$  and $U\in\chi_s$ such that
$$ U_n\to U~~\text{a.e. in}~\mathbb R^{N+1}_+. $$
Then
$$\lim_{n\to\infty}\left(||U_n||^2_{\chi_s}-||U_n-U||^2_{\chi_s}\right)=||U||^2_{\chi_s}.$$
\end{lemma}

\begin{proof}
In fact, we only need to prove
$$\lim_{n\to\infty}\int_{\mathbb{R}^{N+1}_+}y^{1-2s}\left||\nabla U_n|^2-|\nabla U_n-\nabla U|^2-|\nabla U|^2\right|dxdy=0.$$
Note that
\begin{equation*}
\begin{aligned}
\left||\nabla U_n|^2-|\nabla U_n-\nabla U|^2\right|
=&\left||(\nabla U_n-\nabla U)+\nabla U|^2-|\nabla U_n-\nabla U|^2\right|\\
\leq&\max\left\{\left|\left(|\nabla U_n-\nabla U|+|\nabla U|\right)^2-|\nabla U_n-\nabla U|^2\right|, \right .\\
&\left .~~~~~~~ \left|\left(|\nabla U_n-\nabla U|-|\nabla U|\right)^2-|\nabla U_n-\nabla U|^2\right|\right\}.
\end{aligned}
\end{equation*}
Since for any $\epsilon>0$, there is $C(\epsilon)>0$ such that
$$\left||a+b|^2-|a|^2\right|\leq \epsilon |a|^2+C(\epsilon)|b|^2\quad \text{for any }a,b\in\mathbb R.$$
It then follows that
\begin{equation*}
\begin{aligned}
\left||\nabla U_n|^2-|\nabla U_n-\nabla U|^2\right|\leq \epsilon |\nabla U_n-\nabla U|^2+C(\epsilon)|\nabla U|^2.
\end{aligned}
\end{equation*}
As a consequence, we have
$$h^\epsilon_n:=\left(\left||\nabla U_n|^2-|\nabla U_n-\nabla U|^2-|\nabla U|^2\right|-\epsilon |\nabla U_n-\nabla U|^2\right)^+\leq (C(\epsilon)+1)
|\nabla U|^2.$$
The dominated convergence theorem then implies that
\begin{equation}\label{13}
\lim_{n\to \infty}\int_{\mathbb{R}^{N+1}_+}y^{1-2s} h^\epsilon_ndxdy=0.
\end{equation}
Since $\{U_n\}\subset \chi_s$ is bounded, we conclude that
$$\epsilon\int_{\mathbb R^{N+1}_+}y^{1-2s}|\nabla U_n-\nabla U|^2dxdy\leq C\epsilon,$$
which along with \eqref{13} and the definition of $h^\epsilon_n$   implies that
\begin{equation*}
\begin{aligned}
\lim_{n\to\infty}\int_{\mathbb{R}^{N+1}_+}y^{1-2s}\left||\nabla U_n|^2-|\nabla U_n-\nabla U|^2-|\nabla U|^2\right|dxdy
\leq&\lim_{n\to \infty}\int_{\mathbb{R}^{N+1}_+}y^{1-2s} h^\epsilon_ndxdy+C\epsilon
=C\epsilon.
\end{aligned}
\end{equation*}
Letting $\epsilon\to0$, we get the result. The proof is complete.
\end{proof}

We now state a useful lemma which is a direct result of   \cite[Lemma 3.3]{he-zou} as follows.
\begin{lemma}\label{concnetration2}
Let $\{(U_n,V_n)\}\subset \chi_{s}\times \chi_{s}$ be bounded and satisfy
$$\lim_{n\to\infty}\sup_{z\in\mathbb R^N}\int_{B_r(z)}u_n^2+v_n^2dx=0\quad\text{for some~} r>0.$$
Then
$$(u_n,v_n)\to (0,0)\quad \text{in}~ L^s(\mathbb R^N)\times L^t(\mathbb R^N),~\forall s,t\in (2,2_s^*).$$
\end{lemma}
We define a constant
\begin{equation}\label{sharpconstant}
C_{p+1}:=\inf_{U\in\mathcal E}\frac{\int_{\mathbb R^{N+1}_+}y^{1-2s}|\nabla U|^2dxdy+\int_{\mathbb R^N}u^2dx}
{\left(\int_{\mathbb R^N}|u|^{p+1}dx\right)^{\frac{2}{p+1}}},
\end{equation}
where the Hilbert space $\mathcal E$ is defined in \eqref{56}.
\begin{lemma}\label{embedding} If $p\in [1,2_s^*-1),$ then
there exists a nonnegative function  $U\in\mathcal E$ such that the constant $C_{p+1}$ defined by \eqref{sharpconstant} is achieved at $U$.
\end{lemma}
\begin{proof}
Noting that
$$C_{p+1}=\inf_{\substack{U\in\mathcal E\\||u||_{p+1}=1}}\left(\int_{\mathbb R^{N+1}_+}y^{1-2s}|\nabla U|^2dxdy+\int_{\mathbb R^N}u^2dx\right).$$
First, it follows from the interpolation theorem that, for any $U\in\mathcal E$ with $||u||_{p+1}=1$,
$$1=||u||_{p+1}^{p+1}=||u||_{2_s^*}^{\frac{N(p-1)}{2s}}||u||_{2}^{p+1-\frac{N(p-1)}{2s}}\leq C||U||_{\mathcal E}^{\frac{N(p-1)}{2s}}||u||_{2}^{p+1-\frac{N(p-1)}{2s}}
\leq C||U||_{\mathcal E}^{p+1}, $$
where $C$ is independent of $U.$ Hence $C_{p+1}\geq\frac{1}{C}>0.$
Let $\{U_n\}\subset \mathcal E$ with $||u_n||_{p+1}=1$ be such that
$$\int_{\mathbb R^{N+1}_+}y^{1-2s}|\nabla U_n|^2dxdy+\int_{\mathbb R^N}u_n^2dx\to C_{p+1}\quad \text{as~} n\to \infty.$$
In view of Lemma \ref{concnetration2}, there exist $\{z_n\}\subset\mathbb R^N$ and $\theta>0$ such that
$$\liminf_{n\to\infty}\int_{B_1(z_n)}u_n^2dx\geq\theta>0.$$
Now, we define
$$\widetilde{U}_n(x,y)=U_n(x+z_n,y)\quad\text{for any~} (x,y)\in\mathbb R^{N+1}_+.$$
Then $||\tilde{u}_n||_{p+1}=1$ and
$$\lim_{n\to\infty}||\widetilde{U}_n||_{\mathcal E}=C_{p+1}, \quad\liminf_{n\to\infty}\int_{B_1}\tilde{u}_n^2dx\geq \theta>0.$$
Therefore, there exists a subsequence (denoted still by $\widetilde{U}_n$) and $U\in\mathcal E$ such that
$$\widetilde{U}_n\rightharpoonup U\quad \text{in} ~\mathcal E,$$
$$\widetilde{U}_n\to U\quad \text{a.e. in}~ \mathbb R^{N+1}_+,$$
$$\tilde{u}_n\to u\quad \text{in}~L_{\text{loc}}^q(\mathbb R^N),~~\forall q\in [2,2_s^*). $$
As a consequence, we further have
\begin{equation}\label{14}
\int_{B_1}{u}^2dx\geq \theta>0,
\end{equation}
\begin{equation}\label{15}
||u||_{p+1}\leq 1,\quad ||U||_{\mathcal E }\leq \liminf_{n\to\infty} ||\widetilde{U}_n||_{\mathcal E }=C_{p+1}.
\end{equation}
We now claim that $||u||_{p+1}=1.$ Define $W_n=\widetilde{U}_n-U$, then by Lemma \ref{liebtype} and \eqref{sharpconstant},
\begin{equation*}
\begin{aligned}
C_{p+1}=\lim_{n\to\infty}||\widetilde{U}_n||_{\mathcal E}^2
&=\lim_{n\to\infty}||W_n-U||_{\mathcal E}^2+||U||_{\mathcal E}^2\\
&\geq C_{p+1}\lim_{n\to\infty} \left(\left(||\widetilde u_n-u||_{p+1}^{p+1}\right)^{\frac{2}{p+1}}+\left(||u||_{p+1}^{p+1}\right)^{\frac{2}{p+1}}\right) \\
&= C_{p+1}\lim_{n\to\infty}\left(\left(||\widetilde u_n||_{p+1}^{p+1}-||u||_{p+1}^{p+1}\right)^{\frac{2}{p+1}}+\left(||u||_{p+1}^{p+1}\right)^{\frac{2}{p+1}}\right)\\
&=C_{p+1}\left(\left(1-||u||_{p+1}^{p+1}\right)^{\frac{2}{p+1}}+\left(||u||_{p+1}^{p+1}\right)^{\frac{2}{p+1}}\right).
\end{aligned}
\end{equation*}
Then $||u||_{p+1}^{p+1}=1$ in terms of \eqref{14}, and furthermore $||U||_{\mathcal E}=C_{p+1}$ due to \eqref{15}.
Since $\big|\big||U|\big|\big|_{\mathcal E}=\big|\big|U\big|\big|_{\mathcal E}$ and
$\big|\big||u|\big|\big|_{p+1}=\big|\big|u\big|\big|_{p+1}=1,$ we can assume that $U$ is nonnegative.
Define
\begin{equation}\label{hatu}
\widehat{U}(x,y)=C_{p+1}^{\frac{1}{p-1}}U(x,y)\quad \text{in} ~\mathbb R^{N+1}_+.
\end{equation}
Then $||\widehat{U}||_{\mathcal E}^2=C_{p+1}||\hat{u}||_{p+1}^2.$
Moreover, $\widehat{U}$ is a ground state of
\begin{equation*}
\begin{cases}
-{\rm div}  (y^{1-2s}\nabla\widehat U)=0\quad& \text{in}~\mathbb R^{N+1}_+,\\
-\mathop{\lim}\limits_{y\to0}y^{1-2s} \frac{\partial \widehat U(x,y)}{\partial y} +\widehat u= \widehat u^p&\text{in}~\mathbb R^{N}.
\end{cases}
\end{equation*}
Hence, $||\widehat{U}||_{\mathcal E}^2=||\hat{u}||_{p+1}^{p+1}=C_{p+1}^{\frac{p+1}{p-1}}.$ As a result, there hold
\begin{equation}\label{embedengegy}
 \widehat{I}(\widehat{U}):=\frac{1}{2}||\widehat{U}||_{\mathcal E}^2-\frac{1}{p+1}||\hat{u}||_{p+1}^{p+1}
 =\left(\frac{1}{2}-\frac{1}{p+1}\right)C_{p+1}^{\frac{p+1}{p-1}}.
 \end{equation}
\end{proof}
\begin{thm}\label{augroundvalue}
Under assumptions $(P1)$, $(P2)$ and $0<\lambda<\min\{\sqrt{a_0b_0}$, $\sqrt{(a_1-\mu_0)b_0}\},$
\begin{itemize}
\item[\rm (I)] if $x_0\in \partial\Lambda$ then $c_{x_0}=\frac{s}{N}S^{\frac{N}{2s}}.$
\item[\rm (II)] if $a(x_0)\leq\mu_0$  for some $x_0\in \mathbb R^N,$ then $c_{x_0}<\frac{s}{N}S^{\frac{N}{2s}}.$
\end{itemize}
\end{thm}
\begin{proof} Inspired by  some ideas in \cite{chen-zou2012}, we fist show $c_{x_0}\geq\frac{s}{N}S^{\frac{N}{2s}}$ for any $x_0\in\partial\Lambda.$
 In fact, for any  $(U,V)\in X^+_{s,x_0},$ it follows from Young inequality and (P1) that
\begin{equation*}
\begin{aligned}
I_{x_0}(tU,tV)\geq&\frac{t^2}{2}\int_{\mathbb R^{N+1}_+}y^{1-2s}|\nabla U|^2dxdy+\frac{t^2}{2}\int_{\mathbb R^{N}}
\left(a(x_0)-\frac{\lambda^2}{b(x_0)}\right)u^2dx-\frac{t^{p+1}}{p+1}\int_{\mathbb R^{N}}(u^+)^{p+1}dx\\
&+\frac{t^2}{2}\int_{\mathbb R^{N+1}_+}y^{1-2s}|\nabla V|^2dxdy
-\frac{t^{2_s^*}}{2_s^*}\int_{\mathbb R^{N}}(v^+)^{2_s^*}dx\\
&:=g(t,U)+h(t,V).
\end{aligned}
\end{equation*}
We consider the following three cases.

\textbf{Case 1.} if $u^+=0$ in $\mathbb R^N,$ then
\begin{equation*}
\inf_{\substack{(U,V)\in X_{s,x_0}^+\\ u^+=0}}\max_{t>0}I_{x_0}(tU,tV)\geq \inf_{V\in \chi_s}\max_{t>0}h(t,V)=h(1,W_\epsilon)\geq \frac{s}{N}S^{\frac{N}{2s}},
\end{equation*}
where $W_\epsilon$ is defined in (III) of Lemma \ref{59}.

\textbf{Case 2.} if $v^+=0$ in $\mathbb R^N,$ then
\begin{equation*}
\begin{aligned}
\inf_{\substack{(U,V)\in X_{s,x_0}^+\\ v^+=0}}\max_{t>0}I_{x_0}(tU,tV)\geq \inf_{U\in \mathcal E}\max_{t>0}g(t,U).
\end{aligned}
\end{equation*}
Let $\beta^{2s}=a(x_0)-\frac{\lambda^2}{b(x_0)}$, $\gamma=\left(a(x_0)-\frac{\lambda^2}{b(x_0)}\right)^{\frac{1}{p-1}}$
and $\widetilde{U}(x,y)=\gamma \widehat{U}(\beta x,\beta y),$ where $\widehat{U}$ is defined by \eqref{hatu}.
Then  by a direct calculation, we find  $\widetilde{U}$ is a ground state of
\begin{equation*}
\begin{cases}
-{\rm div}  (y^{1-2s}\nabla\widetilde{U})=0\quad& \text{in}~\mathbb R^{N+1}_+,\\
-\mathop{\lim}\limits_{y\to0}y^{1-2s} \frac{\partial \widetilde{U}(x,y)}{\partial y} +\left(a(x_0)-\frac{\lambda^2}{b(x_0)}\right)\tilde{u}= \tilde{u}^p&\text{on}~\mathbb R^{N}.
\end{cases}
\end{equation*}
As a result, there hold
\begin{equation*}
\begin{aligned}
&\inf_{U\in \mathcal E}\max_{t>0}g(t,U)=g(1,\widetilde{U})\\
&=\frac{1}{2}\int_{\mathbb R^{N+1}_+}y^{1-2s}|\nabla \widetilde{U}|^2dxdy+\frac{1}{2}\int_{\mathbb R^{N}}\left(a(x_0)-\frac{\lambda^2}{b(x_0)}\right)\tilde{u}^2dx
-\frac{1}{p+1}\int_{\mathbb R^{N}}(\tilde{u}^+)^{p+1}dx\\
&=\left(\frac{1}{2}-\frac{1}{p+1}\right)\left(a(x_0)-\frac{\lambda^2}{b(x_0)}\right)^{\left(\frac{p+1}{p-1}
-\frac{N}{2s}\right)}||\hat{u}||_{p+1}^{p+1}\\
&=\left(\frac{1}{2}-\frac{1}{p+1}\right)\left(a(x_0)-\frac{\lambda^2}{b(x_0)}\right)^{\left(\frac{p+1}{P-1}
-\frac{N}{2s}\right)}C_{p+1}^{\frac{p+1}{p-1}}\\
&\geq \frac{s}{N}S^{\frac{N}{2s}}.
\end{aligned}
\end{equation*}
Therefore,
$$\inf_{\substack{(U,V)\in X_{s,x_0}^+\\ v^+=0}}\max_{t>0}I_{x_0}(tU,tV)\geq\frac{s}{N}S^{\frac{N}{2s}}. $$
\textbf{Case 3.} If $u^+\neq 0, v^+\neq 0$ in $\mathbb R^N,$ then  there exist
$t_1>0$ and $t_2>0$ such that
$$\max_{t>0}g(t,U)=g(t_1,U),\quad \max_{t>0}h(t,V)=h(t_2,V).$$
Moreover, if $t_1\geq t_2$ then  $g(t_2,U)>0.$ As a consequence,
\begin{equation*}
\max_{t>0}I_{x_0}(tU,tV)\geq h(t_2,V)\geq\inf_{V\in \chi_s}\max_{t>0}h(t,V) =h(1,W_\epsilon)\geq \frac{s}{N}S^{\frac{N}{2s}}.
\end{equation*}
Similarly, if $t_1< t_2,$ then
$$\max_{t>0}I_{x_0}(tU,tV)\geq g(t_1,U)\geq\inf_{U\in \mathcal E}\max_{t>0}g(t,U) =g(1,\widetilde{U})\geq \frac{s}{N}S^{\frac{N}{2s}}.$$
Summarizing the above three cases, we conclude that $c_{x_0}\geq\frac{s}{N} S^{\frac{N}{2s}}$ for any $x_0\in \partial \Lambda.$

Next, we prove that
 $c_{x_0}\leq\frac{s}{N} S^{\frac{N}{2s}}$ for any $x_0\in \mathbb R^N.$
 Let $\phi_0\in C^{\infty}([0,\infty))$ be a cut-off function such that $\phi_0(t)=1$ if $t\in [0,1/2]$  and
 $\phi_0(t)=1$ if $t>1.$  Define $$\phi(x,y)=\phi_0\left(\frac{|(x,y)|}{r}\right),\quad
 \Psi_\epsilon(x,y)=\frac{\phi(x,y) W_\epsilon(x,y)}{||\phi(\cdot,0)W_\epsilon(\cdot,0)||_{2_s^*}},$$
 and $\psi_\epsilon$ as the trace of $\Psi_\epsilon,$   where $W_\epsilon$ is given by (III) in Lemma \ref{59}.Then we know from the argument of \cite{do-Miyagaki-Squassina} that
 $$||\Psi_\epsilon||_{\chi_s}^2=S+O(\epsilon^{N-2s})$$
 and
 \begin{equation*}
 ||\psi_\epsilon||_2^2=h(\epsilon)=
 \begin{cases}
 o(\epsilon^{2s})\quad& \text{if~} N>4s,\\
 o(\epsilon^{2s}\ln\frac{1}{\epsilon})&\text{if}~N=4s,\\
 o(\epsilon^{N-2s}) &\text{if}~ N<4s.
 \end{cases}
 \end{equation*}
 As a consequence,
 $$t_\epsilon:=\frac{||\Psi_\epsilon||_{\chi_s}^2+b(x_0)||\psi_\epsilon||_2^2}{||\psi_\epsilon||_{2_s^*}^{2_s^*}}
=S+o(\epsilon^{N-2s})+b(x_0)h(\epsilon), $$
that is, $t_\epsilon$ is bounded. Therefore, there exist a subsequence (denoted still by $t_\epsilon$)
and $t_0>0$ such that $t_\epsilon\to t_0$ as $\epsilon\to 0.$ Direct calculations yields that
\begin{equation*}
\begin{aligned}
\max_{t>0}I_{x_0} (0,t\Psi_\epsilon)&=I_{x_0} (0,t_\epsilon \Psi_\epsilon)\\
&=\frac{t_\epsilon^2}{2}\left(S+o(\epsilon^{N-2s})+b(x_0)h(\epsilon)\right)-
\frac{t_\epsilon^{2_s^*}}{2_s^*}\\
&\to \frac{t_0^2}{2}S-\frac{t_0^{2_s^*}}{2_s^*}\\
&\leq \max_{t>0}\left(\frac{t^2}{2}S-\frac{t^{2_s^*}}{2_s^*}\right)\\
&=\frac{s}{N}S^{\frac{N}{2s}}\end{aligned}
\end{equation*}
 As a result,
 \begin{equation}\label{17}
  c_{x_0}\leq \mathop{\limsup}\limits_{\epsilon\to 0}\mathop{\max}\limits_{t>0}I_{x_0}(0,t\Psi_\epsilon)\leq\frac{s}{N}S^{\frac{N}{2s}}\quad \text{for any~} x_0\in\mathbb R^N.
  \end{equation}
 Consequently, (I) holds.

We next prove (II). If $a(x_0)\leq\mu_0$ for some  $x_0\in \mathbb R^N,$ define $\alpha =(a(x_0))^{\frac{1}{p-1}},$ $\beta=(a(x_0))^{\frac{1}{2s}}$,
 \begin{equation}\label{hatw}
 \widehat{W}(x,y)=\alpha\widehat U(\beta x, \beta y)\quad \text{for any~}(x,y)\in\mathbb R^{N+1}_+,
 \end{equation}
 where $\widehat{U}$ is given by \eqref{hatu}. Then $\widehat{W}$ is a ground state of
 \begin{equation*}
\begin{cases}
-{\rm div}  (y^{1-2s}\nabla\widehat{W})=0\quad& \text{in}~\mathbb R^{N+1}_+,\\
-\mathop{\lim}\limits_{y\to0}y^{1-2s} \frac{\partial \widehat{W}(x,y)}{\partial y} +a(x_0)\widehat{w}= \widehat{w}^p&\text{on}~\mathbb R^{N}.
\end{cases}
\end{equation*}
 Hence, there hold
 \begin{equation}\label{18}
 \begin{aligned}
 \inf_{U\in\mathcal E}\max_{t>0}I_{x_0}(tU,0)&=I_{x_0}(\widehat{W},0)\\
 &=\left(\frac{1}{2}-\frac{1}{p+1}\right)(a(x_0))^{\frac{p+1}{p-1}-\frac{N}{2s}}||\hat{u}||_{p+1}^{p+1}\\
 &=\left(\frac{1}{2}-\frac{1}{p+1}\right)(a(x_0))^{\frac{p+1}{p-1}-\frac{N}{2s}}C_{p+1}^{\frac{p+1}{p-1}}.
 \end{aligned}
 \end{equation}
If $a(x_0)<\mu_0$, then it follows from $p<2_s^*,$  (P1)-(P2) and the definition of $\mu_0$ that
\begin{equation}\label{16}
c_{x_0}\leq\inf_{U\in\mathcal E}\max_{t>0}I_{x_0}(tU,0)=\left(\frac{1}{2}-\frac{1}{p+1}\right)(a(x_0)
)^{\frac{p+1}{p-1}-\frac{N}{2s}}C_{p+1}^{\frac{p+1}{p-1}}<\frac{s}{N}S^{\frac{N}{2s}}.
\end{equation}
If $a(x_0)=\mu_0,$ we claim that $c_{x_0}<\frac{s}{N}S^{\frac{N}{2s}}.$ Otherwise, we have
$$c_{x_0}=\frac{s}{N}S^{\frac{N}{2s}}=I_{x_0}(\widehat{W},0)$$
due to \eqref{17} and \eqref{18}, where $\widehat{W}$ is defined in \eqref{hatw}.
Hence $(\widehat{W},0)$  is a ground state of \eqref{eqautonomoussys}, which is impossible.
The proof is complete.
\end{proof}

\begin{lemma}\label{psnonnegative}
Assume $(P1)$ and $\lambda<\sqrt{a_0b_0}.$
If $\{(U_n,V_n)\}\subset X_{s,x_0}$ is a  $(PS)_{d}$ sequence of $I_{x_0}$ with $0<d<\frac{s}{N}S^{\frac{N}{2s}},$
then there exists $\theta>0,$ $r>0$ and $\{z_n\}\subset\mathbb R^N$ such that
$$\liminf_{n\to \infty}\int_{B_r(z_n)}u_n^2+v_n^2dx\geq\theta>0.$$
\end{lemma}
\begin{proof}
Assume to the contrary that
$$\lim_{n\to \infty}\sup_{z\in\mathbb R^N}\int_{B_r(z)}u_n^2+v_n^2dx=0,$$
then Lemma \ref{concnetration2} implies that
$$(u_n,v_n)\to (0,0)\quad \text{in}~ L^s(\mathbb R^n)\times L^t(\mathbb R^n),~\forall s,t\in (2,2_s^*).$$
Moreover, we can cover $\mathbb R^N$ by a countable number of balls $\{B_r(x_j)\}_{j\in J}$ in a way that each point of $\mathbb R^N$ is contained
in at most $N+1$ balls, hence
$$\int_{\mathbb R^N} |u_nv_n|dx\leq \sum_{j\in J}\int_{B_r(x_j)}|u_nv_n|dx\leq (N+1)\sup_{z\in\mathbb R^N}\left(\int_{B_{r}(z)}u_n^2dx\right)^{\frac{1}{2}}
\left(\int_{\mathbb R^N}v_n^2dx\right)^{\frac{1}{2}}\to 0$$ as $n\to\infty.$
It follows from boundedness of $\{(U_n,V_n)\}$ and $I_{x_0}^\prime(U_n,V_n)\to 0$ as $n\to \infty $ that
\begin{equation}\label{11}
\int_{\mathbb R^{N+1}_+} y^{1-2s} |\nabla U_n|^2 dxdy+\int_{\mathbb R^N}a(x_0)u_n^2 dx=\int_{\mathbb R^N}u_n^{p+1}dx+\lambda
\int_{\mathbb R^N}u_nv_ndx +o_n(1)\to 0 ~\text{as~} n\to \infty.
\end{equation}
Assume that
$$\lim_{n\to\infty}\int_{\mathbb R^{N+1}_+} y^{1-2s} |\nabla V_n|^2 dxdy+\int_{\mathbb R^N}b(x_0)v_n^2 dx=l.$$
Then we obtain by $I_{x_0}^\prime(U_n,V_n)(0,V_n)\to 0$ as $n\to \infty $ that
$$\lim_{n\to\infty}\int_{\mathbb R^N}|v_n^+|^{2_s^*}dx=l.$$
Thanks to Lemma \ref{59}, there is
$l\geq Sl^{\frac{2}{2_s^*}}.$ Namely, either $l=0$ or $l\geq S^{\frac{N}{2s}}.$
If the later one holds, then
\begin{equation*}
\begin{aligned}
\frac{s}{N}S^{\frac{N}{2s}}\geq d+o_n(1)&=I_{x_0}(U_n,V_n)-\frac{1}{2}I^\prime_{x_0}(U_n,V_n)(U_n,V_n)\\
&=\left(\frac{1}{2}-\frac{1}{p+1}\right)\int_{\mathbb R^N}|u_n^+|^{p+1}dx+\left(\frac{1}{2}-\frac{1}{2_s^*}\right)\int_{\mathbb R^N}|v_n^+|^{2_s^*}dx\\
&\to \frac{s}{N}l
>\frac{s}{N}S^{\frac{N}{2s}}.
\end{aligned}
\end{equation*}
This contradiction  confirms that $l=0.$ Namely,
$$\lim_{n\to\infty}\int_{\mathbb R^{N+1}_+} y^{1-2s} |\nabla V_n|^2 dxdy+\int_{\mathbb R^N}b(x_0)v_n^2 dx=0,$$
which along with \eqref{11} further implies that
$(U_n,V_n)\to (0,0)$ as $n\to\infty.$
This contradicts $I_{x_0}(U_n,V_n)\to d>0.$ The proof is complete.
\end{proof}

\begin{thm}\label{augroundexistence}
Assume $(P1)$. If $\lambda<\sqrt{a_0b_0},$ then for any $x_0\in \mathbb R^N$ with $0<c_{x_0}<\frac{s}{N}S^{{\frac{N}{2s}}},$
$I_{x_0}$ admits a positive vector critical point $(U,V)\in X_{s,x_0}$  such that
$I_{x_0}(U,V)=c_{x_0}.$
\end{thm}
\begin{proof}
By a minor modification to the proof of Lemma \ref{mmp}, we can prove that $I_{x_0}$ satisfies the mountain pass geometry.
As a consequence,  there exists $\{(U_n,V_n)\}\subset X_{s,x_0}$
such that
$$I_{x_0}(U_n,V_n)\to c_{x_0}, ~~I^\prime_{x_0}(U_n,V_n)\to 0\quad \text{as~} n\to \infty.$$
By virtue of Lemma \ref{psnonnegative}, there exist
$\theta>0,$ $r>0$ and $\{z_n\}\subset\mathbb R^N$ such that
$$\int_{B_r(z_n)}u_n^2+v_n^2dx\geq\theta>0.$$
Now define
$$\left(\widetilde{U}_n(x,y),\widetilde{V}_n(x,y)\right)=\left(U_n(x+z_n,y),V_n(x+z_n,y)\right).$$
Then $\{(\widetilde{U}_n,\widetilde{V}_n)\}$ is a $(PS)_{c_{x_0}}$ sequence for $I_{x_0}$ as well, and
\begin{equation}\label{12}
\liminf_{n\to\infty}\int_{B_r(0)}\tilde{u}_n^2+\tilde{v}_n^2dx \geq\theta>0.
\end{equation}
It then follows from the boundedness of  $(PS)_{c_{x_0}}$ sequence that there  exist a subsequence denoted still by
$\{(\widetilde{U}_n,\widetilde{V}_n)\}$ and $(U,V)\in X_{s,x_0}$ such that
$$(\widetilde{U}_n,\widetilde{V}_n)\rightharpoonup(U,V) \quad\text{in~} X_{s,x_0},$$
$$(\widetilde{U}_n,\widetilde{V}_n)\to(U,V) \quad\text{a.e. in~} \mathbb R^{N+1}_+,$$
$$(\widetilde{u}_n,\widetilde{v}_n)\to (u,v)\quad\text{in~}L^s_{loc}(\mathbb R^N)\times L^t_{loc}(\mathbb R^N),\forall s,t\in [2,2_s^*).$$
By a direct calculation, we infer that $(U,V)$ is a critical point of $I_{x_0}.$ As a consequence,
\begin{equation}\label{66}
\begin{aligned}
c_{x_0}&\leq I_{x_0}(U,V)
=I_{x_0}(U,V)-\frac{1}{p+1}I^\prime_{x_0}(U,V)(U,V)\\
&=\left(\frac{1}{2}-\frac{1}{p+1}\right)\left(\int_{\mathbb R^{N+1}_+} y^{1-2s} |\nabla U|^2 dxdy+\int_{\mathbb R^{N+1}_+} y^{1-2s} |\nabla V|^2 dxdy+\int_{\mathbb R^N}a(x_0)u^2 dx \right .\\
 &\left .~~~+\int_{\mathbb R^N}b( x_0)v^2 dx-\lambda\int_{\mathbb R^N}uvdx\right) +\left(\frac{1}{p+1}-\frac{1}{2_s^*}\right)
 \int_{\mathbb R^N}(v^+)^{2_s^*}dx \\
 &\leq\liminf_{n\to\infty} \left[\left(\frac{1}{2}-\frac{1}{p+1}\right)\left(\int_{\mathbb R^{N+1}_+} y^{1-2s} |\nabla \widetilde{U}_n|^2 dxdy+\int_{\mathbb R^{N+1}_+} y^{1-2s} |\nabla \widetilde{V}_n|^2 dxdy\right .\right .\\
 &\left .\left .~~~+\int_{\mathbb R^N}a(x_0)\tilde{u}_n^2 dx +\int_{\mathbb R^N}b( x_0)\tilde{v}_n^2 dx-\lambda\int_{\mathbb R^N}\tilde{u}_n\tilde{v}_ndx\right) +\left(\frac{1}{p+1}-\frac{1}{2_s^*}\right)
 \int_{\mathbb R^N}(\tilde{v}_n^+)^{2_s^*}dx\right]\\
 &=\liminf_{n\to \infty}\left(I_{x_0}(\widetilde{U}_n,\widetilde{V}_n)-\frac{1}{p+1}I^\prime_{x_0}(\widetilde{U}_n,\widetilde{V}_n)(\widetilde{U}_n,
 \widetilde{V}_n)\right)\\
 &= c_{x_0},
 \end{aligned}
\end{equation}
where the second inequality is obtained by $\lambda<\sqrt{a_0b_0}$ and Fatou's lemma. Hence $(U,V)$ is a ground state of
$I_{x_0}.$ By a similar argument to the proof of Theorem \ref{energybound} and \eqref{12} , we conclude $u>0$ and $v>0$ in $\mathbb R^N$.
The proof is complete.
\end{proof}
\begin{lemma}
Under assumptions $(P1)$ and $(P2),$ if $\lambda<\min\{\sqrt{a_0b_0},\sqrt{(a_1-\mu_0)b_0}\},$ then
 $$\mathcal O:=\left\{x\in\Lambda: c_x=c_0:=\mathop{\inf}\limits_{y\in\Lambda}c_y\right\}\neq\emptyset, \quad \mathcal O\subset\subset\Lambda.$$
 Moreover, $c_x<\frac{s}{N}S^{\frac{N}{2s}}$ for any $x\in\mathcal O.$
\end{lemma}
\begin{proof}
It follows from Theorem \ref{augroundvalue} and (P2) that $\mathcal O\neq \emptyset$ and $c_x<\frac{s}{N}S^{\frac{N}{2s}}$ for any $x\in\mathcal O.$ To complete the proof,
 we only need to verify that the function $l:\overline{\Lambda}\to \mathbb R$ defined as $l(x)=c_x$ is continuous.
If $x_n\to x_0$ with $c_{x_0}<\frac{s}{N} S^{\frac{N}{2s}}$ as $n\to\infty,$ then it follows from Theorems \ref{augroundexistence}  that there exist   $(U_0,V_0)\in N_{x_0}$ and $t_n>0$  such that
 $I_{x_0}(U_0,V_0)=c_{x_0}$ and $(t_nU_0,t_nV_0)\in N_{x_n}.$ By a similar argument to the proof of (II)
in Lemma \ref{Xepsilon}, we see that $t_n$ is bounded. So we can assume that $t_n\to t_0>0$ as $n\to\infty.$ As a consequence,
\begin{equation}\label{19}
\lim_{n\to\infty}c_{x_n}\leq\lim_{n\to\infty}I_{x_n}(t_nU_0,t_nV_0)=I_{x_0}(t_0U_0.t_0V_0)\leq I_{x_0}(U_0,V_0)=c_{x_0}.
\end{equation}
Hence, we can assume $c_{x_n}<\frac{s}{N}S^{\frac{N}{2s}}$ for large $n.$ It follows from Theorem \ref{augroundexistence}
that $I_{x_n}$ has a positive vector ground state $(U_n,V_n)$ for large $n.$
By a similar argument to the proof of (I) in lemma \ref{psbound}, $\{(U_n,V_n)\}$
is bounded in $X_{s,x_0}.$ As a result, there exist a subsequence denoted still by $\{(U_n,V_n)\}$  and $(U,V)\in X_{s,x_0}$ such that
$$(U_n,V_n)\rightharpoonup(U,V)\quad\text{as~} n\to \infty,$$
and  $(U,V)$ is a critical of $I_{x_0}.$ Moreover, Similar to \eqref{66}, we see
\begin{equation}\label{69}
\begin{aligned}
c_{x_0}&\leq I_{x_0}(U,V)-\frac{1}{p+1}I_{x_0}^\prime(U,V)(U,V)\\&\leq\lim_{n\to\infty}\left(I_{x_n}(U_n,V_n)
-\frac{1}{p+1}I_{x_n}^\prime(U_n,V_n)(U_n,V_n)\right)=\lim_{n\to\infty}c_{x_n}.
\end{aligned}
\end{equation}
This together with \eqref{19} implies that $\mathop{\lim}\limits_{n\to\infty}l(x_n)=l(x_0)$ for $x_0$ with $c_{x_0}<\frac{s}{N}S^{\frac{N}{2s}}.$
Assume $x_n\to x_0$ with $c_{x_0}=\frac{s}{N}S^{\frac{N}{2s}},$
we only need to prove that the limit holds for $\{x_n\}$ with $c_{x_n}<\frac{s}{N}S^{\frac{N}{2s}}.$
In fact, we conclude by a similar argument to \eqref{69} that
$$\frac{s}{N}S^{\frac{N}{2s}}=c_{x_0}\leq\lim_{n\to\infty}c_{x_n}\leq\frac{s}{N}S^{\frac{N}{2s}}.$$
 The proof is complete.
\end{proof}
It follows from a similar argument to \eqref{66} that the (PS)$_{c_{x_0}}$ sequence for $I_{x_0}$ with $x_0\in\mathcal O$ has a convergent subsequence. Furthermore, we can investigate the compactness for a sequence in $N_{x_0},$ which plays a crucial role in the discussion of multiplicity.

\begin{lemma}\label{42}
Under assumptions $(P1)$, $(P2)$ and $0<\lambda<\min\{\sqrt{a_0b_0}$, $\sqrt{(a_1-\mu_0)b_0}\},$
 if $\{(U_n,V_n)\}\subset N_{x_0}$ be such that $I_{x_0}(U_n,V_n)\to c_{x_0}$ for some $x_0\in \mathcal O,$  then there exists a convergent subsequence.
\end{lemma}
\begin{proof}
We divide the proof into two steps.

{\bf Step 1.} We claim that if $\{(U_n,V_n)\}\subset S_{s,x_0}^+$ satisfies dist$((U_n,V_n),\partial S_{s,x_0}^+)\to 0$ as
$n\to\infty,$ then $||\overline{\tau}_{x_0}(U_n,V_n)||\to\infty$ and $I_{x_0}(\overline{\tau}_{x_0}(U_n,V_n))\to \infty.$
 In fact, for any $(U,V)\in\partial S_{s,x_0}^+ $, we have $u(x)\leq 0$, $v(x)\leq 0$ in $\mathbb R^N.$
 As a consequence, $u_n^+(x)\leq |u_n(x)-u(x)|$ for any $n.$ Hence, for any $(U,V)\in\partial S_{s,x_0}^+,$
 \begin{equation*}
 \begin{aligned}
 \int_{\mathbb R^N}(u_n^+)^{p+1}dx&\leq \int_{\mathbb R^N}|u_n-u|^{p+1}dx
 \leq C||(U_n-U,V_n-V)||^{p+1}_{X_{s,x_0}}.
 \end{aligned}
 \end{equation*}
Noting that dist$((U_n,V_n),\partial S_{s,x_0}^+)\to 0$, then
 $$\limsup_{n\to\infty}\int_{\mathbb R^N}(u_n^+)^{p+1}dx=0.$$
 Similarly, we have $$\limsup_{n\to\infty}\int_{\mathbb R^N}(v_n^+)^{2_s^*}dx=0.$$
 As a result, for any $t>0,$
 \begin{equation*}
 \liminf_{n\to\infty}I_{x_0}(\overline\tau_{x_0}(U_n,V_n))\geq \liminf_{n\to\infty}I_{x_0}(tU_n,tV_n)\geq
  \liminf_{n\to\infty}\frac{\delta_0t^2}{4}||(U_n,V_n)||_{X_{s,x_0}}^2=\frac{\delta_0t^2}{4}.
 \end{equation*}
 Hence $I_{x_0}(\overline{\tau}_{x_0}(U_n,V_n))\to \infty.$ Thanks to Lemma \ref{59}, we conclude that
 if $\{\overline \tau_{x_0}(U_n,V_n)\}$ is bounded, then so is $\{I_{x_0}(\overline\tau_{x_0}(U_n,V_n))\}$. The claim then holds.

 {\bf Step 2.} In view of $I_{x_0}(U_n,V_n)\to c_{x_0},$ we have
 $$(\overline{U}_n,\overline{V}_n):=\overline{\tau}_{x_0}^{-1}(U_n,V_n)\in S_{s,x_0}^+,\quad
 \overline{\Phi}_{x_0}(\overline{U}_n,\overline{V}_n)=I_{x_0}(U_n,V_n)\to c_{x_0}.$$
 Define $\Upsilon:\overline{S_{s,x_0}^+}\to \mathbb R\cup\{\infty\}$ by
 \begin{equation*}
\Upsilon(U,V)=
\begin{cases}
 \overline{\Phi}_{x_0}(U,V)\quad&\text{if ~} (U,V)\in  S_{s,x_0}^+,\\
 \infty&\text{if ~} (U,V)\in  \partial S_{s,x_0}^+.
 \end{cases}
 \end{equation*}
 Then $\Upsilon\in C(\overline{S_{s,x_0}^+}, \mathbb R\cup\{\infty\})$ is bounded from below by $c_{x_0}$. Therefore, there is
 a (PS)$_{c_{x_0}}$ sequence
 $\{(\overline{W}_n,\overline{Z}_n)\}\subset  S_{s,x_0}^+$ for $\overline{\Phi}_{x_0}$ such that
 $$||(\overline{W}_n-\overline{U}_n,\overline{Z}_n-\overline{V}_n)||_{X_{s,x_0}}\to 0.$$
 Set $(W_n,Z_n)=\overline{\tau}_{x_0}(\overline{W}_n,\overline{Z}_n),$ then $\{(W_n,Z_n)\}$ is a (PS)$_{c_{x_0}}$ sequence of $I_{x_0}$ and
 \begin{equation}\label{41}
 ||({W}_n-{U}_n,{Z}_n-{V}_n)||_{X_{s,x_0}}\to 0.
 \end{equation}
Similar to the proof of Theorem \ref{augroundexistence}, the sequence
$\{(W_n,Z_n)\}$ has a convergent subsequence in $X_{s,x_0}$, and then so does $\{(U_n,V_n)\}$ in terms of \eqref{41}.
The proof is complete.
\end{proof}

\section{Multiplicity of  solutions for the modified  system }
In this section, we relate the number of positive vector solutions of \eqref{mp} to the topology of the set $\mathcal O$.
For this aim,  we choose $\delta>0$ such that $\mathcal O_\delta\subset\subset \Lambda. $

\begin{lemma}\label{37}
Under assumptions $(P1)$ and $(P2),$ if $\lambda<\min\{\sqrt{a_0b_0},\sqrt{(a_1-\mu_0)b_0}\},$ then
there holds
$$\lim_{\epsilon\to0}I_\epsilon\left(\tau_{\epsilon}\left(\widetilde{U}_{\epsilon,x_0},\widetilde{V}_{\epsilon,x_0}\right)\right)
=
c_0 \quad \text{uniformly for~} x_0\in \mathcal{O}, $$
where $\left(\widetilde{U}_{\epsilon,x_0},\widetilde{V}_{\epsilon,x_0}\right)$ is defined in \eqref{wide},
and $\tau_\epsilon$ is specified by (III) in Lemma \ref{Xepsilon}.
\end{lemma}
\begin{proof}
It follows from (II) in Lemma \ref{Xepsilon} that there exists $t_\epsilon>0$ such that $$\tau_{\epsilon}\left(\widetilde{U}_{\epsilon,x_0},\widetilde{V}_{\epsilon,x_0}\right)
=\left(t_\epsilon\widetilde{U}_{\epsilon,x_0},t_\epsilon\widetilde{V}_{\epsilon,x_0}\right)\in N_\epsilon.$$
By a similar argument to the proof of Lemma \ref{energybound}, we see that $t_\epsilon$ is bounded.
Moreover, if there exists a    subsequence denoted still by $\{t_\epsilon\}$ such that $t_\epsilon\to t_0$ as
$\epsilon\to 0,$  then $t_0>0$. Furthermore, thanks to $I_\epsilon^\prime\left(t_\epsilon\widetilde{U}_{\epsilon,x_0},t_\epsilon\widetilde{V}_{\epsilon,x_0}\right)
\left(t_\epsilon\widetilde{U}_{\epsilon,x_0},t_\epsilon\widetilde{V}_{\epsilon,x_0}\right)=0$,
we conclude by a similar discussion to the proof of \eqref{9} that
\begin{equation*}
\begin{aligned}
0&=I_\epsilon^\prime\left(t_\epsilon\widetilde{U}_{\epsilon,x_0},t_\epsilon\widetilde{V}_{\epsilon,x_0}\right)
\left(t_\epsilon\widetilde{U}_{\epsilon,x_0},t_\epsilon\widetilde{V}_{\epsilon,x_0}\right)
\to I_{x_0}^\prime(t_0U, t_0V)(t_0U, t_0V).
\end{aligned}
\end{equation*}
Namely, $(t_0U, t_0V)\in N_{x_0}.$ Recall that $(U,V)$ is a ground state of $I_{x_0}$, then $t_0=1$, this along with \eqref{9} implies that
$$\lim_{\epsilon\to0}I_\epsilon\left(t_\epsilon\widetilde{U}_{\epsilon,x_0},t_\epsilon\widetilde{V}_{\epsilon,x_0}\right)
=I_{x_0}(U,V)=c_0.$$
By virtue of $I_\epsilon^\prime\left(t_\epsilon\widetilde{U}_{\epsilon,x_0},t_\epsilon\widetilde{V}_{\epsilon,x_0}\right)
\left(t_\epsilon\widetilde{U}_{\epsilon,x_0},t_\epsilon\widetilde{V}_{\epsilon,x_0}\right)=0$ and a direct calculation similar to that in the proof of Lemma \ref{energybound}, we can see  the limit holds uniformly for $x_0\in\mathcal O.$
The proof is complete.
\end{proof}
We define $H_\epsilon:\mathcal O\to N_\epsilon$ by
\begin{equation*}
H_\epsilon(z)=\tau_\epsilon\left(\widetilde{U}_{\epsilon,z},\widetilde{V}_{\epsilon,z}\right)
=\left(t_\epsilon\widetilde{U}_{\epsilon,z},t_\epsilon\widetilde{V}_{\epsilon,z}\right).
\end{equation*}
Then by Lemma \ref{37}, there holds
\begin{equation}\label{39}
\lim_{\epsilon\to 0}I_\epsilon(H_\epsilon(z))=c_0\quad \text{uniformly in~} z\in\mathcal O.
\end{equation}
Moreover,
$t_\epsilon\to 1 $ as $\epsilon\to0.$ For any fixed $z\in\mathcal O,$
let $h(\epsilon)=\left|I_\epsilon\left(\tau_\epsilon\left(\widetilde{U}_{\epsilon,z},
\widetilde{V}_{\epsilon,z}\right)\right)-c_0\right|$, and
\begin{equation}\label{subnehari}
\widetilde{N}_\epsilon=\{(U,V)\in N_\epsilon:I_\epsilon(U,V)\leq c_0+h(\epsilon)\},
\end{equation}
then $h(\epsilon)\to 0$ as $\epsilon\to 0,$ and $\widetilde{N}_\epsilon\neq \emptyset$ because of
$\tau_\epsilon\left(\widetilde{U}_{\epsilon,z},\widetilde{V}_{\epsilon,z}\right)\in \widetilde{N}_\epsilon$
for any $z\in\mathcal O.$

Next, we show a concentration property for the functions in $\widetilde {N}_\epsilon.$  More  general, we have

\begin{lemma}\label{30}
Assume  (P1)-(P2) and $\lambda<\min\{\sqrt{a_0b_0},\sqrt{(a_1-\mu_0)b_0}\}.$
Let sequence $\{(U_\epsilon,V_\epsilon)\}$ satisfy $(U_\epsilon,V_\epsilon)\in N_\epsilon$ for any small $\epsilon$ and  $I_\epsilon(U_\epsilon,V_\epsilon)\to c_0$ as $\epsilon\to0.$ If either $(U_\epsilon,V_\epsilon)$ is a critical point of $I_\epsilon$ for each  small $\epsilon$  or assumption {\rm(P3)} holds, then
\begin{itemize}
\item [\rm(I)] there exist a subsequence $\{(U_{\epsilon_n},V_{\epsilon_n})\}$, a sequence $\{z_{\epsilon_n}\}\subset\mathbb R^N$ and $(\widetilde{U},\widetilde{V})\in\mathcal E\times \mathcal E$ such that
$$(\widetilde{U}_{\epsilon_n},\widetilde{V}_{\epsilon_n})\to (\widetilde{U},\widetilde{V}) \quad \text{as~} n\to \infty.$$
where
\begin{equation}\label{tildeu}
  (\widetilde{U}_{\epsilon_n}(x,y),\widetilde{V}_{\epsilon_n}(x,y)):=({U}_{\epsilon_n}(x+z_{\epsilon_n},y),
{V}_{\epsilon_n}(x+z_{\epsilon_n},y))
\quad \text{in~}\mathbb R^{N+1}_+.
\end{equation}
\item[\rm (II)]there exists
$z_0\in \mathcal O$ such that
$\epsilon_nz_{\epsilon_n}\to z_o$ as $n\to\infty.$
 \item [\rm(III)] $(\widetilde{U},\widetilde{V})$ is a positive vector ground state of system \eqref{autonomoussys} with $x_0$ replaced by $z_0.$
\end{itemize}
\end{lemma}

\begin{proof}
Note that  $\{(U_\epsilon,V_\epsilon)\}$ is bounded by assumptions. Then we claim that there exist $\{z_\epsilon\}\subset\mathbb R^N,$ $r>0$
and  $\theta>0$ such that
\begin{equation}\label{21}
\liminf_{\epsilon\to0}\int_{B_r(z_\epsilon)}u_\epsilon^2+v^2_\epsilon dx>\theta.
\end{equation}
Otherwise, let $l=\mathop{\lim}\limits_{\epsilon \to 0}||(U_\epsilon.V_\epsilon)||^2_{X_{s,\epsilon}},$ then it follows from Lemma \ref{concnetration2} and $(U_\epsilon.V_\epsilon)\in N_\epsilon$ that
$$\lim_{\epsilon\to 0}\int_{\mathbb R^N}v_\epsilon^{2_s^*}dx\geq l,$$
which along with Lemmas \ref{nehari} and \ref{59} that $l>S^{\frac{N}{2s}}.$ Furthermore,
$$\frac{s}{N}S^{\frac{N}{2s}}\leq\lim_{\epsilon\to 0}\left(I_\epsilon(U_\epsilon,V_\epsilon)-\frac{1}{2}I^\prime_\epsilon(U_\epsilon,V_\epsilon)(U_\epsilon,V_\epsilon)\right)=\mathop{\lim}\limits_{\epsilon\to 0} I_\epsilon(U_\epsilon,V_\epsilon)=c_0<\frac{s}{N}S^{\frac{N}{2s}},$$
we reach a contradiction.
Now define
\begin{equation}\label{22}
 (\widetilde{U}_{\epsilon}(x,y),\widetilde{V}_{\epsilon}(x,y))=({U}_{\epsilon}(x+z_{\epsilon},y),
{V}_{\epsilon}(x+z_{\epsilon},y))
\quad \text{in~}\mathbb R^{N+1}_+.
\end{equation}
Then $\{(\widetilde{U}_{\epsilon},\widetilde{V}_{\epsilon})\}$ is bounded in $\mathcal E\times\mathcal E,$
and then there exist a subsequence denoted still by  $\{(\widetilde{U}_{\epsilon},\widetilde{V}_{\epsilon})\}$ for convenience and $(\widetilde{U},\widetilde{V})$ such that
\begin{equation}\label{46}
(\widetilde{U}_{\epsilon},\widetilde{V}_{\epsilon})\rightharpoonup(\widetilde{U},\widetilde{V})~~ \text{in~} \mathcal E\times \mathcal E,\quad (\widetilde{u}_{\epsilon},\widetilde{v}_{\epsilon})\to(\widetilde{u},\widetilde{v})~~\text{in~}
L_{\text{loc}}^s(\mathbb R^N)\times L_{\text{loc}}^t(\mathbb R^N),~~\forall s,t\in [2,2_s^*).
\end{equation}
As a consequence,
\begin{equation}\label{25}
\int_{B_r}\tilde{u}^2+\tilde{v}^2 dx>\theta.
\end{equation}
We consider two cases in the following.

{\bf Case 1.} $(U_\epsilon,V_\epsilon)$ is a critical point of $I_\epsilon$ for each small $\epsilon.$ Then
$u_\epsilon\geq 0,$ $v_\epsilon\geq0$ in $\mathbb R^N$ for each small $\epsilon$.
 We
claim that $\mathop{\lim}\limits_{\epsilon\to0}{\rm dist}(\epsilon z_\epsilon,\Lambda)=0.$  Otherwise,
there exists $\rho>0$ such that ${\rm dist}(\epsilon z_\epsilon,\Lambda)>\rho$ for small $\epsilon.$ As a result,
\begin{equation*}
\Lambda_\epsilon\subset \mathbb R^N\setminus B_{\frac{\rho}{\epsilon}}(z_\epsilon).
\end{equation*}
It then follows from \eqref{mp} and \eqref{22} that for any  $(W,Z)\in C_c^\infty{(\overline{\mathbb R^{N+1}_+})}
\times C_c^\infty{(\overline{\mathbb R^{N+1}_+})}$ with $W\geq 0,Z\geq 0$ in $ \overline{\mathbb R^{N+1}_+},$
\begin{align*}
&\int_{\mathbb R^{N+1}_+}y^{1-2s}\nabla\widetilde{U}_\epsilon\cdot\nabla Wdxdy+\int_{\mathbb R^{N+1}_+}y^{1-2s}\nabla\widetilde{V}_\epsilon\cdot\nabla Zdxdy
+\int_{\mathbb R^N}a(\epsilon x+\epsilon z_\epsilon) \tilde{u}_\epsilon wdx \\
&+\int_{\mathbb R^N}b(\epsilon x+\epsilon z_\epsilon) \tilde{v}_\epsilon zdx
-\lambda\int_{\mathbb R^N} \tilde{u}_\epsilon zdx
-\lambda\int_{\mathbb R^N} \tilde{v}_\epsilon wdx
-\int_{(\Lambda_\epsilon-z_\epsilon)^c} f_{\epsilon}(x+z_\epsilon,\tilde{u}_\epsilon)wdx\\
&-\int_{(\Lambda_\epsilon-z_\epsilon)^c} g_{\epsilon}(x+z_\epsilon,\tilde{v}_\epsilon)zdx\\
=&\int_{\Lambda_\epsilon-z_\epsilon}\tilde u_\epsilon^{p}wdx
+\int_{\Lambda_\epsilon-z_\epsilon}\tilde v_\epsilon^{2_s^*-1}zdx.
\end{align*}
Letting  $\epsilon\to0$, then
\begin{equation*}
\begin{aligned}
&\int_{\mathbb R^{N+1}_+}y^{1-2s}\nabla\widetilde{U}\cdot\nabla Wdxdy+\int_{\mathbb R^{N+1}_+}y^{1-2s}\nabla\widetilde{V}\cdot\nabla Zdxdy
+\int_{\mathbb R^N}a_0 \tilde{u}wdx +\int_{\mathbb R^N}b_0 \tilde{v} zdx
\\
&
-\lambda\int_{\mathbb R^N} \tilde{u} zdx
-\lambda\int_{\mathbb R^N} \tilde{v} wdx
-\alpha^{p-1}\int_{\mathbb R^N} \tilde{u} wdx
-\alpha^{2_s^*-2}\int_{\mathbb R^N}\tilde{v} zdx\\
\leq& 0.
\end{aligned}
\end{equation*}
By the density of $C_c^\infty{(\overline{\mathbb R^{N+1}_+})}
$
in $\mathcal E$, we further get
$$\int_{\mathbb R^{N+1}_+}y^{1-2s}|\nabla\widetilde{U}|^2 dxdy+\int_{\mathbb R^{N+1}_+}y^{1-2s}|\nabla\widetilde{V}|^2dxdy
+\int_{\mathbb R^N}a_0 \tilde{u}^2dx
+\int_{\mathbb R^N}b_0 \tilde{v}^2dx=0,$$
which contradicts \eqref{25}. Hence the claim holds. As a consequence, $\{\epsilon z_\epsilon\}$ is bounded.
Then there exist a subsequence denoted still by $\{\epsilon z_\epsilon\}$ and $z_0\in\overline{\Lambda}$
such that
\begin{equation}\label{45}
\lim_{\epsilon\to0}\epsilon z_{\epsilon}=z_0.
\end{equation}
In view of $(\widetilde{U}_{\epsilon},\widetilde{V}_{\epsilon})\rightharpoonup
(\widetilde{U},\widetilde{V})$ in $\mathcal E\times\mathcal E$, we conclude that $(\widetilde{U},\widetilde{V})$
is a solution of system
 \begin{equation}\label{limitsys}
 \begin{cases}
 -{\rm div} (y^{1-2s}\nabla U)=0\quad &\text{in} \ \mathbb R^{N+1}_+,\\
 -{\rm div} (y^{1-2s}\nabla V)=0& \text{in}\ \mathbb R^{N+1}_+,\\
 -\mathop{\lim}\limits_{y\to0 }y^{1-2s}\frac{\partial U(x,y)}{\partial y}+a(z_0)u=\chi u^p
 +(1-\chi)f(u)+\lambda v &\text{on}\ \mathbb R^N,\\
 -\mathop{\lim}\limits_{y\to0 }y^{1-2s}\frac{\partial V(x,y)}{\partial y}+b(z_0)v=\chi v^{2_s^*-1}
 +(1-\chi)g(v)+\lambda u\quad &\text{on}\ \mathbb R^N,
 \end{cases}
 \end{equation}
 where $\chi(x)=\mathop{\lim}\limits_{\epsilon\to0}\chi_{\Lambda}(\epsilon x+\epsilon z_\epsilon)$ and
 hence, if $z_0\in\Lambda$ then $\chi$ is the characteristic function of $\mathbb R^N,$ and if $z_0\in\partial\Lambda$
 then $\chi$ is the characteristic function of a half-space in $\mathbb R^N.$
  The Euler-Lagrange functional associated to system \eqref{limitsys} is given by
 \begin{equation}\label{limiteuler}
 \begin{aligned}
 L_{z_0}(U,V)=&\frac{1}{2}\int_{\mathbb R^{N+1}_+} y^{1-2s} |\nabla U|^2 dxdy+\frac{1}{2}\int_{\mathbb R^{N+1}_+} y^{1-2s} |\nabla V|^2 dxdy+\frac{1}{2}\int_{\mathbb R^N}a(z_0)u^2 dx \\
 &+\frac{1}{2}\int_{\mathbb R^N}b( z_0)v^2 dx-\int_{\mathbb R^N}\chi \left(\frac{u^{p+1}}{p+1}+\frac{v^{2_s^*}}{2_s^*}\right) dx\\
 &-\int_{\mathbb R^N}(1-\chi)(F(u)+G(v))dx -\lambda\int_{\mathbb R^N}uvdx.
 \end{aligned}
 \end{equation}
 It then follows from \eqref{fandg} and a similar result to Lemma \ref{eeq} that
 \begin{equation}\label{26}
 L_{z_0}(\widetilde{U},\widetilde{V})=\max_{t>0}L_{z_0}(t\widetilde{U},t\widetilde{V})
 \geq \max_{t>0}I_{z_0}(t\widetilde{U},t\widetilde{V})\geq c_{z_0}.
 \end{equation}
 On the other hand, by Fatou's Lemma,
 \begin{equation}\label{27}
 \begin{aligned}
 L_{z_0}(\widetilde{U},\widetilde{V})=&L_{z_0}(\widetilde{U},\widetilde{V})-\frac{1}{2}
 L_{z_0}^\prime(\widetilde{U},\widetilde{V})(\widetilde{U},\widetilde{V})\\
 =&\int_{\mathbb R^N}\chi\left[\left(\frac{1}{2}-\frac{1}{p+1}\right)\tilde{u}^{p+1}
 +\left(\frac{1}{2}-\frac{1}{2_s^*}\right)\tilde{v}^{2_s^*}\right]dx\\
 &+\int_{\mathbb R^N}(1-\chi)\left(\frac{1}{2}f(\tilde{u})\tilde{u}-F(\tilde{u})+
 \frac{1}{2}g(\tilde{v})\tilde{v}-G(\tilde{v})\right)dx\\
 \leq&\lim_{\epsilon\to0}\int_{\mathbb R^N}\chi_{\Lambda_\epsilon}(x+z_\epsilon)\left[\left(\frac{1}{2}-\frac{1}{p+1}\right)\tilde{u}_\epsilon^{p+1}
 +\left(\frac{1}{2}-\frac{1}{2_s^*}\right)\tilde{v}_\epsilon^{2_s^*}\right]dx\\
 &+\lim_{\epsilon\to 0}\int_{\mathbb R^N}(1-\chi_{\Lambda_\epsilon}(x+z_\epsilon))\left(\frac{1}{2}f(\tilde{u}_\epsilon)\tilde{u}_\epsilon-F(\tilde{u}_\epsilon)+
 \frac{1}{2}g(\tilde{v}_\epsilon)\tilde{v}_\epsilon-G(\tilde{v}_\epsilon)\right)dx\\
 =&\lim_{\epsilon\to 0}\left(I_{\epsilon}(U_\epsilon,V_{\epsilon})-
 \frac{1}{2}I^\prime_{\epsilon}(U_\epsilon,V_{\epsilon})(U_\epsilon,V_{\epsilon})\right)\\
 =&\lim_{\epsilon\to 0}I_{\epsilon}(U_\epsilon,V_{\epsilon})=c_0.
 \end{aligned}
 \end{equation}
It then follows from \eqref{26} that
 $c_{z_0}=c_0=\mathop{\inf}\limits_{y\in\Lambda}c_y,$ that is, $z_0\in \mathcal O\subset\subset\Lambda.$
 Hence $\chi=1$ in $\mathbb R^N,$  and then $(\widetilde{U},\widetilde{V})$ is a ground state of system \eqref{autonomoussys} with $x_0$ replaced by $z_0.$ Moreover,
 \begin{equation*}
 \begin{aligned}
 \lim_{\epsilon\to 0}\int_{\Lambda_\epsilon-z_\epsilon}\tilde{u}_\epsilon^{p+1}
 dx
 =\int_{\mathbb R^N}\tilde{u}^{p+1}dx
  ,\quad
 \lim_{\epsilon\to 0}\int_{\Lambda_\epsilon-z_\epsilon}
 \tilde{v}_\epsilon^{2_s^*}dx
 =\int_{\mathbb R^N}
 \tilde{v}^{2_s^*}dx.
 \end{aligned}
 \end{equation*}
 Furthermore, by straight calculations, we see
 \begin{align*}
&\int_{\mathbb R^{N+1}_+}y^{1-2s}|\nabla\widetilde{U}_\epsilon|^2dxdy+\int_{\mathbb R^{N+1}_+}y^{1-2s}|\nabla\widetilde{V}_\epsilon|^2dxdy+\frac{a_0\delta_0}{2}\int_{\mathbb R^N} \tilde{u}_\epsilon ^2dx
+\frac{b_0\delta_0}{2}\int_{\mathbb R^N} \tilde{v}_\epsilon ^2dx\\
&+\int_{\mathbb R^N}\left(a(\epsilon x+\epsilon z_\epsilon)-\frac{a_0\delta_0}{2}\right) \tilde{u}_\epsilon^2dx +\int_{\mathbb R^N}\left(b(\epsilon x+\epsilon z_\epsilon)-\frac{b_0\delta_0}{2}\right) \tilde{v}_\epsilon ^2dx\\
&-\int_{(\Lambda_\epsilon-z_\epsilon)^c} f_{\epsilon}(x+z_\epsilon,\tilde{u}_\epsilon)\tilde{u}_\epsilon dx-\int_{(\Lambda_\epsilon-z_\epsilon)^c} g_{\epsilon}(x+z_\epsilon,\tilde{v}_\epsilon)\tilde{v}_\epsilon dx-2\lambda\int_{\mathbb R^N} \tilde{u}_\epsilon \tilde{v}_\epsilon dx\\
=&\frac{1}{p+1}\int_{\Lambda_\epsilon-z_\epsilon}\tilde{u}_\epsilon^{p+1}dx
+\frac{1}{2_s^*}\int_{\Lambda_\epsilon-z_\epsilon}\tilde{v}_\epsilon^{2_s^*}dx\\
\to& \frac{1}{p+1}\int_{\mathbb R^N}\tilde{u}^{p+1}dx
+\frac{1}{2_s^*}\int_{\mathbb R^N}\tilde{v}^{2_s^*}dx\\
=&\int_{\mathbb R^{N+1}_+}y^{1-2s}|\nabla\widetilde{U}|^2dxdy+\int_{\mathbb R^{N+1}_+}y^{1-2s}|\nabla\widetilde{V}|^2dxdy+\frac{a_0\delta_0}{2}\int_{\mathbb R^N} \tilde{u} ^2dx
+\frac{b_0\delta_0}{2}\int_{\mathbb R^N} \tilde{v} ^2dx\\
&+\int_{\mathbb R^N}\left(a(z_0)-\frac{a_0\delta_0}{2}\right) \tilde{u}^2dx +\int_{\mathbb R^N}\left(b( z_0)-\frac{b_0\delta_0}{2}\right) \tilde{v} ^2dx
-2\lambda\int_{\mathbb R^N} \tilde{u} \tilde{v} dx.
\end{align*}
On the other hand, we infer from \eqref{delta0} that
\begin{equation*}
\begin{aligned}
&\left(a(\epsilon x+\epsilon z_\epsilon)-\frac{a_0\delta_0}{2}\right) \tilde{u}_\epsilon^2 + \left(b(\epsilon x+\epsilon z_\epsilon)-\frac{b_0\delta_0}{2}\right) \tilde{v}_\epsilon ^2
-f_{\epsilon}(x+z_\epsilon,\tilde{u}_\epsilon)\tilde{u}_\epsilon  -g_{\epsilon}(x+z_\epsilon,\tilde{v}_\epsilon)\tilde{v}_\epsilon \\
&-2\lambda \tilde{u}_\epsilon \tilde{v}_\epsilon
\geq0.
\end{aligned}
 \end{equation*}
The Fatou's Lemma then implies that
 \begin{equation*}
 \begin{aligned}
 &\lim_{\epsilon\to 0}\int_{\mathbb R^{N+1}_+}y^{1-2s}|\nabla\widetilde{U}_\epsilon|^2dxdy+\int_{\mathbb R^{N+1}_+}y^{1-2s}|\nabla\widetilde{V}_\epsilon|^2dxdy\\
 =&\int_{\mathbb R^{N+1}_+}y^{1-2s}|\nabla\widetilde{U}|^2dxdy+\int_{\mathbb R^{N+1}_+}y^{1-2s}|\nabla\widetilde{V}|^2dxdy,
 \end{aligned}
 \end{equation*}
 $$\lim_{\epsilon\to 0}\int_{\mathbb R^N}\tilde{u}_\epsilon^2dx=\int_{\mathbb R^N}\tilde{u}^2dx,\quad
 \lim_{\epsilon\to 0}\int_{\mathbb R^N}\tilde{v}_\epsilon^2dx=\int_{\mathbb R^N}\tilde{v}^2dx. $$
 Hence $(\widetilde{U}_\epsilon,\widetilde{V}_\epsilon)\to (\widetilde{U},\widetilde{V})$
 as $\epsilon\to 0$
 in $\mathcal E\times\mathcal E.$

{\bf Case 2.} There is $x_0\in\Lambda$ such that $a(x_0)=a_0$ and $b(x_0)=b_0.$ Then
 there exists $t_\epsilon>0$ such that
$(t_\epsilon\widetilde{U}_{\epsilon},t_\epsilon\widetilde{V}_{\epsilon})\in N_{x_0}.$
By  \eqref{fepsilon}, \eqref{gepsilon} and a change of variables, we have
\begin{equation}\label{47}
\begin{aligned}
c_0\leq& I_{x_0}(t_\epsilon\widetilde{U}_{\epsilon},t_\epsilon\widetilde{V}_{\epsilon})\\
=&\frac{t_\epsilon^2}{2}\int_{\mathbb R^{N+1}_+}y^{1-2s}|\nabla {U}_{\epsilon}|^2dxdy
+\frac{t_\epsilon^2}{2}\int_{\mathbb R^{N+1}_+}y^{1-2s}|\nabla {V}_{\epsilon}|^2dxdy\\
&+\frac{t_\epsilon^2}{2}\int_{\mathbb R^N} a(x_0) {u}^2_\epsilon dx +\frac{t_\epsilon^2}{2}\int_{\mathbb R^N} b(x_0) {v}^2_\epsilon dx-\frac{t_\epsilon^{p+1}}{p+1}\int_{\mathbb R^N} {u}^{p+1}_\epsilon dx\\
&-\frac{t_\epsilon^{2_s^*}}{2_s^*}\int_{\mathbb R^N} {v}^{2_s^*}_\epsilon dx-
2\lambda \int_{\mathbb R^N}{u}_\epsilon{v}_\epsilon dx\\
\leq & I_\epsilon (t_\epsilon U_\epsilon, t_\epsilon V_\epsilon)\\
\leq& c_\epsilon\to c_0,
\end{aligned}
\end{equation}
and then $\{(t_\epsilon\widetilde{U}_{\epsilon},t_\epsilon\widetilde{V}_{\epsilon})\}$ is bounded,
which along with the boundedness of $\{(\widetilde{U}_{\epsilon},\widetilde{V}_{\epsilon})\}$ implies that
there exists a subsequence of $t_\epsilon$ (still denoted by $t_\epsilon$) such that it converges to $t_0>0$.
In view of Lemma \ref{42} and \eqref{46}, we then derive, up to a subsequence,
\begin{equation}\label{44}
(\widetilde{U}_{\epsilon},\widetilde{V}_{\epsilon})\to\left(\widetilde{U},\widetilde{V}\right)~~ \text{in~} \mathcal E\times \mathcal E.
\end{equation}
We now claim that $\mathop{\lim}\limits_{\epsilon\to0}{\rm dist}(\epsilon z_\epsilon,\Lambda)=0.$
Otherwise, there exists $\rho>0$ such that ${\rm dist}(\epsilon z_\epsilon,\Lambda)>\rho$ for small $\epsilon.$ As a result,
\begin{equation}\label{43}
\Lambda_\epsilon\subset \mathbb R^N\setminus B_{\frac{\rho}{\epsilon}}(z_\epsilon).
\end{equation}
Since $I^\prime_\epsilon({U}_{\epsilon},{V}_{\epsilon})({U}_{\epsilon},{V}_{\epsilon})=0$, we have
\begin{equation*}
\begin{aligned}
\frac{\delta_0}{2}||(\widetilde U, \widetilde V)||_{X_{s,x_0}}&=
\frac{\delta_0}{2}\lim_{\epsilon\to0}||( \widetilde U_\epsilon,  \widetilde V_\epsilon)||_{X_{s,x_0}}\leq
\lim_{\epsilon\to0}\left(\int_{\Lambda_\epsilon-z_\epsilon} |u_\epsilon|^{p+1}dx + \int_{\Lambda_\epsilon-z_\epsilon} |v_\epsilon|^{2_s^*}dx\right)=0,
\end{aligned}
\end{equation*}
 which contradicts \eqref{25}.
The claim then holds. As a consequence,  there is $z_0\in \overline{\Lambda}$ such that, up to a subsequence,
$\epsilon z_\epsilon\to z_0$ as $\epsilon\to 0.$  Now we  prove $z_0\in \mathcal O.$ In fact,
it follows from Fatou's lemma, \eqref{44}, \eqref{fepsilon}, \eqref{gepsilon}
and $(t_\epsilon\widetilde{U}_\epsilon, t_\epsilon\widetilde{V}_\epsilon)\to(t_0\widetilde{U}, t_0\widetilde{V})$ that
\begin{equation}\label{48}
\begin{aligned}
&I_{z_0}(t_0\widetilde{U},t_0\widetilde{V})\\
\leq &\liminf_{\epsilon\to 0}\left( \frac{t_\epsilon^2}{2}\int_{\mathbb R^{N+1}_+}y^{1-2s}|\nabla \widetilde{U}_\epsilon|^2+ \frac{t_\epsilon^2}{2}\int_{\mathbb R^{N+1}_+}y^{1-2s}|\nabla \widetilde{V}_\epsilon|^2
+\frac{t_\epsilon^2}{2}\int_{\mathbb R^{N}}a(\epsilon x+\epsilon z_\epsilon)\widetilde{u}_\epsilon^2dx \right.\\
&\qquad\qquad\left.+\frac{t_\epsilon^2}{2}\int_{\mathbb R^{N}}b(\epsilon x+\epsilon z_\epsilon)\widetilde{v}_\epsilon^2dx -\int_{\mathbb R^{N}} F_\epsilon(x+z_\epsilon, t_\epsilon\widetilde{u}_\epsilon)dx\right.\\
&\qquad\qquad\left.-\int_{\mathbb R^{N}} G_\epsilon(x+z_\epsilon, t_\epsilon\widetilde{v}_\epsilon)dx-2\lambda t_\epsilon^2\int_{\mathbb R^{N}}\widetilde{u}_\epsilon\widetilde{u}_\epsilon dx
\right)\\
=&\liminf_{\epsilon\to 0}I_\epsilon(t_\epsilon U_\epsilon,t_\epsilon V_\epsilon)\leq c_0.
\end{aligned}
\end{equation}
 On the other hand, $c_0=I_{x_0}(t_0\widetilde{U},t_0\widetilde{V})$ due to \eqref{47} and \eqref{44}.
 Hence $a(z_0)=a_0$ and $b(z_0)=b_0$ by \eqref{48}, that is $z_0\in \mathcal O.$ It then follows from \eqref{44} that
 $(\widetilde{U},\widetilde{V})$ is a ground state of \eqref{autonomoussys} with $x_0$ replaced by $z_0.$
 The proof is complete.
\end{proof}
\begin{cor}
Assume  (P1)-(P2) and $\lambda<\min\{\sqrt{a_0b_0},\sqrt{(a_1-\mu_0)b_0}\}.$ Let $(U_\epsilon,V_\epsilon)$ be a
ground state of $I_\epsilon,$ then $I_\epsilon(U_\epsilon,V_\epsilon)\to c_0$ as $\epsilon\to0.$
\end{cor}
\begin{proof}
The proof follows directly from \eqref{26}, \eqref{27}, Lemma \ref{energybound} and Theorem \ref{augroundvalue}.
\end{proof}

Set $\rho=\rho(\delta)$ such that $\mathcal O_\delta\subset B_\rho.$
Define $\chi:\mathbb R^N\to \mathbb R^N$ as $\chi(x)=x$ if $x\in B_\rho$ and
$\chi(x)=\frac{\rho x}{|x|}$ if $|x|>\rho.$ Consider the barycenter map $\beta_\epsilon: N_\epsilon\to \mathbb R^N,$
\begin{equation*}
\beta_{\epsilon}(U,V):= \frac{\int_{\mathbb R^N}\chi(\epsilon x)u^2dx+\int_{\mathbb R^N}\chi(\epsilon x)v^2dx} {\int_{\mathbb R^N}u^2dx+\int_{\mathbb R^N}v^2dx}.
\end{equation*}
Then by a change of variables $\bar x=x-\frac{z}{\epsilon}$ and the dominated convergence theorem, we have
\begin{equation}\label{40}
\begin{aligned}
&\beta_\epsilon\left(\tau_\epsilon\left(\widetilde{U}_{\epsilon,z},\widetilde{V}_{\epsilon,z}\right)\right)\\
=&\frac{t_\epsilon^2\int_{\mathbb R^N}\chi(\epsilon \bar x +z)u^2(\bar x)\eta^2(\epsilon \bar x,0)d\bar x+t_\epsilon^2\int_{\mathbb R^N}\chi(\epsilon \bar x +z)v^2(\bar x)\eta^2(\epsilon \bar x,0)d\bar x} {\int_{\mathbb R^N}u^2(\bar x)\eta^2(\epsilon\bar x,0)d\bar x+\int_{\mathbb R^N}v^2(\bar x)\eta^2(\epsilon \bar x,0)d\bar x}\\
=&t_\epsilon^2 z+\epsilon t_\epsilon^2 \frac{\int_{\mathbb R^N} \bar x u^2(\bar x)\eta^2(\epsilon \bar x,0)d\bar x+\int_{\mathbb R^N} \bar x v^2(\bar x)\eta^2(\epsilon \bar x,0)d\bar x} {\int_{\mathbb R^N}u^2(\bar x)\eta^2(\epsilon\bar x,0)d\bar x+\int_{\mathbb R^N}v^2(\bar x)\eta^2(\epsilon \bar x,0)d\bar x}\\
\to& z\quad\text{uniformly in~} z\in \mathcal O~\text{as~} \epsilon\to 0.
\end{aligned}
\end{equation}
where $\left(\widetilde{U}_{\epsilon,z},\widetilde{V}_{\epsilon,z}\right)$ is defined in \eqref{wide},
and $\tau_\epsilon$ is specified by (III) in Lemma \ref{Xepsilon}.
We further discuss  the concentration property  of  barycenters  for the functions in $\widetilde{N}_\epsilon$.

\begin{lemma}Assume  (P1)-(P3) and $\lambda<\min\{\sqrt{a_0b_0},\sqrt{(a_1-\mu_0)b_0}\}.$ Then
for any  $\delta>0$ such that $\mathcal O_\delta\subset\subset \Lambda,$  there holds
\begin{equation*}
\lim_{\epsilon\to 0}\sup_{(U,V)\in\widetilde{N}_\epsilon}{\rm dist}(\beta_\epsilon(U,V),\mathcal O_\delta)=0.
\end{equation*}
\end{lemma}
\begin{proof}
For any $\epsilon>0$, there exists $(U_\epsilon,V_\epsilon)\in\widetilde{N}_\epsilon$ such that
$$\sup_{(U,V)\in\widetilde{N}_\epsilon}{\rm dist}(\beta_\epsilon(U,V),\mathcal O_\delta)\leq
{\rm dist}(\beta_\epsilon(U_\epsilon,V_\epsilon),\mathcal O_\delta)+o(\epsilon).$$
In view of \eqref{26} and \eqref{27}, we see
$$c_0\leq\liminf_{\epsilon\to0}c_\epsilon\leq I_\epsilon(U_\epsilon,V_\epsilon)\leq c_0+h(\epsilon).$$
Lemma \ref{30} then implies that there exists $\{\tilde{z}_\epsilon\}\subset\mathbb R^N$ such that
\begin{equation}\label{38}
 z_\epsilon:=\epsilon \tilde{z}_\epsilon
\to z_0\in \mathcal O.
\end{equation}
As a consequence, for small $\epsilon,$ we have
$z_\epsilon\in \mathcal O_\delta$
 and
\begin{equation*}
\begin{aligned}
\beta_\epsilon(U_\epsilon,V_\epsilon)=z_\epsilon+\frac{\int_{\mathbb R^N}(\chi(\epsilon \bar x+z_\epsilon)-z_\epsilon) u^2(\bar x+\tilde{z}_\epsilon)d\bar x+\int_{\mathbb R^N} (\chi(\epsilon \bar x+z_\epsilon)-z_\epsilon) v^2(\bar x+\tilde{z}_\epsilon)d\bar x} {\int_{\mathbb R^N}u^2(\bar x+\tilde{z}_\epsilon)d\bar x+\int_{\mathbb R^N}v^2(\bar x+\tilde{z}_\epsilon)d\bar x}.
\end{aligned}
\end{equation*}
The dominated convergence theorem further yields that $$\lim_{\epsilon\to0} {\rm dist}(\beta_\epsilon(U_\epsilon,V_\epsilon),z_\epsilon)=0.$$
As a consequence,
$$\sup_{(U,V)\in\widetilde{N}_\epsilon}{\rm dist}(\beta_\epsilon(U,V),\mathcal O_\delta)\leq{\rm dist}
(\beta_\epsilon(U_\epsilon,V_\epsilon),\mathcal O_\delta)+o(\epsilon)\leq  {\rm dist}(\beta_\epsilon(U_\epsilon,V_\epsilon),z_\epsilon)+o(\epsilon)\to 0$$
as $\epsilon\to0.$ The proof is complete.
\end{proof}

\begin{thm}\label{multiplicity}
Assume  (P1)-(P3) and $\lambda<\min\{\sqrt{a_0b_0},\sqrt{(a_1-\mu_0)b_0}\}.$
Then for any  $\delta>0$ such that $\mathcal O_\delta\subset\subset \Lambda$ and small $\epsilon$, system \eqref{mp} has at least $cat_{\mathcal O_\delta}(\mathcal O)$ positive vector solutions.
Moreover, if we denote by $\{(U_\epsilon,V_\epsilon)\}$ a sequence of these solutions, then
$$\lim_{\epsilon\to0} I_\epsilon(U_\epsilon, V_\epsilon)\to c_0.$$
\end{thm}
\begin{proof}
For any $\epsilon>0,$ define $\pi_\epsilon:\mathcal O\to S_\epsilon^+$ by
\begin{equation*}
\pi_\epsilon(z)=\overline{\tau}_\epsilon^{-1}(H_\epsilon(z)),\quad \forall z\in \mathcal O,
\end{equation*}
where $\overline{\tau}_\epsilon$ is defined in Lemma \ref{Xepsilon}. Moreover, by \eqref{phi} and \eqref{39},
 $$\lim_{\epsilon\to 0}\Phi_\epsilon(\overline{\tau}_\epsilon^{-1}(H_\epsilon(z)))=\lim_{\epsilon\to0}
 I_\epsilon(H_\epsilon(z))=c_0\quad\text{uniformly in~} z\in\mathcal O.$$
 As a consequence, there exists $\hat{\epsilon}>0$ such that for any $\epsilon\in (0,\hat{\epsilon})$,
 $$\pi_\epsilon(\mathcal O)\subset \widetilde{S}_\epsilon^+:=\{(U,V)\in S_\epsilon^+:\Phi_\epsilon(U,V)\leq c_0+h(\epsilon)\}.$$
Consequently, we have
$$\mathcal O \xrightarrow {H_\epsilon}H_\epsilon(\mathcal O)\xrightarrow{\overline{\tau}_\epsilon^{-1}
}\pi_\epsilon(\mathcal O)\xrightarrow {\overline{\tau}_\epsilon}H_\epsilon(\mathcal O)\xrightarrow{\beta_\epsilon}
\mathcal O_\delta. $$
On the other hand,  for any $z\in\mathcal O$ and small $\epsilon$,
 there exists  $\theta(\epsilon,z)\in\mathbb R^N$
 with $|\theta(\epsilon,z)|<\frac{\delta}{2}$ such that $\beta_\epsilon(H_\epsilon(z))=z+\theta(\epsilon,z)$ thanks to \eqref{40}.
 Define $Q:[0,1]\times\mathcal O\to \mathcal O_\delta$ by
  $$Q(t,z)=z+(1-t)\theta(\epsilon,z),$$
  then $Q$ is continuous. Obviously, $Q(0,z)=\beta_\epsilon(H_\epsilon(z)),$ $Q(1,z)=z$ for all $z\in\mathcal O.$ Therefore,
  $Q$ is a homotopy between $\beta_\epsilon\circ H_\epsilon=\beta_\epsilon\circ \tau_\epsilon\circ\pi_\epsilon$
  and the inclusion map $Id:\mathcal O\to \mathcal O_\delta.$ As a result,
  $$cat_{\pi_\epsilon(\mathcal O)}\pi_\epsilon(\mathcal O)\geq cat_{\mathcal O_\delta}(\mathcal O).$$
  It follows from Corollary \ref{spherepscondition} and category theory (see \cite[Corollary 28]{szukin-weth}) that $\overline{\Phi}_\epsilon$ has at least
  $cat_{\pi_\epsilon(\mathcal O)}\pi_\epsilon(\mathcal O)$  critical points on
  $\widetilde {S}_\epsilon^+:=\{(U,V)\in S_\epsilon^+:\overline{\Phi}_\epsilon(U,V)\leq c_0+h(\epsilon)\}.$ In view of Lemma \ref{sphere},
  $I_\epsilon$ admits at least $cat_{\mathcal O_\delta}(\mathcal O)$ critical points in $\widetilde{N}_\epsilon.$
  Namely, system \eqref{mp} has at least $cat_{\mathcal O_\delta}(\mathcal O)$ solutions.
  By a similar argument to the proof of Theorem \ref{energybound} and \eqref{12},
  we conclude that the components of these solutions are positive.
  The proof is complete.
\end{proof}

\section{Proofs of main theorems }

This section is devoted to the proofs of main theorems. To prove the existence and multiplicity of positive vector solutions of system \eqref{op},
we only need to show that the solutions of system \eqref{mp} obtained  in Theorem \ref{multiplicity} and Lemma \ref{energybound} also solve system \eqref{ep} for any small $\epsilon$.
Then we further investigate the decay estimate and concentration property of positive vector solutions of system \eqref{op}.

\begin{lemma}\label{uniformlybound}
Assume  (P1)-(P3) and $\lambda<\min\{\sqrt{a_0b_0},\sqrt{(a_1-\mu_0)b_0}\}.$
 Let $({U}_\epsilon,{V}_\epsilon)$ be the positive vector solution of system \eqref{mp} obtained in Theorem \ref{multiplicity} or Lemma \ref{energybound} for any small $\epsilon$.
Then there exists a positive constant $C$ independent of $\epsilon$ such that
for any $(\widetilde U_\epsilon,\widetilde V_\epsilon )$ defined in \eqref{tildeu}, there holds
$$||\tilde{u}_\epsilon+\tilde{v}_\epsilon||_\infty\leq C.$$
\end{lemma}
\begin{proof}
 Let $\widetilde{W}_\epsilon=\widetilde{U}_\epsilon+\widetilde{V}_\epsilon$ in $\mathbb R^{N+1}_+,$
 then it follows from \eqref{fepsilon} and \eqref{gepsilon} that
$$f_{\epsilon}(x+z_\epsilon,\tilde{u}_\epsilon)+g_\epsilon(x+z_\epsilon,\tilde{v}_\epsilon)+\lambda \tilde{w}_\epsilon
\leq \tilde{u}_\epsilon^p+\tilde{v}_\epsilon^{2_s^*-1}+\lambda \tilde{w}_\epsilon\leq C(1+\tilde{w}_\epsilon^{2_s^*-1}),$$
where $C$ is a positive constant independent of $\epsilon.$
Furthermore, we can verify from \eqref{mp} and \eqref{tildeu} that $\widetilde{W}_\epsilon$ is a subsolution  of
\begin{equation}\label{29}
\begin{cases}
-{\rm div} (y^{1-2s}\nabla \widetilde{W}_\epsilon)=0\quad&\text{in~} \mathbb R^{N+1}_+,\\
-\mathop{\lim}\limits_{y\to0}y^{1-2s}\frac{\partial \widetilde{W}_\epsilon(x,y)}{\partial y}=C(1+\tilde{w}_\epsilon^{2_s^*-1})
&\text{on~} \mathbb R^{N}.
\end{cases}
\end{equation}
Then by the Moser's iteration  for scalar equations (see, e.g., \cite[Lemma 4.1]{he-zou}), we can complete the proof.
\end{proof}

\begin{proof}[\bf Proofs of Theorem \ref{53} and \ref{mul}]
To verify the existence and multiplicity of positive vector solutions of system \eqref{op}, we only need to show that
there exists $\epsilon_0$ such that any positive vector solution $({U}_\epsilon,{V}_\epsilon)$ of system \eqref{mp} obtained in Theorem \ref{multiplicity} and Lemma \ref{energybound} also solves system \eqref{ep} for any $0<\epsilon<\epsilon_0.$
Set $\widetilde{W}_\epsilon=\widetilde{U}_\epsilon+\widetilde{V}_\epsilon$,
where $(\widetilde{U}_\epsilon,\widetilde{V}_\epsilon)$ is defined in \eqref{tildeu}.
Then there exists $\widetilde{W}\in\mathcal E$ such that $\widetilde{W}_\epsilon\to \widetilde{W}$  as $\epsilon\to 0$ due to Lemma \ref{30}, and
 $||\tilde{w}_\epsilon||_\infty<C$ for some constant $C>0$ independent of $\epsilon.$
Moreover, $\widetilde{W}_\epsilon\in C(\overline{\mathbb R^{N+1}_+})$ by \cite[Corollary 2.1]{Jin-Li-Xiong}, and  $\widetilde{W}_\epsilon$ is a subsolution  of
\begin{equation*}
\begin{cases}
-{\rm div} (y^{1-2s}\nabla \widetilde{W}_\epsilon)=0\quad&\text{in~} \mathbb R^{N+1}_+,\\
-\mathop{\lim}\limits_{y\to0}y^{1-2s}\frac{\partial \widetilde{W}_\epsilon(x,y)}{\partial y}=\lambda \tilde{w}_\epsilon+C\tilde{w}_\epsilon^{\frac{4s}{N+2s}} &\text{on} ~\mathbb R^N.
\end{cases}
\end{equation*}
It follows from \cite[Proposition 2.4]{Jin-Li-Xiong} that for any $x\in\mathbb R^N,$
\begin{equation}\label{31}
\begin{aligned}
\sup_{B_1(x)\times [0,1) }\widetilde{W}_\epsilon&\leq\left(\int_{B_2(x)\times (0,2)}y^{1-2s}|\widetilde{W}_\epsilon|^{2\gamma}dxdy\right)^{\frac{1}{2\gamma}}+
C\left(\int_{B_2(x)}\tilde{w}_\epsilon^2\right)^{\frac{2s}{N+2s}}\\
&\leq\left(\int_{B_2(x)\times (0,2)}y^{1-2s}\left|\widetilde{W}_\epsilon-\widetilde{W}\right|^{2\gamma}dxdy\right)^{\frac{1}{2\gamma}}
+\left(\int_{B_2(x)\times (0,2)}y^{1-2s}|\widetilde{W}|^{2\gamma}dxdy
\right)^{\frac{1}{2\gamma}}\\
&~~~~+C\left(\int_{B_2(x)}(\tilde{w}_\epsilon-\tilde{w})^2dx+\int_{B_2(x)}\tilde{w}^2dx\right)^{\frac{2s}{N+2s}}.
\end{aligned}
\end{equation}
Since $\widetilde{W}_\epsilon\to\widetilde{W}$ in $\mathcal E,$  for any $\theta>0$ there exist $\epsilon_0>0$ and $R>0$ such that  for any $0<\epsilon<\epsilon_0$,
$$\int_{\mathbb R^{N+1}_+}y^{1-2s}\left|\widetilde{W}_\epsilon-\widetilde{W}\right|^{2\gamma}dxdy\leq\theta,\quad
\int_{\mathbb R^N}(\tilde{w}_\epsilon-\tilde{w})^2dx\leq\theta,$$
$$\int_{\mathbb R^{N+1}_+\setminus\mathcal B^+_{R}}y^{1-2s}|\widetilde{W}|^{2\gamma}dxdy\leq\theta,\quad
\int_{\mathbb R^N\setminus B_R}\tilde{w}^2dx\leq\theta,$$
which along with \eqref{31}  implies that
\begin{equation}\label{34}
\lim_{|x|\to\infty}\tilde{w}_\epsilon(x)=0\quad \text{uniformly in~} \epsilon.
\end{equation}
In other words, there exists $R_\alpha>0$ such that  $\tilde{w}_\epsilon<\alpha$
for any $\epsilon\in(0,\epsilon_0)$
and $|x|>R_\alpha.$ Consequently,
\begin{equation}\label{32}
u_\epsilon(x)+v_\epsilon(x)=\tilde{w}_\epsilon(x-z_\epsilon)<\alpha\quad\text{for any }x\in\mathbb R^N\setminus B_{{R_\alpha}}(z_\epsilon).
\end{equation}
On the other hand, since $\epsilon z_\epsilon\to z_0\in\mathcal O$ as $\epsilon\to0,$
there exists $r>0$ such that $B_r(\epsilon z_\epsilon)\subset\subset \Lambda$ for small $\epsilon.$
Namely, $\Lambda_\epsilon^c\subset\subset \mathbb R^N\setminus B_{\frac{r}{\epsilon}}(z_\epsilon)$
for small $\epsilon.$ Let $\epsilon$ be sufficiently small such that $\epsilon R_\alpha<r$, then
$\Lambda_\epsilon^c\subset\subset\mathbb R^N\setminus B_{R_\alpha}(z_\epsilon).$
It then follows from \eqref{32} that
\begin{equation}\label{33}
 ||u_\epsilon+v_\epsilon||_{L^\infty(\mathbb R^N\setminus \Lambda_\epsilon)}<\alpha
\quad\text{for small~} \epsilon.
\end{equation}
In terms of \eqref{fepsilon} and \eqref{gepsilon}, we have
$$f_\epsilon(x,{u}_\epsilon)={u}_\epsilon^{p},\quad
g_\epsilon(x,{v}_\epsilon)={v}_\epsilon^{2_s^*-1}\quad \text{for all~} x\in \mathbb R^N
\text{~and small~}\epsilon.$$
That is to say, $({U}_\epsilon,{V}_\epsilon)$ is a solution of system \eqref{ep}.

We  claim that $||w_\epsilon||_{L^\infty(\mathbb R^N)}\geq\alpha$ for small $\epsilon$.
Otherwise, by \eqref{fepsilon} and \eqref{gepsilon}, we have
$$f_\epsilon(x,u_\epsilon)\leq \alpha^{p-1}u_\epsilon,\quad g_\epsilon(x,v_\epsilon)\leq \alpha^{2_s^*-2}v_\epsilon,\quad
\forall x\in \mathbb R^N. $$
Furthermore, it follows from $I_\epsilon^\prime(U_\epsilon,V_\epsilon)(U_\epsilon,V_\epsilon)=0$ and \eqref{delta0}
that $||(U_\epsilon,V_\epsilon)||_{X_{s,\epsilon}}=0,$ which contradicts (I) in Lemma \ref{nehari}.
The claim holds.
Therefore, there exists $x_\epsilon\in\Lambda_\epsilon$ such that
$$w_\epsilon(x_\epsilon)=\max_{x\in\mathbb R^N} w_\epsilon(x).$$
By a direct calculation we see that $(\hat{u}_\epsilon(x),\hat{v}_\epsilon(x)):=({u}_\epsilon(\frac{x}{\epsilon}),{v}_\epsilon(\frac{x}{\epsilon}))$
in $\mathbb R^N$  is a positive vector solution of system \eqref{op}. Moreover, $\hat{u}_\epsilon+\hat{v}_\epsilon$ achieves its maximum at $\epsilon x_\epsilon.$
  Next, we prove $\epsilon x_\epsilon\to z_0$ as $\epsilon\to 0.$
In fact,
 $\epsilon x_\epsilon\in B_{\epsilon R_\alpha}(\epsilon z_\epsilon)$ for small $\epsilon$ due to \eqref{32}, which along with
Lemma \ref{30} implies
$\epsilon x_\epsilon\to z_0$
as $\epsilon\to0,$ and then the property (II) in Theorems is a direct result of Lemma \ref{30}.

Thanks to $\lambda<(1-\delta_0)\sqrt{a_0b_0},$ we have
$$\frac{\lambda}{(1-\delta_0)a_0}<\frac{(1-\delta_0)b_0}{\lambda}.$$
Let $\theta=\frac{1}{2}\left(\frac{\lambda}{(1-\delta_0)a_0}+\frac{(1-\delta_0)b_0}{\lambda}\right)$
and $\overline{W}_\epsilon=\theta \widetilde{U}_\epsilon+\widetilde{ V}_\epsilon$. Then
\begin{equation}\label{35}
\begin{aligned}
(-\triangle)^s\overline{w}_\epsilon&\leq (\lambda-a_0\theta)\tilde{u}_\epsilon+(\lambda\theta-b_0)\tilde{v}_\epsilon+\theta \tilde{u}_\epsilon^{p}+\tilde{v}_\epsilon^{2_s^*-1}\\
&\leq -\delta_0a_0\theta \tilde{u}_\epsilon-\delta_0b_0 \tilde{v}_\epsilon+\theta \tilde{u}_\epsilon^{p}+\tilde{v}_\epsilon^{2_s^*-1}\\
&\leq-\bar{\theta}\overline{w}_\epsilon +\theta^{1-p} \overline{w}_\epsilon^{p}+\overline{w}_\epsilon^{2_s^*-1},
\end{aligned}
\end{equation}
where $\bar{\theta}=\min\{\delta_0a_0\theta,\delta_0b_0\}.$ By virtue of \eqref{34}, there exists $R>0$ such that
for any $|x|>R$ and small $\epsilon$ there holds
$$-\bar{\theta}\overline{w}_\epsilon +\theta^{1-p} \overline{w}_\epsilon^{p}+\overline{w}_\epsilon^{2_s^*-1}
\leq-\frac{\bar{\theta}}{2}\overline{w}_\epsilon,$$
which together with \eqref{35} yields
\begin{equation*}
(-\triangle)^s \overline{w}_\epsilon+\frac{\bar{\theta}}{2}\overline{w}_\epsilon\leq0,\quad \forall |x|>R.
\end{equation*}
Noting from \cite[Lemma 4.3]{Felmer-Quaas-Tan} that there is a continuous function $w$  such that
\begin{equation*}
(-\triangle)^s w+\frac{1}{2}w=0,\quad \forall |x|>1,
\end{equation*}
and $0<w<\frac{C_2}{|x|^{N+2s}}$ for some $C_2>0.$ Now we set $\phi(x)=w\left(\bar{\theta}^{\frac{1}{2s}}x\right)$ in $\mathbb R^N,$ then
\begin{equation*}
(-\triangle)^s \phi+\frac{\bar{\theta}}{2}\phi=0,\quad \forall |x|>\bar{\theta}^{-\frac{1}{2s}},
\end{equation*}
and $0<\phi<\frac{C^\prime}{|x|^{N+2s}}$ with $C^\prime>0.$ Define
$$R^\prime=\max\{\bar{\theta}^{-\frac{1}{2s}},R\},\quad
\alpha =\sup_{\epsilon}||\overline{w}_\epsilon||_\infty <\infty, \quad\beta=\inf_{B_{R^\prime}}\phi(x)>0$$
and $\psi=\alpha \phi-\beta\overline{w}_\epsilon$ in $\mathbb R^N,$ then
\begin{equation}\label{36}
\begin{cases}
(-\triangle)^s \psi+\frac{\bar{\theta}}{2}\psi\geq0\quad &\text{in~}\mathbb R^N\setminus B_{R^\prime},\\
\psi\geq0&\text{in~} B_{R^\prime}.
\end{cases}
\end{equation}
We claim that $\psi\geq0$ in $\mathbb R^N.$ Otherwise, there exists $\{x_n\}\subset\mathbb R^N$ such that
$$\lim_{n\to\infty}\psi(x_n)=\inf_{x\in\mathbb R^N}\psi(x)<0.$$
Since $\psi(x)\to 0$ as $|x|\to\infty$, $\{x_n\}$ is bounded. Hence there exists $x_0\in\mathbb R^N$
such that, up to a subsequence, $x_n\to x_0.$ It then follows  that
$$\psi(x_0)=\inf_{x\in\mathbb R^N}\psi(x)<0,$$ which along with \eqref{36} implies that $x_0\in\mathbb R^N\setminus B_{R^\prime}.$ On the other hand, we see that
$$(-\triangle)^s\psi(x_0)=C_{N,s}PV\int_{\mathbb R^N}\frac{\psi(x_0)-\psi(y)}{|x-y|^{N+2s}}dy\leq0.$$
As a result, we have
$$(-\triangle)^s \psi(x_0)+\frac{\bar{\theta}}{2}\psi(x_0)<0,$$
which contradicts \eqref{36}. Therefore $\psi\geq 0$ in $\mathbb R^N$. In other words, there holds
$$\tilde{u}_\epsilon(x)+\tilde{v}_\epsilon(x)\leq \frac{C}{1+|x|^{N+2s}}$$
for some positive constant $C$ independent of $\epsilon.$ As a consequence,
\begin{equation*}
\begin{aligned}
\hat{u}_\epsilon(x)+\hat{v}_\epsilon(x)&=u_\epsilon\left(\frac{x}{\epsilon}\right)+v_\epsilon\left(\frac{x}{\epsilon}\right)\\
&=\bar{u}_\epsilon\left(\frac{x}{\epsilon}-z_\epsilon\right)+\bar{v}_\epsilon\left(\frac{x}{\epsilon}-z_\epsilon\right)\\
&\leq \frac{C\epsilon^{N+2s}}{\epsilon^{N+2s}+|x-\epsilon z_\epsilon|^{N+2s}}\\
&\leq\frac{C\epsilon^{N+2s}}{\epsilon^{N+2s}+|x-\epsilon x_\epsilon|^{N+2s}},
\end{aligned}
\end{equation*}
where the last inequality is obtained by the fact that $\epsilon z_\epsilon\to z_0$ and $\epsilon x_\epsilon\to z_0$  as
$\epsilon\to 0.$ The proof is complete.
\end{proof}

\section*{Acknowledgements}
The first author  would like to thank the China Scholarship Council of China (201706180064) for financial support during the period of his overseas study and to express his gratitude to   the Department of Mathematical Sciences in
Yeshiva University for its kind hospitality. The first author  is also partially supported by the Fundamental Research Funds for the Central Universities (lzujbky-2017-it53) and the second author is partially supported by the National Natural Science Foundations of China (No.11471147).

\end{document}